\documentclass{amsart}
\usepackage{amsmath}
\usepackage{amssymb}
\usepackage[all]{xy}
\usepackage{graphicx}
\usepackage{mathrsfs}
\usepackage{mathabx}
\usepackage{mathrsfs}
\usepackage{bbold}

\numberwithin{equation}{section}

\usepackage{url}
\usepackage{longtable,tabu}
\usepackage{float}

\usepackage[latin1]{inputenc}
\usepackage{xspace,amssymb,amsfonts,euscript}
\usepackage{amsthm,amsmath}
\usepackage{palatino}
\usepackage{euscript}
\input xy \xyoption {all}

\usepackage{tikz,environ}
\usetikzlibrary{arrows}
\usetikzlibrary{patterns,snakes}

\RequirePackage{color}
\definecolor{myred}{rgb}{0.75,0,0}
\definecolor{mygreen}{rgb}{0,0.5,0}
\definecolor{myblue}{rgb}{0,0,0.65}

\tikzset{%
  DynNode/.style={circle, inner sep=0pt, draw=black, fill=white,minimum size=20pt},
  Greater/.style={pos=0.65, inner sep=0mm, outer sep=0mm},
  highlight/.style={rectangle,rounded corners,fill=red!15,draw=red,
    fill opacity=0.5,thick},
  bendBelow/.style={bend left=70, looseness=2},
  bendAbove/.style={bend right=70, looseness=2},
  object/.style={circle, fill, inner sep=1.5pt, outer sep=0mm},
  labeling/.style={outer sep=0mm, inner sep=0mm},
  1morph/.style={->, shorten >= 0.5pt, >=stealth'},
  2morph/.style={-implies,double,double equal sign distance,
                 shorten >=2pt, shorten <=3pt},
  spot/.style={color=black, thin, dashed},
  sline/.style={color=blue, line width=2pt},
  tline/.style={color=red, line width=2pt},
  uline/.style={color=green, line width=2pt},
  line/.style={color=#1, line width=2pt},
  line/.default=blue,
  sdot/.style={color=blue, thin, fill},
  tdot/.style={color=red, thin, fill},
  udot/.style={color=green, thin, fill},
  dot/.style={color=#1, thin, fill},
  dot/.default=blue
}

\RequirePackage{ifpdf}
\ifpdf
  \IfFileExists{pdfsync.sty}{\RequirePackage{pdfsync}}{}
  \RequirePackage[pdftex,
   colorlinks = true,
   urlcolor = myblue, % \href{...}{...} external (URL)
   citecolor = mygreen, % \cite{...}
   linkcolor = myred, % \ref{...} and \pageref{...}
   pagebackref,
   bookmarksopen=true]{hyperref}
\else
  \RequirePackage[hypertex]{hyperref}
\fi

\RequirePackage{ae, aecompl, aeguill} % to have pretty pdf

  \def\hg{{\mathfrak h}}

    \def\ZM{{\mathbb{Z}}}

    \def\FC{{\mathcal{F}}}
    
    \def\HC{{\mathcal{H}}}

\newcommand{\nc}{\newcommand} \newcommand{\renc}{\renewcommand}
\def\a{\alpha}

\renc{\l}{\lambda}

\newcommand{\rdots}{\mathinner{ \mkern1mu\raise1pt\hbox{.}
    \mkern2mu\raise4pt\hbox{.}
    \mkern2mu\raise7pt\vbox{\kern7pt\hbox{.}}\mkern1mu}}

\DeclareMathOperator{\can}{can}

\def\un{\underline}

\def\to{\rightarrow}

\def\laongto{\laongrightarrow}

\nc{\triright}{\stackrel{[1]}{\to}}
\nc{\laongtriright}{\stackrel{[1]}{\laongto}}

\nc{\Hb}{H^\bullet}

\nc{\Br}{\mathcal{B}}
\nc{\HotRR}{{}_R\mathcal{K}_R}
\nc{\HotR}{\mathcal{K}_R}
\nc{\excise}[1]{}
\nc{\defect}{\text{df}}
\nc{\h}[1]{\underline{H}_{#1}}

\nc{\Ga}{\mathbb{G}_a} % additive group
\nc{\Gm}{\mathbb{G}_m} % multiplicative group

\nc{\Perv}{{\mathbf{P}}}

\nc{\IH}{{\mathrm{IH}}}
\nc{\ic}{\mathbf{IC}}

\nc{\gl}{{\mathfrak{gl}}}
\renc{\sl}{{\mathfrak{sl}}}
\renc{\sp}{{\mathfrak{sp}}}

\renc{\Im}{\textrm{Im}}

\nc{\HBM}{H^{BM}}

\DeclareMathOperator{\Hom}{Hom}
 
\DeclareMathOperator{\End}{End}

\DeclareMathOperator{\Rep}{\mathrm{Rep}}

\DeclareMathOperator{\id}{id}

\DeclareMathOperator{\std}{std}

\newtheorem{thm}{Theorem}[section]
\newtheorem{lem}[thm]{Lemma}
\newtheorem{prop}[thm]{Proposition}

\newtheorem{conj}[thm]{Conjecture}

\theoremstyle{definition}
\newtheorem{ex}[thm]{Example}

\newtheorem{defn}[thm]{Definition}

\theoremstyle{remark}
\newtheorem{remark}[thm]{Remark}

\newcommand{\ra}{\rightarrow}

\nc{\simto}{\stackrel{\sim}{\to}}

\nc{\ext}{\textrm{ext}}

\nc{\fW}{{}^f W}
\nc{\pdot}{ \bullet_p}
\nc{\wdot}{ \bullet}
\nc{\la}{\langle}
\renc{\ra}{\rangle}

\nc{\Wf}{W}
\nc{\Wa}{\mathcal{W}}
\nc{\Sa}{\mathcal{S}}
\nc{\Wae}{\mathcal{W}^{\mathrm{ext}}}

\nc{\Fr}{\mathrm{Fr}} % Frobenius twist
\nc{\Repp}{\Rep_0}
\nc{\Repe}{\mathrm{Rep}_0^{\mathrm{ext}}}

\nc{\AntiS}{\mathrm{AS}}
\nc{\CAS}{\mathcal{AS}}

\nc{\HCe}{\HC^{\mathrm{ext}}}

\nc{\Kar}{\mathrm{Kar}}
\nc{\bs}{\mathrm{BS}}

\nc{\Iw}{\mathrm{Iw}}
\nc{\ev}{\mathrm{ev}}

\nc{\SL}{\mathrm{SL}}
\nc{\GL}{\mathrm{GL}}
\nc{\Sp}{\mathrm{Sp}}
\nc{\Eeight}{\mathrm{E}_8}
\nc{\Gtwo}{\mathrm{G}_2}

\nc{\He}{\mathrm{H}} % affine Hecke algebra
\nc{\Hee}{\mathrm{H}^\mathrm{ext}} % extended affine Hecke algebra
\nc{\Zvv}{\mathbb{Z}[v]}
\nc{\Zv}{\mathbb{Z}[v^{\pm 1}]}

\nc{\HCat}{\mathcal{H}} % affine Hecke category
\nc{\HCate}{\mathcal{H}^\mathrm{ext}} % extended affine Hecke category

\nc{\pcan}{{}^p\un{h}}

\nc{\Loc}{\mathrm{Loc}}
\nc{\mult}{\mathrm{mult}}
\nc{\geom}{\mathrm{geom}}
\nc{\diag}{\mathrm{diag}}

\nc{\Dmix}{D^{mix}}

\nc{\ses}{\mathrm{Semis}}
\nc{\Par}{\mathrm{Par}}

\nc{\IF}{IF}

\nc{\comm}[1]{{\color{red}GW: #1}}
\nc{\bencomm}[1]{{\color{blue}BE: #1}}

\makeatletter
\newcommand\leftdash{\!\rotatebox[origin=c]{-60}{$\dabar@\dabar@\dabar@$}\!}
\newcommand\rightdash{\!\rotatebox[origin=c]{60}{$\dabar@\dabar@\dabar@$}\!}
\makeatother

\newcommand{\ot}{\otimes}

\newcommand{\SBim}{\mathbb{S}\mathrm{Bim}}
\newcommand{\StdBim}{\mathrm{StdBim}}

\newcommand{\reds}{{\color{red} {s}}}
\newcommand{\bluet}{{\color{blue} {t}}}

\newcommand{\TTL}{2\mathcal{TL}}
\newcommand{\JWalt}{\mathcal{JW}}
\newcommand{\JWalttwo}{{}^2 \mathcal{JW}}
\newcommand{\base}{A'}
\DeclareMathOperator{\ptr}{pTr}
\DeclareMathOperator{\poly}{poly}

\newcommand{\Std}{Std}

\title{Localized calculus for the Hecke category}

\author{Ben Elias}
\author{Geordie Williamson}

\date{\today}

\begin{document}

\maketitle

\begin{abstract} We construct a functor from the Hecke category to a
  groupoid built from the underlying Coxeter group. This fixes a gap in an earlier work of the
  authors. This functor provides an abstract realization of the
  localization of the Hecke category at the field of
  fractions. Knowing explicit formulas for the localization is a key technical tool in software
  for computations with Soergel bimodules.
\end{abstract}

\section{Introduction}

Soergel bimodules provide a concrete incarnation of the Hecke
category associated to any Coxeter group. This category is a fundamental object in geometric representation theory. In
\cite{EW} the authors provided a diagrammatic (or ``generators and
relations'')  description
of Soergel bimodules, building on earlier work of Libedinsky
\cite{LLL, LibRA}, Elias-Khovanov \cite{EKh} and the first author
\cite{EDC}. This description has provided new impetus for the
algebraic study of the Hecke category, and has had significant
applications in modular and higher representation theory (see
e.g. \cite{RW,AMRW2,ELo,HeW,WT, BCH, EHFTDiag, GHyified}). Abe has
also introduced another more algebraic incarnation of the Hecke
category \cite{Abe}, which is equivalent to the diagrammatic category
in many situations, and has
already found important applications in representation theory \cite{Abe2,BR}.
The reader is referred to the
introduction to \cite{EW} and \cite{WParHecke} for a discussion of
the origins of the Hecke algebra and category, as well as its significance
in geometric and higher representation theory.

Soergel bimodules are a full monoidal subcategory of graded bimodules
over a polynomial ring $R$. A key tool in the theory is provided by
localization, which embeds Soergel bimodules inside a much simpler
monoidal category, a subcategory of bimodules over
the quotient field $Q$ of $R$ which we call the ``$Q$-groupoid'' associated to the underlying
Coxeter group. (In geometric settings, this
  technique is related to the localization theorem in equivariant
  cohomology.) After base change from $R$ to $Q$, any Soergel bimodule splits a direct sum of (grading shifts of) objects in the $Q$-groupoid.

In \cite{EW} this idea was imported into the diagrammatic study, and provided a key tool in the proofs. First, in \cite{EWdiag} we provided a diagrammatic
desription of this $Q$-groupoid by generators and relations, using the topology of the Coxeter complex to prove the correctness of our presentation. In \cite{EW} we provided a diagrammatic description of a ``mixed calculus''
which combined the diagrammatics of the Hecke category and the $Q$-groupoid. This mixed category is tailor-made to be simultaneously equivalent to both \begin{itemize} \item the
Karoubi envelope of the diagrammatic Hecke category, after base change from $R$ to $Q$, and \item the additive envelope of the diagrammatic $Q$-groupoid. \end{itemize} Combining these equivalences, we
obtain an implicit construction of a functor, which we denote $\Lambda$ in this paper, from the diagrammatic Hecke category to the additive envelope of the $Q$-groupoid. The main
goal of this paper is to make this ``localization functor'' $\Lambda$ explicit.

As kindly pointed out to us to Simon Riche, the treatment of localization in \cite{EW} contains a gap as written. We constructed a full and essentially surjective functor, denoted
$\Std$ in \cite{EW}, from the additive closure of the $Q$-groupoid to our mixed calculus, but the faithfulness of this functor was not properly justified. The faithfulness of
$\Std$ is what demonstrates that the mixed calculus, and hence the diagrammatic Hecke category, is nonzero. We gave two separate arguments in \cite{EW} that the functor $\Std$ is
faithful. The first (see the end of \cite[\S 5.4]{EW}) relies on the existence of a functor $\FC$ from the diagrammatic Hecke category to Soergel bimodules. We made a
separate error here (in \cite[Claim 5.13]{EW}), claiming to construct the functor $\FC$ in more generality than we actually did. Regardless, our intention was to study the
diagrammatic Hecke category in additional generality, where such a functor $\FC$ might not exist. Consequently we wrote \cite[\S 5.5, in particular Proposition 5.25]{EW}, where we
sketched a method for constructing the localization functor $\Lambda$ explicitly. In turn, the localization functor would induce a functor from the mixed calculus which would be
quasi-inverse to $\Std$, thus proving its faithfulness. Computer calculations of the second author verified this method for (most realizations of) those Coxeter groups which satisfy
$m_{st} \le 4$ for all simple reflections $s, t$. However, it was a major oversight to think that the general case would be trivial! This entire paper is a construction of
$\Lambda$, and thereby completes the sketch of \cite[Proposition 5.25]{EW}.

\begin{remark} Between the argument using $\FC$ and the computer calculations made, the proofs in \cite{EW} are sufficient to deal with many versions of the Hecke category (e.g.
coming from different realizations of the Coxeter group), and in particular, are sufficient for all of the applications to date. However, our intention in \cite{EW} was
to work in the broadest reasonable generality, and this paper allows us to. For more information on the errors in \cite{EW}, and a summary of the assumptions now required to quote
various theorems, see \S\ref{sec:error}. \end{remark}

\begin{remark} The target of the localization functor of \cite{EW} was to be the additive envelope of the diagrammatic version of the $Q$-groupoid, as constructed in \cite{EWdiag}
using the topology of the Coxeter complex. In this paper, the target of our localization functor $\Lambda$ is the algebraic $Q$-groupoid, defined below in \S\ref{subsec:Qgroupoid}.
Originally, it was the topological proof of \cite{EWdiag} that formed the bedrock upon which the well-definedness of the Hecke category was established. This paper removes the
dependency on topology. \end{remark}

\begin{remark}   Throughout this paper, we assume our Coxeter system has no parabolic
  subgroup of type $H_3$. The diagrammatic Hecke category is not
  defined in type $H_3$, because one crucial relation (the
  ``Zamolodchikov relation'') is missing. Further discussion on type $H_3$ can be found in \S\ref{subsec:H3}. \end{remark}

A secondary purpose of this paper is the following: when writing \cite{EW} the
authors viewed localization primarily as a theoretical,
rather than computational, tool. However subsequently there have
been several instances where it has been necessary to understand the
localization of the diagrammatic category in concrete terms. An
important example is software developed by the second author and Thorge Jensen for
efficient computation of the $p$-canonical basis of Hecke algebras (see
\cite{JensenW,GJW} for details, and \cite{LWbilliards} for an application of
this technology). A second instance is recent work of Libedinsky and
the second author \cite{LW} which uses localization to prove the positivity of
anti-spherical Kazhdan-Lusztig polynomials.

Initially, ad hoc formulas were written down for the
localization (and these were used to do various computer checks in
\cite{EW}). Subsequently, much cleaner and general formulations were found (see
\eqref{subeq:onecolor}, \eqref{fdef}, Lemma \ref{lem:E}). In this paper we provide
these formulas.

A third goal of this paper is to improve the state of the literature relating to (two-colored) Jones-Wenzl projectors, certain idempotents in the (two-colored) Temperley-Lieb
category. The two-colored Jones-Wenzl projectors were studied heavily in \cite{EDC}, where they were used to construct idempotents projecting to indecomposable Soergel bimodules
for the dihedral groups\footnote{The relationship between Soergel bimodules for dihedral groups and the two-colored Temperley-Lieb 2-category is an instance of ``quantum geometric
Satake at a root of unity'' in type $A_1$, see \cite{EQAGS}.}. In particular, for the diagrammatic Hecke category to be well-defined, it is extremely important whether or not a
given two-colored Jones-Wenzl projector is \emph{rotatable}, i.e. it is rotation-invariant up to scalar. Negligibility of the Jones-Wenzl projector is a necessary but not
a sufficient condition. Another important property (that we need for our construction of $\Lambda$) is that the ``polynomial evaluation'' of a particular Jones-Wenzl projector is
the product of the positive roots for the corresponding dihedral group. The discussion of these properties in \cite{EDC} was not done in sufficient generality for our needs in this
paper, and new ideas were needed to generalize them. Our results here
are new and should have independent interest even in the study of
ordinary Jones-Wenzl projectors. We
postpone further discussion of these aspects to the end of the introduction.

We now provide a more technical introduction to the results of this paper.

\subsection{The $Q$-groupoid} \label{subsec:Qgroupoid}

Let $W$ be a group acting by homomorphisms on a commutative ring
$Q$. From this data one can construct a monoidal category $\Omega_Q W$
(the ``$Q$-groupoid of $W$''), which provides a categorification of $W$. Its
definition is as follows. It has objects $\{ r_x \; | \; x \in W \}$ with
\begin{equation} \label{QSchur}
\Hom(r_x, r_y) = \begin{cases} Q & \text{if $x = y$,} \\ 0 &
  \text{otherwise}. \end{cases}
\end{equation}
The monoidal structure on objects is given by $r_xr_y := r_{xy}$. The tensor product
of morphisms is as follows: if $f, g \in Q$ are regarded as
elements of $\End(r_x)$ and $\End(r_y)$ respectively, then 
\begin{equation} \label{tensorpolys} f \otimes g := f x(g)\end{equation} as an element of $Q = \End(r_{xy})$.

If the $W$-action on $Q$ is faithful, $\Omega_Q W$ may be understood concretely as follows. For $x \in W$ one can define a $Q$-bimodule $Q_x$, which is isomorphic to $Q$ as a left
$Q$-module, and has right $Q$-action twisted by $x$: $q \cdot q' := x(q') \cdot q$ for $q \in Q_x$ and $q' \in Q$. Such bimodules are called \emph{standard bimodules}. There exists
a canonical isomorphism $Q_x \otimes_Q Q_y = Q_{xy}$, so standard bimodules form a monoidal category inside $Q$-bimodules. It is easy to construct a functor $\FC_{\std}$ from
$\Omega_Q W$ to the category of standard bimodules, sending $r_x$ to $Q_x$. When the $W$-action on $Q$ is faithful this functor is an equivalence.

However, when the $W$-action on $Q$ is not faithful, there are isomorphisms $Q_x \cong Q_y$ for $x \ne y$; such an isomorphism does not exist in the abstractly-defined category
$\Omega_Q W$. In this case $\FC_{\std}$ is faithful but not full; it induces an isomorphism from all \emph{nonzero} morphism spaces in $\Omega_Q W$ to the
corresponding morphism spaces between bimodules, but entirely misses some other morphism spaces between bimodules. One can think of $\Omega_Q W$ as a model for the category of
standard bimodules, but a model which does not detect degenerate behavior. The isomorphism classes in $\Omega_Q W$ are always in bijection with $W$, regardless of any surprises
that the bimodule category may have in store.

\begin{remark} \label{rmk:stdmorphismsarestd} We will view morphism spaces in $\Omega_Q W$ as being $Q$-bimodules, just as morphism spaces between standard bimodules are. The
identification of $\End(r_x)$ with $Q$ intertwines the left action of $Q$ on morphism spaces. Meanwhile, the right action of $Q$ on $\End(r_x)$ is twisted by $x$, so that
$\End(r_x)$ is actually isomorphic to $Q_x$ as a $Q$-bimodule, just as $\End(Q_x)$ is. With this convention, the functor $\FC_{\std}$ preserves the $Q$-bimodule structure on morphism spaces. \end{remark}

\subsection{The Hecke category and our main theorem}

Now fix a commutative domain $\Bbbk$, and choose a \emph{realization} of $(W,S)$ over $\Bbbk$, see \cite[Definition 3.1]{EW}. Roughly speaking, a realization is the data of a
representation of $W$ on a free $\Bbbk$-module $\hg^*$ for which $W$ acts by ``reflections,'' together with data on the reflecting hyperplanes (i.e. roots and coroots). Associated
to this realization one has a polynomial ring $R$, whose linear terms are given by $\hg^*$. Let $Q$ denote the fraction field of $R$. By slight abuse of language, we still refer to
elements of $Q$ as polynomials. Note that $W$ acts on $R$ and $Q$ by algebra homomorphisms. In a general realization the action of $W$ on $\hg^*$ (and hence on $R$ and $Q$) need
not be faithful.

Let $\HC$ denote the (diagrammatic) Hecke category associated to this realization, as defined in \cite{EW}. Morphism spaces in $\HC$ are $R$-bimodules. The Hecke category
$\HC$ is meant to encode the morphisms between certain $R$-bimodules known as Bott-Samelson bimodules (whose summands are called Soergel bimodules): for certain realizations there
is a functor $\FC$ from $\HC$ to the category of Bott-Samelson bimodules. Just like $\Omega_Q W$, $\HC$ is a model which does not detect any degenerate behavior which may appear in
the bimodule setting for special realizations.

\begin{remark} In order to define the diagrammatic Hecke category associated to a realization, one requires the existence and rotatability of certain two-colored
Jones-Wenzl projectors. This requirement is a (somewhat mysterious) condition on the Cartan matrix of the realization (i.e. the pairing of roots and coroots), and on
the base ring $\Bbbk$. We assume that these projectors exist until
\S\ref{sec:error}, and discuss the conditions for them to exist in
more detail in \S\ref{sec:JW}. \end{remark}

Let $B$ be a Bott-Samelson bimodule. After base change from $R$ to $Q$ on one side, it becomes possible to equip $B \ot_R Q$ with a left action of $Q$, making it into a
$Q$-bimodule. Then one can prove that $B \ot_R Q$ splits into a direct sum of standard bimodules $Q_x$. In fact, this splitting lifts to the diagrammatic models as well. Let $(\Omega_Q W)_{\oplus}$ denote the additive envelope of $\Omega_Q W$. The main result of this paper is the following theorem:

\begin{thm} \label{thm:main} There exists a monoidal functor $\Lambda : \HC \to (\Omega_Q W)_{\oplus}$. \end{thm}

\begin{remark} In the literature, the symbol $\HC_{BS}$ is often used to refer to the pre-additive category defined by generators and relations in \cite{EW}, while the symbol $\HC$ refers to the Karoubi envelope of its additive, graded closure. We only care about the category $\HC_{BS}$ in this paper, so we refer to it as $\HC$ for simplicity. Since $(\Omega_Q W)_{\oplus}$ is Karoubian, the functor $\Lambda$ automatically lifts from $\HC$ to its Karoubi envelope. \end{remark}

The functor $\Lambda \colon \HC \to (\Omega_Q W)_{\oplus}$ is meant to describe what happens to a morphism between Bott-Samelson bimodules after passage to the localization, where
this morphism becomes a matrix of morphisms between the standard summands. For some realizations, such as when the $W$-action is not faithful, neither the functor $\FC$ from $\HC$
to the category of Bott-Samelson bimodules, nor the functor $\FC_{\std}$ from $\Omega_Q W$ to the category of standard bimodules, is an equivalence. Nonetheless, one can show
abstractly that after base changing the morphisms from $R$ to $Q$ on one side, the objects in $\HC$ split as direct sums of objects in $\Omega_Q W$. The functor $\Lambda$ still
records what happens to morphisms after localization.

\begin{remark} Recently, Abe has proved that the
  formulas here define morphisms in his bimodule
  theoretic incarnation of the Hecke category, under a natural
  assumption on the realization \cite{Abe3}. (The key difficulty is to show that a
  morphism between localizations of bimodules is in fact the
  localization of a morphism.)
  This constructs a version of the
  functor $\FC$ from the Hecke category to Soergel bimodules in broad
  generality. We discuss the Abe's results in more detail in \S\ref{sec:existence}. 
\end{remark}

Morphisms between objects in $\Omega_Q W$ are relatively easy to understand, thanks to \eqref{QSchur}: they are polynomials in $Q$, or are zero. Morphisms in $(\Omega_Q
W)_{\oplus}$ between direct sums of objects in $\Omega_Q W$ are also easy to understand: they are matrices of morphisms in $\Omega_Q W$, and hence matrices whose entries are either
polynomials or zero. This is unlike morphisms in $\HC$, which are relatively complicated and difficult to compose.

One use of the functor $\Lambda$ is to prove that the double leaves basis (see \cite[\S 6]{EW}) of $\HC$ is, in fact, a basis. One proves that these morphisms span using a
complicated diagrammatic argument. To prove linear independence, one applies the functor $\Lambda$, and proves that their image is still linearly independent using an upper
triangularity argument. Consequently, double leaves are a basis, and the functor $\Lambda$ is therefore faithful. For this reason, the well-definedness of $\Lambda$ was crucial to
the results of \cite{EW}. See \cite[Propositions 6.6 and 6.9]{EW} for more details.

Another use of the functor $\Lambda$ is to enable computer
calculations. With current software, computers are not good at
diagrammatic algebra. On the other hand, computers are excellent at multiplying sparse matrices of polynomials. Since one can check that two morphisms are equal after applying $\Lambda$, this transforms a
complicated diagrammatic problem into matrix multiplication. To program this, one must know the matrix of polynomials associated to any generating morphism of $\HC$. For each pair
of simple reflections forming a finite dihedral subgroup of order $2m$, there is a generating morphism called the $2m$-valent vertex. For small values of $m$, the matrix associated
to the $2m$-valent vertex was already in use by the authors for their computer calculations, but was not in the literature. Here, we present a precise formula for all $m$.

\subsection{Organization of the paper}

In Chapter \S\ref{sec:defnoffunctor} we define the functor $\Lambda$ by defining it on generators. In \S\ref{sec:checkrelations} we check that the functor is well-defined by
checking the relations. Thus we prove Theorem \ref{thm:main}.

In \S\ref{sec:defnoffunctor} and \S\ref{sec:checkrelations} we provide and manipulate a whole lot of matrices with values in $Q$ (which describe the functor $\Lambda$ on
morphisms), without explaining where these matrices come from. The goal of chapter \S\ref{sec:heuristics} is to explain how we computed these matrices. Much of the authors'
intuition for the functor $\Lambda$ is gained from a diagrammatic presentation of $\Omega_Q W$, which we developed in \cite[\S 4]{EW}. This led to a combined diagrammatic calculus
for localized Soergel bimodules, found in \cite[\S 6]{EWdiag}, in which one can describe the heuristic method for computing $\Lambda$. We recall and use this technology in
\S\ref{sec:heuristics}, and we encourage the reader to briefly acquaint themselves with \cite[\S 4]{EW} and \cite[\S 6]{EWdiag}. Because these heuristics are not needed to follow
the rest of the paper, and because we wished to emphasize the logical independence of Theorem \ref{thm:main} from this additional diagrammatic technology, we chose to put the cart
in front of the horse and prove the theorem before motivating the construction.

In \S\ref{sec:error} we describe and rectify the error from our previous work, and elaborate on which assumptions one needs to use various results and constructions.

Our proofs in \S\ref{sec:checkrelations} rely on several technical properties of (two-colored) Jones-Wenzl projectors. These properties all appear in our
earlier works, but in varying degrees of explicitness and with varying levels of proof. Jones-Wenzl projectors are traditionally studied using recursive formulas, but these
recursive formulas might not apply in general (e.g. when we need to study the $n$-th Jones-Wenzl projector, but the $k$-th Jones-Wenzl projector is not defined for $k < n$). Some
of the proofs\footnote{In particular, see Lemma \ref{lem:evaluation} for a result where the proof in \cite{EDC} is insufficiently general.} in \cite{EDC} tacitly relied on these
recursive formulas, and so they do not apply in all cases (e.g. for non-faithful realizations of dihedral groups). Altogether, we felt it was necessary to provide a complete
discussion of Jones-Wenzl projectors and to carefully prove the required properties even when recursive formulas do not apply, which we do in \S\ref{sec:JW}. We give a general
argument that, whenever a Jones-Wenzl projector exists over a (sufficiently nice) base ring, then its coefficients ``agree'' with the coefficients over a generic field. Over the
generic field we can use the standard recursive formulas. This
argument is new and general, and may be of independent interest. In
\S\ref{subsec:specialcases} we describe Jones-Wenzl projectors in some
special cases of note, including types $B_2$ and $G_2$ in small
characteristic.

The diagrammatic calculus for $\HC$ has some additional complications
when the realization is \emph{odd-unbalanced}. We also describe the
localization functor $\Lambda$ in this 
setting. The additional technicalities are distracting, so we postpone
this case until the final chapter. Again, we also feel like the
accounts given of the unbalanced case in 
previous work are hasty and not easy for the reader to piece
together. Thus we have aimed to present a self-contained exposition in
\S \ref{sec-unbalanced}. 

\subsection*{Acknowledgements} We would like to thank Simon Riche for
initially pointing out the gap in \cite{EW}, which gave rise to this
paper, as well as several very interesting questions and observations
(particularly concerning Jones-Wenzl projectors in small
characteristic) which gave rise to \S\ref{subsec:specialcases}. We
would also like to thank the referee for a detailed reading, and
suggesting changes which have clarified the exposition. The first author was supported by NSF grants DMS-1553032 and DMS-1800498 and DMS-2201387. Part of this work was written while the authors were visiting the IAS, a visit supported by NSF grant DMS-1926686. Both authors wish to thank the IAS for their hospitality.

\section{Definition of the localization functor} \label{sec:defnoffunctor}

Both $\HC$ and $(\Omega_Q W)_{\oplus}$ are monoidal and $\HC$ is presented by generators and relations. The strategy to construct the functor $\Lambda \colon \HC \to (\Omega_Q
W)_{\oplus}$ is simple: we define the images of the generating objects and morphisms, and then check the relations. In this section we define the functor on objects and morphisms,
and in the subsequent sections we check the relations. We assume throughout that our realization is such that $\HC$ is well-defined, and elaborate further on what this entails when it becomes relevant. Starting in \S\ref{subsec:defntwocolor} we assume our realization is balanced for simplicity; the unbalanced case is treated in \S\ref{sec-unbalanced}.

\subsection{Definition on objects} \label{subsec:defnobjects}

The Hecke category $\HC$ is generated as a monoidal category by the
set $\{ B_s \; | \; s \in S\}$. 
On objects we set
\begin{equation}
  \label{eq:objects}
\Lambda:   B_s \mapsto r_{\id} \oplus r_s 
\end{equation}
If we think of $\id = s^0$ and $s = s^1$ inside $W$, then 
\begin{equation}
\Lambda(B_s) = \bigoplus_{e \in \{0,1\}} r_{s^e}.
\end{equation}
By definition, we have canonical inclusions and
  projections from $\Lambda(B_s)$ to its direct summands $r_{\id}$ and
  $r_s$.

A sequence $\un{w} = (s_1, \ldots, s_d)$ of simple reflections ($s_i
\in S$) is called an \emph{expression}, and it has \emph{length}
$d$. We remove the underline and write $w$ for the corresponding
element in $W$, i.e. $w  = s_1 \cdots s_d$. In $\Omega_Q W$ we have 
\begin{equation} r_w = r_{s_1} \otimes r_{s_2} \otimes \cdots \otimes r_{s_d}. \end{equation}
If $\un{w} = (s_1, \ldots, s_d)$ is an expression, a
\emph{subexpression} of $\un{w}$ is a sequence $\un{e} = (e_1, \ldots,
e_d)$ of zeroes and ones of the same length as $\un{w}$. We write
$\un{e} \subset \un{w}$. This subexpression is thought of as
representing the subword $s_1^{e_1} s_2^{e_2} \cdots s_d^{e_d}$ of
$\un{w}$, and we call the corresponding element of $W$ the
\emph{endpoint} of the subexpression. We employ the notation
\begin{equation} \un{w}^{\un{e}} := s_1^{e_1} \dots s_d^{e_d} \in W \end{equation}
and we have
\begin{equation} r_{\un{w}^{\un{e}}} = r_{s_1^{e_1}} \otimes \cdots \otimes r_{s_d^{e_d}}. \end{equation}

The objects of $\HC$ are tensor products of various objects $B_s$, and
thus can be identified with expressions. For an expression $\un{w} =
(s_1, \dots, s_d)$ we write $B_{\un{w}} := B_{s_1} \otimes \cdots
\otimes B_{s_d}$.  Because $\Lambda$ is monoidal, it follows that we have a canonical isomorphism
\begin{equation}
  \label{eq:can_decomp}
\Lambda( B_{\un{w}}) = (r_{\id} \oplus r_{s_1} ) \otimes \dots \otimes
(r_{\id} \oplus r_{s_d}) = \bigoplus_{\un{e} \subset \un{w}} r_{\un{w}^{\un{e}}}.
\end{equation}
In this way, subexpressions of $\un{w}$ identify the $2^d$ summands inside $\Lambda(B_{\un{w}})$. Given a subexpression $\un{e} \subset \un{w}$ with $\un{w}^{\un{e}} =
x \in W$ we denote by
\[
i(\un{e}) : r_x \to \Lambda( B_{\un{w}}) \quad \text{and} \quad
p(\un{e}) : \Lambda( B_{\un{w}}) \to r_x 
\]
the inclusion of and projection to the canonical summand of
\eqref{eq:can_decomp} indexed by $\un{e}$.

\subsection{Definition on one-color morphisms} 
Morphisms in additive categories can be described as matrices of
morphisms with respect to a fixed direct sum decomposition. Thus, we can specify a morphism
\[
  \psi \colon \Lambda(B_{\un{w}}) \to \Lambda(B_{\un{x}})
\]
  by encoding it as a matrix with elements in $Q$, whose columns are indexed by subexpressions
$\un{e} \subset \un{w}$ and whose rows are indexed by subexpressions $\un{f} \subset \un{x}$. The entry in column $\un{e}$ and row $\un{f}$ is the composition $p(\un{f})
\circ \psi \circ i(\un{e})$, which is a map $r_{\un{w}^{\un{e}}} \to r_{\un{x}^{\un{f}}}$, viewed as an element of $Q$.

The images of the one-color morphisms are as follows:
\begin{subequations} \label{subeq:onecolor}
\begin{gather}
  \label{eq:f}
\begin{array}{c}
\tikz[scale=0.35]{
%\draw[dashed] (3,0) circle (1cm);
\draw[dashed] (2,-1) to (4,-1); \draw[dashed] (2,1) to (4,1);
\node at (3,0) {$f$};}
\end{array}
\mapsto \begin{array}{c|c} & \emptyset \\ \hline\emptyset & f \end{array} \\ 
  \label{eq:idot}
\begin{array}{c}
\tikz[scale=0.35]{
%\draw[dashed] (3,0) circle (1cm);
\draw[dashed] (2,-1) to (4,-1); \draw[dashed] (2,1) to (4,1);
\draw[color=red] (3,-1) to (3,0);
\node[circle,fill,draw,inner sep=0mm,minimum size=1mm,color=red] at (3,0) {};}
\end{array}
\mapsto \begin{array}{c|cc}
& 0 & 1 \\
 \hline
\emptyset & {\color{red} \alpha} & 0 
\end{array} \\ 
  \label{eq:ddot}
\begin{array}{c}
\tikz[xscale=0.35,yscale=-0.35]{
%\draw[dashed] (3,0) circle (1cm);
\draw[dashed] (2,-1) to (4,-1); \draw[dashed] (2,1) to (4,1);
\draw[color=red] (3,-1) to (3,0);
\node[circle,fill,draw,inner sep=0mm,minimum size=1mm,color=red] at (3,0) {};}
\end{array} \mapsto 
\begin{array}{c|c}
& \emptyset \\
 \hline
0 & 1 \\
1 & 0 \end{array} \\
  \label{eq:peace}
\begin{array}{c}
\tikz[xscale=0.35,yscale=-0.35]{
%\draw[dashed] (3,0) circle (1cm);
\draw[dashed] (2,-1) to (4,-1); \draw[dashed] (2,1) to (4,1);
\draw[color=red] (2.3,-1) to[out=90,in=-150] (3,0);
\draw[color=red] (3.7,-1) to[out=90,in=-30] (3,0);
\draw[color=red] (3,0) to (3,1);
}
\end{array} 
\mapsto \begin{array}{c|cc}
& 0 & 1\\
\hline 00 & 1/{\color{red}\alpha} & 0 \\
01 & 0 & 1/{ \color{red}\alpha} \\
10 & 0 & -1/{ \color{red}\alpha} \\
11 & -1/{ \color{red}\alpha}  & 0 
\end{array}
\\
  \label{eq:mercedes}
\begin{array}{c}
\tikz[xscale=0.35,yscale=0.35]{
%\draw[dashed] (3,0) circle (1cm);
\draw[dashed] (2,-1) to (4,-1); \draw[dashed] (2,1) to (4,1);
\draw[color=red] (2.3,-1) to[out=90,in=-150] (3,0);
\draw[color=red] (3.7,-1) to[out=90,in=-30] (3,0);
\draw[color=red] (3,0) to (3,1);
}
\end{array} 
\mapsto \begin{array}{c|cccc}
  & 00 & 01 & 10 & 11 \\
 \hline 0 & 1 & 0 & 0 & 1 \\
1 & 0 & 1 & 1 & 0 
\end{array}
\end{gather}
\end{subequations}
If {\color{red} red} represents the simple reflection $s \in S$, then
{\color{red} $\alpha$} represents the corresponding simple root
$\alpha_s$.

\begin{remark} As a reminder of the diagrammatic conventions, morphisms are composed from bottom to top, and their corresponding matrices are composed right to left (in the usual
way). \end{remark}

\begin{remark}
Note that the degrees of the matrix entries can be
  determined as follows. The degree of any given matrix coefficient in
  $\Lambda(\phi)$ is the degree of $\phi$, plus the length of the
  source expression, minus the length of the target expression. Recall that
  $\deg({\color{red} \alpha}) = 2$.
  \end{remark}

Recall that in $\HC$, the cups and caps are obtained as compositions
of trivalent vertices and dots:
\begin{equation} \label{eq:cupcap}
\begin{array}{c}
\tikz[xscale=0.35,yscale=-0.35]{
%\draw[dashed] (3,0) circle (1cm);
\draw[dashed] (2,-1) to (4,-1); \draw[dashed] (2,1) to (4,1);
\draw[color=red] (2.3,-1) to[out=90,in=180] (3,.2);
\draw[color=red] (3.7,-1) to[out=90,in=0] (3,.2);
}
\end{array} :=
\begin{array}{c}
\tikz[xscale=0.35,yscale=-0.35]{
%\draw[dashed] (3,0) circle (1cm);
\draw[dashed] (2,-1) to (4,-1); \draw[dashed] (2,1) to (4,1);
\draw[color=red] (2.3,-1) to[out=90,in=-150] (3,0);
\draw[color=red] (3.7,-1) to[out=90,in=-30] (3,0);
\draw[color=red] (3,0) to (3,.5);
\node[circle,fill,draw,inner sep=0mm,minimum size=1mm,color=red] at (3,.5) {};
}
\end{array} 
\quad
\begin{array}{c}
\tikz[xscale=0.35,yscale=0.35]{
%\draw[dashed] (3,0) circle (1cm);
\draw[dashed] (2,-1) to (4,-1); \draw[dashed] (2,1) to (4,1);
\draw[color=red] (2.3,-1) to[out=90,in=180] (3,.2);
\draw[color=red] (3.7,-1) to[out=90,in=0] (3,.2);
}
\end{array} 
:=
\begin{array}{c}
\tikz[xscale=0.35,yscale=0.35]{
%\draw[dashed] (3,0) circle (1cm);
\draw[dashed] (2,-1) to (4,-1); \draw[dashed] (2,1) to (4,1);
\draw[color=red] (2.3,-1) to[out=90,in=-150] (3,0);
\draw[color=red] (3.7,-1) to[out=90,in=-30] (3,0);
  \draw[color=red] (3,0) to (3,.5);
  \node[circle,fill,draw,inner sep=0mm,minimum size=1mm,color=red] at (3,.5) {};
}
\end{array} 
  \end{equation}
Composing the corresponding matrices above gives the images of the
cups and caps under $\Lambda$:
\begin{gather}
  \label{eq:cup}
\begin{array}{c}
\tikz[xscale=0.35,yscale=-0.35]{
%\draw[dashed] (3,0) circle (1cm);
\draw[dashed] (2,-1) to (4,-1); \draw[dashed] (2,1) to (4,1);
\draw[color=red] (2.3,-1) to[out=90,in=180] (3,.2);
\draw[color=red] (3.7,-1) to[out=90,in=0] (3,.2);
}
\end{array} \mapsto
\begin{array}{c|c}
& \emptyset \\
\hline 00 & 1/{\color{red}\alpha} \\
01 & 0 \\
10 & 0 \\
11 & -1/ {\color{red}\alpha} 
\end{array}\\
  \label{eq:cap}
\begin{array}{c}
\tikz[xscale=0.35,yscale=0.35]{
%\draw[dashed] (3,0) circle (1cm);
\draw[dashed] (2,-1) to (4,-1); \draw[dashed] (2,1) to (4,1);
\draw[color=red] (2.3,-1) to[out=90,in=180] (3,.2);
\draw[color=red] (3.7,-1) to[out=90,in=0] (3,.2);
}
\end{array} 
\mapsto 
\begin{array}{c|cccc}
& 00 & 01 & 10 &  11\\
\hline \emptyset & {\color{red}\alpha} & 0 & 0 & {\color{red}\alpha} 
\end{array}
\end{gather}

Note that all the zero entries in the matrices above are forced to be zero, simply because $\Hom(r_x, r_y) = 0$ if $x \ne y$. When it saves space, we ignore the rows and columns in our matrices which are forced to be zero, and only write the interesting part. For example, instead of \eqref{eq:idot}, we could have simply written
\begin{equation}
 \label{eq:idot2}
\begin{array}{c}
\tikz[scale=0.35]{
%\draw[dashed] (3,0) circle (1cm);
\draw[dashed] (2,-1) to (4,-1); \draw[dashed] (2,1) to (4,1);
\draw[color=red] (3,-1) to (3,0);
\node[circle,fill,draw,inner sep=0mm,minimum size=.5mm,color=red] at (3,0) {};}
\end{array}
\mapsto \begin{array}{c|c}
& 0  \\
 \hline
\emptyset & {\color{red} \alpha} 
\end{array}.
\end{equation}

\subsection{Monoidal structure} For the sake of explicitness, we 
give some more detail and examples concerning the monoidal structure. Recall from Remark
\ref{rmk:stdmorphismsarestd} that the identification of $\End(r_x)$
with $Q$ uses the
left action of $Q$ on morphism spaces. Tensoring with (the image
  under $\Lambda$ of) $B_s$ on the
right does not affect this left action, but tensoring with (the image
  under $\Lambda$ of) $B_s$ on the left does.

\begin{remark}
In the diagrammatic language (see \S\ref{sec:heuristics}), one can
think that the matrices are recording polynomials which should be placed on the left of a diagram. When morphisms are concatenated horizontally, the left of a subdiagram is no
longer the left of the entire diagram, and the polynomial must be pushed to the left using the monoidal structure in $\Omega_Q W$, which was defined in \eqref{tensorpolys}.  
\end{remark}

\begin{ex}
  Multiplication by the polynomial $f$ is an endomorphism of the
  monoidal identity, and has the $1 \times 1$ matrix (see \eqref{eq:f})
\begin{equation}
  \begin{array}{c}
\tikz[scale=0.35]{
%\draw[dashed] (3,0) circle (1cm);
\draw[dashed] (2,-1) to (4,-1); \draw[dashed] (2,1) to (4,1);
\node at (3,0) {$f$};}
\end{array}
%\BE{geordie, please draw}
\mapsto \begin{array}{c|c}
& \emptyset  \\
 \hline
\emptyset & f 
\end{array}.
\end{equation}
Since $B_s \cong r_{\id} \oplus r_s$, we have
\begin{equation} \label{eq:polyactsilly}
  f \otimes \id_{B_s} = %\BE{please draw}
\begin{array}{c}
\tikz[scale=0.35]{
%\draw[dashed] (3,0) circle (1cm);
\draw[dashed] (2,-1) to (5,-1); \draw[dashed] (2,1) to (5,1);
  \draw[color=red] (4,-1) to (4,1);
  \node at (3,0) {$f$};}
%\node[circle,fill,draw,inner sep=0mm,minimum size=1mm,color=red] at (3,0) {};}
\end{array}
  \mapsto \begin{array}{c|cc}
  & 0 & 1 \\
 \hline 0 & f & 0 \\
1 & 0 & f  
\end{array}
\end{equation}
and
\begin{equation}
  \label{eq:polyact}
  \id_{B_s} \otimes f = %\BE{please draw}
\begin{array}{c}
\tikz[scale=0.35]{
%\draw[dashed] (3,0) circle (1cm);
\draw[dashed] (2,-1) to (5,-1); \draw[dashed] (2,1) to (5,1);
  \draw[color=red] (3,-1) to (3,1);
  \node at (4,0) {$f$};}
%\node[circle,fill,draw,inner sep=0mm,minimum size=1mm,color=red] at (3,0) {};}
\end{array}
  \mapsto \begin{array}{c|cc}
& 0 & 1 \\
\hline 0 & f & 0  \\ 1 & 0 & s(f) \end{array}.
\end{equation}
On the second row of the matrix of $\id_{B_s} \otimes f$, the
polynomial $f$ must be pushed left through $r_s$, so it is acted on by
$s$.
\end{ex}

Similarly, if we know the $m \times n$ matrix of $\Lambda(\phi)$ for some morphism $\phi$, then the matrix of $\Lambda(\phi \otimes \id_{B_s})$ (resp. $\Lambda(\id_{B_s} \otimes f)$) is a $2m \times 2n$ matrix, where each entry $f \in Q$ of the original matrix is replaced by a $2 \times 2$ matrix (indexed by subexpressions of $(s)$) according to
the matrix of \eqref{eq:polyactsilly} (resp. \eqref{eq:polyact}).

\begin{ex}
  The reader may use these rules to check that
\begin{equation}
\begin{array}{c}
\tikz[scale=0.45]{
%\draw[dashed] (3,0) circle (1cm);
\draw[dashed] (2,-1) to (6,-1); \draw[dashed] (2,1) to (6,1);
  \draw[color=red] (3,-1) to (3,1);
  \draw[color=red] (4,1) to[out=-90,in=180] (4.5,0) to[out=0,in=-90] (5,1);
%\node[circle,fill,draw,inner sep=0mm,minimum size=1mm,color=red] at
%(3,0) {};
  }
\end{array} \mapsto \begin{array}{c|cc}
& 0 & 1 \\ \hline
000 & 1/{\color{red} \alpha} & 0\\
011 & -1/{\color{red} \alpha} & 0\\
100 & 0 & s(1/{\color{red} \alpha})\\
111 & 0 & s(-1/{\color{red} \alpha}) \end{array} =
\begin{array}{c|cc}
& 0 & 1 \\ \hline
000 & 1/{\color{red} \alpha} & 0\\
011 & -1/{\color{red} \alpha} & 0\\
100 & 0 & -1/{\color{red} \alpha}\\
111 & 0 & 1/{\color{red} \alpha} \end{array}.
\end{equation}
We have omitted four rows (indexed by the subexpressions $001$,
$010$, $101$ and $110$) which are zero.
\end{ex}

\begin{remark} 
  \label{rmk:nicerpeacecup} 
Sometimes the decision to push all
  polynomials to the left of the diagram can obscure a nicer
  formula. For example, in a product $r_xr_y$, the fraction field $Q$
  acts on the left and on the right, but also in the middle, between
  the two tensor factors. The formulas for \eqref{eq:peace} and
  \eqref{eq:cup} can both be rewritten so that the polynomial
  $\frac{1}{{\color{red} \a}}$, with no sign, always appears in the middle of the target
  $\Lambda(B_s \ot B_s)$. When $\frac{1}{{\color{red} \alpha}}$ is
  pushed left, it picks up a sign based on whether the first $B_s$ is
  replaced with $r_{\id}$ or with $r_s$. \end{remark}

\subsection{Notation for positive roots} \label{subsec-pos-roots}

Our next goal is to provide the image of the $2m_{st}$-valent
vertex under $\Lambda$. This will require a little more notation first.

Let $\un{w} = (s_1, \ldots, s_d)$ be an expression of length $d \ge 1$. We define the corresponding element $\beta_{\un{w}} \in \hg^*$ by the formula
\begin{equation} \label{defn:beta} \beta_{\un{w}} = s_1 s_2 \cdots s_{d-1} (\alpha_{s_d}). \end{equation}
Similarly, if $\un{e} \subset \un{w}$ is a subexpression, we write
\begin{equation} \beta_{\un{e}} = s_1^{e_1} s_2^{e_2} \cdots s_{d-1}^{e_{d-1}} (\alpha_{s_d}). \end{equation}
Note that the value of $e_d \in \{0,1\}$ does not affect $\beta_{\un{e}}$.

Let $\un{x} = (s_1, \ldots, s_d)$ be any expression. We write $X_{\un{x}}$ for the set of \emph{leading subexpressions} of $\un{x}$. That is, 
\begin{equation} X_{\un{x}} = \{(s_1), (s_1, s_2), \cdots, (s_1, s_2, \ldots, s_{d-1}), \un{x} \}. \end{equation}
For leading subexpressions we do not use the zero-one notation $\un{e} \subset \un{x}$, as we did for arbitrary subexpressions; this is because we will be interested later in subexpressions of leading subexpressions!

For the geometric realization, it is well-known that one can enumerate the positive roots $\Phi^+$ of a finite Coxeter group using leading subexpressions of any given reduced expression for the longest element. If $\un{x}$ is a reduced expression for the longest element of $W$, then
\begin{equation} \label{eq:phiplusdefn} \Phi^+ = \{\beta_{\un{w}} \mid \un{w} \in X_{\un{x}} \}. \end{equation}
Both $\Phi^+$ and $X_{\un{x}}$ have size equal to the number of reflections in $W$, so this description is irredundant. What happens for other realizations will be discussed in Remark \ref{rmk:posrootsotherrealizations} below. Regardless, this discussion of positive roots is only for motivational purposes.

\begin{remark} It is worth noting that one can \emph{not} use tailing subexpressions to enumerate positive roots. For example, if $Y_{\un{x}} = \{(s_d), (s_{d-1}, s_d), \ldots, (s_2, \ldots, s_d), \un{x} \}$ then $\Phi^+$ need not equal $\{ \beta_{\un{w}} \mid \un{w} \in Y_{\un{x}}\}$. Every such $\beta_{\un{w}}$ will be in the $W$-orbit of $\alpha_{s_d}$, and not all positive roots need be in the same $W$-orbit (e.g. long and short roots are in two different orbits). \end{remark}

Specifically for finite dihedral groups, we can enumerate the positive roots of the geometric representation (without redundancy) in two standard ways. Let $X_s$ be the set of leading expressions of $(s,t,s,t,\ldots)$
\begin{equation} X_s = \{s, st, sts, \ldots, stst\cdots \} \end{equation}
where the final expression has length $m$, where $m$ denotes the order
of $st$. Similarly, let
\begin{equation} X_t = \{t, ts, tst, \ldots, tsts\cdots \}. \end{equation}
We define polynomials $\pi_{s,t}$ and $\pi_{t,s}$ by the following formulas:
\begin{equation} \label{eq:pist} \pi_{s,t} = \alpha_s \cdot s(\alpha_t) \cdot st(\alpha_s) \cdots = \prod_{\un{w} \in X_s} \beta_{\un{w}} \end{equation}
and
\begin{equation} \label{eq:pits} \pi_{t,s} = \alpha_t \cdot t(\alpha_s) \cdot ts(\alpha_t) \cdots = \prod_{\un{w} \in X_t} \beta_{\un{w}}. \end{equation}
This definition makes sense in any realization, even when the realization is not faithful or balanced.

For the geometric realization we have
\begin{equation} \label{eq:twoenumerations} \{ \beta_{\un{w}} \mid \un{w} \in X_s \} = \{ \beta_{\un{w}} \mid \un{w} \in X_t\}, \end{equation}
and this is the set of positive roots $\Phi^+_{s,t}$. In particular, if $\un{w}_t \in X_t$ is the longest expression, then $\beta_{\un{w}_t} = \alpha_s$, and vice versa. Now $\pi_{s,t}$ is the product of the positive roots in $\Phi^+_{s,t}$, and \eqref{eq:twoenumerations} gives us
\begin{equation} \label{eq:pistispits} \pi_{s,t} = \pi_{t,s}. \end{equation}
Moreover, $s$ preserves the set $\Phi^+_{s,t} \setminus \{\alpha_s\}$, and sends $\alpha_s \mapsto -\alpha_s$. Thus $\pi_{s,t}$ is $s$-antiinvariant:
\begin{equation} \label{eq:pistantiinvariantfirsttime} s(\pi_{s,t}) = - \pi_{s,t}. \end{equation}

When the realization is not faithful, the sets in \eqref{eq:twoenumerations} might not have size $m$, instead giving a redundant description of a smaller set of roots. When the
realization is not balanced, the sets in \eqref{eq:twoenumerations} may be non-equal. Indeed, if $\un{w}_t \in X_t$ is the longest expression, then $\beta_{\un{w}_t}$ is a non-trivial
scalar multiple of $\alpha_s$, and all the roots coming from $X_t$ are rescalings of roots coming from $X_s$. See \S\ref{subsec:rootspolys} for more details. In these cases it is a
technical question what the appropriate definition of ``positive roots'' should be, but this is not of importance for us. The definitions of $\pi_{s,t}$ and $\pi_{t,s}$ make sense as
above, and it is only important what properties they satisfy.

\begin{ex}
To follow this example one should be familiar with two-colored quantum numbers, see \S\ref{subsec:2quantum}. In the geometric realization of the symmetric group $S_3$ with simple reflections $s$ and $t$, we have
\[ ts(\alpha_t) = \alpha_s, \quad t(\alpha_s) = s(\alpha_t), \quad \alpha_t = st(\alpha_s), \]
\[ \pi_{s,t} = \pi_{t,s} = \alpha_s \alpha_t s(\alpha_t). \]
Meanwhile, for an arbitrary realization, one has $st(\alpha_s) = [3] \alpha_s + [2]_t \alpha_t$. For any realization $[3]=0$ and $[2]_s [2]_t =1$, but $[2]_t = 1$ if and only if the realization is balanced. For example, if $[2]_t = -5$ and $[2]_s = \frac{-1}{5}$ we have
\[ -5 ts(\alpha_t) = \alpha_s, \quad t(\alpha_s) = -5 s(\alpha_t), \quad -5 \alpha_t = st(\alpha_s), \]
\[ \pi_{s,t} = -5 \cdot \pi_{t,s}. \]
\end{ex}

For the remainder of this chapter we stick with the case where the realization is balanced (but not necessarily faithful). This is sufficient to imply that \eqref{eq:pistispits} and
\eqref{eq:pistantiinvariantfirsttime} hold, and these are the properties of $\pi_{s,t}$ that we use. We identify $\pi_{s,t}$ and $\pi_{t,s}$ in our formulas for simplicity. We deal
with the unbalanced case in \S\ref{sec-unbalanced}. The proofs found in \S\ref{sec-unbalanced} also suffice to prove the desired statements \eqref{eq:pistispits} and \eqref{eq:pistantiinvariantfirsttime} in the balanced case, regardless of faithfulness.

\begin{remark} The equality \eqref{eq:pistispits} is not really essential for our proofs, so long as one carefully distinguishes between $\pi_{s,t}$ and $\pi_{t,s}$ in the formulas, as we do in \S\ref{sec-unbalanced}. The anti-invariance of $\pi_{s,t}$, namely \eqref{eq:pistantiinvariantfirsttime}, is essential for our proofs, and holds only when the realization is even-balanced (see Definition \ref{defn:balanced}). This is one reason why even-balancedness is essential for the well-definedness of $\HC$. \end{remark}

\begin{remark} \label{rmk:posrootsotherrealizations} For arbitrary realizations of finite Coxeter groups, one could fix a reduced expression $\un{w}$ for the longest element, and
define positive roots (as a multiset) using \eqref{eq:phiplusdefn}. For faithful realizations, this multiset has no nontrivial multiplicities, a well-known fact for the geometric realization. If the realization is balanced, then the
multiset of positive roots does not depend on the choice of reduced expression. Otherwise, the positive roots will depend on the choice of reduced expression, though only up to
rescaling: the lines spanned by these roots will be independent of the choice of reduced expression. We do not know where to find proofs of these observations in the literature.
Briefly, one should prove these claims using the observation that $\beta_{\un{w}} = \beta_{\un{w}'}$ in the geometric realization if the two expressions are related by a sequence of transformations of the form \begin{itemize}
\item Coxeter relations applied to all but the last index in the expression, and
\item the equality $\beta_{\un{w}_t} = \alpha_s$, which we can use to transform the end of the expression.
\end{itemize}
These transformations continue to apply verbatim for any balanced realization, and apply up to scalar for unbalanced realizations. \end{remark}

\subsection{Definition on two-color morphisms} \label{subsec:defntwocolor}

Suppose that ${\color{red}s}, {\color{blue}t} \in
S$ are two distinct reflections, such that $st$ has order $m < \infty$ in
$W$. We need to specify the image of the $2m$-valent vertex in
$(\Omega_Q W)_{\oplus}$. First, we specify the image of the
morphism
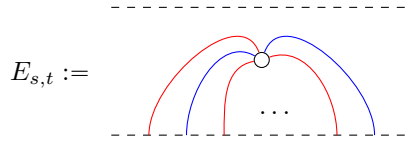
\[
E_{s,t} :=  \begin{array}{c}
              \begin{tikzpicture}
                \draw[dashed] (-.5,0) to (3.5,0); \draw[dashed] (-0.5,1.7) to (3.5, 1.7);
\node [circle,draw,inner sep=0mm,minimum size=2mm] (c) at (1.5,1) {};
\draw [red] (0,0) to [out=90,in=110] (c);  
\draw [blue] (0.5,0) to [out=90,in=150] (c);  
\draw [red] (1,0) to [out=90,in=190] (c);
\node at (1.7,0.3) {$\dots$};
\draw [red] (2.5,0) to [out=90,in=20] (c);  
\draw [blue] (3,0) to [out=90,in=70] (c);  
\end{tikzpicture} \end{array}
\]
obtained by ``twisting all arms of the $2m$-valent vertex
down.'' Given that we already know the images of the cups and caps
under $\Lambda$, we can recover knowledge of the $2m_{st}$-valent
vertex from the knowledge of the image of $E_{s,t}$ under $\Lambda$. We do this in \S\ref{subsec:altform} below.

Let $\un{w} := (s, t, s, t, \dots)$ have length $2m$, so that $B_{\un{w}}$
is the source of $E_{s,t}$. Define
\[
f : \Lambda(B_{\un{w}}) \to \Lambda(\mathbb{1})
\]
by
\begin{equation} \label{fdef}
f_{\un{e}}^\emptyset := \begin{cases} \pi_{s,t} & \text{if
    $\un{w}^{\un{e}} = {\id}$,} \\
0 & \text{otherwise.} \end{cases}
\end{equation}
(Thus, in matrix notation, $f$ consists of a single row in which every
possible non-zero entry is $\pi_{s,t}$). We set:
\begin{equation}
  \label{eq:E}
  \Lambda : E_{s,t} \mapsto f.
\end{equation}

\subsection{An alternative formula} \label{subsec:altform}
Here we give a formula for the image of
the $2m$-valent vertex before ``twisting.'' Consider the following
explicit realization of the $2m$-valent vertex:
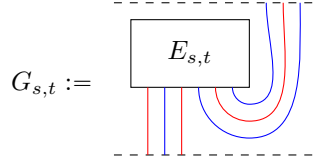
\[
G_{s,t} := 
\begin{array}{c}
  \tikz[xscale=0.45,yscale=0.45]{
                  \draw[dashed] (-.5,-3) to (5.5, -3); \draw[dashed] (-0.5,1.5) to (5.5, 1.5);
\draw (0,1) rectangle (3.5,-1);
\node at (1.75,0) {$E_{s,t}$};
\draw[red] (.5,-1) to (.5,-3);
\draw[blue] (1,-1) to (1,-3);
\draw[red] (1.5,-1) to (1.5,-3);
\draw[blue] (2,-1) to[out=-90,in=180] (3.5,-2.5) to[out=0,in=-90] (5,1.5);
\draw[red] (2.5,-1) to[out=-90,in=180] (3.5,-2) to[out=0,in=-90] (4.5,1.5);
\draw[blue] (3,-1) to[out=-90,in=180] (3.5,-1.5) to[out=0,in=-90] (4,1.5);
} \end{array}
\]
(We draw the $m = 3$ case, but the reader can guess what $G_{s,t}$ looks
like in general.)

When we express $\Lambda(G_{s,t})$ as a matrix, it will have a potential nonzero entry for each $\un{e} \subset (s,t,s,\ldots)$ and each $\un{e}' \subset (t,s,t,\ldots)$ with the same endpoint. Let the corresponding matrix entry be denoted $(G_{s,t})^{\un{e}'}_{\un{e}} \in Q$.

Recall that $X_t$ is the set of leading subexpressions of the reduced
expression $(t,s,t, \ldots)$. If $\un{w} \in
X_t$ has length $i$ and $\un{e}' \subset (t,s,t,\ldots)$ is a
subexpression, then $\un{e'}_{\subset \un{w}} := \un{e}'_{\le i} = (e'_1, e'_2, \ldots, e'_i)$ is a subexpression of $\un{w}$. Let
\begin{equation} \zeta(\un{e}') = \prod_{\un{w} \in X_t} \beta_{\un{e}'_{\subset \un{w}}} = \prod_{i = 1}^m s_1^{e_1'} s_2^{e_2'} \dots s_{i-1}^{e_{i-1}'}(\alpha_{s_{i}}). \end{equation}
Note that $\zeta(\un{e}')$ is a product of $m$ roots, and has the same degree as $\pi_{s,t}$, regardless of the choice of $\un{e}'$.

\begin{lem} \label{lem:E}
 If $\un{e} \subset (s,t,s,\ldots)$ and $\un{e}' \subset (t,s,t,\ldots)$ have the same endpoint, then
  \begin{equation}
    \label{eq:2}
    (G_{s,t})_{\un{e}}^{\un{e}'} = \frac{\pi_{s,t}}{\zeta(\un{e}')}.
  \end{equation}
\end{lem}
We delay the proof of this lemma until the end of this section.

\begin{remark}
In the matrix of $\Lambda(G_{s,t})$, all potential nonzero entries are
nonzero, and the nonzero entries in each row are equal.  The lemma
also shows that the formula for the image of the $2m$-valent
  vertex $G_{s,t}$ is more complicated than its twisted cousin $E_{s,t}$.
\end{remark}

\begin{ex} \label{ex:m2} Let us consider the special case when $m=2$, so $s$ and $t$ commute. Then $\pi_{s,t} = \alpha_s \alpha_t$, and $\pi_{s,t} = \zeta(\un{e'})$ for all subexpressions $\un{e'} \subset (t,s)$.  Thus the matrix of $\Lambda(G_{s,t})$ is easy: all entries which could be nonzero are equal to $1$. \end{ex}

\begin{ex} Suppose that $m=3$. Then
	\[ (G_{s,t})_{(000)}^{(000)} = (G_{s,t})_{(101)}^{(000)} = \frac{\alpha_s \alpha_t s(\alpha_t)}{\alpha_t \alpha_s \alpha_t} = \frac{s(\alpha_t)}{\alpha_t}. \]
\end{ex}

Let us deduce two consequences of this lemma. If $\un{e} = (1,1,\dots
1)$ and $\un{e'} = (1,1,\dots, 1, 1)$ then
\[
\zeta(e') = \prod_{i = 1}^m s_1 s_2 \dots s_{i-1}(\alpha_{s_{i}}) = \pi_{s,t}
\]
and hence
\begin{equation}
  \label{eq:111}
  (G_{s,t})_{\un{e}}^{\un{e'}} = 1.
\end{equation} Similarly, if $\un{e} =
(0,1,\dots, 1,1)$ and $\un{e'} = (1,1,\dots, 1, 0)$ then
\begin{equation}
  \label{eq:110}
  (G_{s,t})_{\un{e}}^{\un{e'}} = 1.
\end{equation}

\begin{proof}[Proof of Lemma \ref{lem:E}]
Consider the
following morphism (again we draw the case $m = 3$, and leave it to
the reader to imagine the general case):
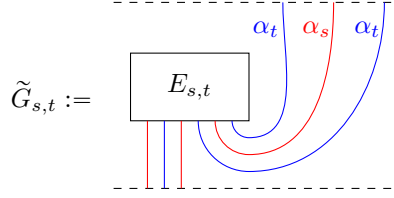
\[
\widetilde{G}_{s,t} := 
\begin{array}{c}
  \tikz[xscale=0.45,yscale=0.45]{
                    \draw[dashed] (-.5,-3) to (8, -3); \draw[dashed] (-0.5,2.5) to (8, 2.5);
\draw (0,1) rectangle (3.5,-1);
\node at (1.75,0) {$E_{s,t}$};
\draw[red] (.5,-1) to (.5,-3);
\draw[blue] (1,-1) to (1,-3);
\draw[red] (1.5,-1) to (1.5,-3);
\draw[blue] (2,-1) to[out=-90,in=180] (3.5,-2.5) to[out=0,in=-90] (7.5,2.5);
  \draw[red] (2.5,-1) to[out=-90,in=180] (3.5,-2) to[out=0,in=-90] (6,2.5);
\draw[blue] (3,-1) to[out=-90,in=180] (3.5,-1.5) to[out=0,in=-90]
  (4.5,2.5);
 \node[blue] at (4,1.7) {$\alpha_t$};
  \node[red] at (5.5,1.7) {$\alpha_s$};
   \node[blue] at (7,1.7) {$\alpha_t$};
  } \end{array}
\]
We compute the image of $\widetilde{G}_{s,t}$ both via \eqref{eq:E}
and via the lemma, and see that they agree.

If we compute the image of $\widetilde{G}_{s,t}$ using \eqref{eq:E}
and \eqref{eq:cup} then each of the simple roots appearing in the
formulas for the cup cancels those in the diagram for
$\widetilde{G}_{s,t}$, and we deduce that all potentially non-zero
entries are equal to $\pi_{s,t}$ (see Remark
\ref{rmk:nicerpeacecup}).

On the other hand, if we compute  image of $\widetilde{G}_{s,t}$ using
\eqref{eq:2} then, for each subexpression $\un{e}' \subset
(t,s,t,\dots)$, the simple roots along the top of $\widetilde{G}_{s,t}$ will pull to the left to become
$\zeta(\un{e}')$, and hence all potentially non-zero
entries are equal to $\pi_{s,t}$.
\end{proof}

\section{Checking the relations} \label{sec:checkrelations}

\subsection{One-color}

Checking the one-color relations is routine. We give some
sample computations.

The relation
\[
\begin{array}{c}
\tikz[xscale=0.35,yscale=0.35]{
\draw[dashed] (2,1.5) to (4,1.5); \draw[dashed] (2,-1) to (4,-1);
\draw[color=red] (3,-.5) to (3,1);
\node[circle,fill,draw,inner sep=0mm,minimum size=1mm,color=red] at (3,-.5) {};
\node[circle,fill,draw,inner sep=0mm,minimum size=1mm,color=red] at (3,1) {};
}
\end{array}
=
\begin{array}{c}
\tikz[xscale=0.35,yscale=0.35]{
\draw[dashed] (2,1.5) to (4,1.5); \draw[dashed] (2,-1) to (4,-1);
\node[color=red] at (3,0.25) {$\a$};
}
\end{array}
\]
follows from the identity
\begin{gather*}
\left ( \begin{matrix} \a & 0 \end{matrix}
\right )
  \left ( \begin{matrix} 1 \\ 0 \end{matrix}
\right ) = ( \a)
\end{gather*}
the first matrix on the left is \eqref{eq:idot}, the second \eqref{eq:ddot}.

The relation
\begin{equation} \label{eq:unit}
\begin{array}{c}
\tikz[xscale=0.35,yscale=0.35]{
\draw[dashed] (2,1.5) to (4,1.5); \draw[dashed] (2,-1) to (4,-1);
\draw[color=red] (2.3,1.5) to[out=-90,in=150] (3,0);
\draw[color=red] (3.7,1) to[out=-90,in=30] (3,0);
\draw[color=red] (3,0) to (3,-1);
\node[circle,fill,draw,inner sep=0mm,minimum size=1mm,color=red] at (3.7,1) {};
}
\end{array}
=
\begin{array}{c}
\tikz[xscale=0.35,yscale=0.35]{
\draw[dashed] (2,1.5) to (4,1.5); \draw[dashed] (2,-1) to (4,-1);
\draw[color=red] (3,1.5) to (3,-1);
}
\end{array}
\end{equation}
follows from the identity
\begin{gather*}
\left (
  \begin{matrix} \a & 0 & 0 & 0 \\ 0 & 0 & -\a & 0 \end{matrix}
\right )
\left (
  \begin{matrix} 1/\a & 0 \\ 0 & 1/\a \\ 0 & -1/\a \\ -1/\a & 0 \end{matrix}
\right)
= \left ( \begin{matrix} 1 & 0 \\ 0 & 1 \end{matrix} \right).
\end{gather*}
The second matrix on the left is \eqref{eq:peace}. The first matrix is obtained from \eqref{eq:idot}, by
applying \eqref{eq:polyact}.

The relation
\[
\begin{array}{c}
\tikz[xscale=0.35,yscale=-0.35]{
\draw[dashed] (2,1.5) to (4,1.5); \draw[dashed] (2,-1) to (4,-1);
\draw[color=red] (2.3,1.5) to[out=-90,in=150] (3,0);
\draw[color=red] (3.7,1) to[out=-90,in=30] (3,0);
\draw[color=red] (3,0) to (3,-1);
\node[circle,fill,draw,inner sep=0mm,minimum size=1mm,color=red] at (3.7,1) {};
}
\end{array}
=
\begin{array}{c}
\tikz[xscale=0.35,yscale=0.35]{
\draw[dashed] (2,1.5) to (4,1.5); \draw[dashed] (2,-1) to (4,-1);
\draw[color=red] (3,1.5) to (3,-1);
}
\end{array}
\]
follows from the identity
\begin{gather*}
\left (
  \begin{matrix} 1 & 0 & 0 & 1 \\ 0 & 1 & 1 & 0 \end{matrix}
\right )
\left (
  \begin{matrix} 1 & 0 \\ 0 & 0 \\ 0 & 1 \\ 0 & 0 \end{matrix}
\right)
= \left ( \begin{matrix} 1 & 0 \\ 0 & 1 \end{matrix} \right).
\end{gather*}
The first matrix on the left is \eqref{eq:mercedes}. The second matrix is obtained from \eqref{eq:ddot}, by
applying \eqref{eq:polyact}.

The relation
\[
\begin{array}{c}
\tikz[xscale=0.35,yscale=0.35]{
\draw[dashed] (2,1.5) to (4,1.5); \draw[dashed] (2,-1) to (4,-1);
\draw[color=red] (3,-1) to (3,-.5);
\draw[color=red] (3,1) to (3,1.5);
\draw[color=red] (3,-.5) to[out=150,in=-150] (3,1);
\draw[color=red] (3,-.5) to[out=30,in=-30] (3,1);
}
\end{array}
=
0
\]
follows from the identity
\begin{gather*}
\left (
  \begin{matrix} 1 & 0 & 0 & 1 \\ 0 & 1 & 1 & 0 \end{matrix}
\right )
\left ( \begin{matrix} 1/\a & 0 \\ 0 & 1/\a \\ 0 & -1/\a \\ -1/\a & 0 \end{matrix} \right) = 0.
\end{gather*}
The first matrix on the left is \eqref{eq:mercedes}, the second \eqref{eq:peace}.

We leave it to the reader to check the other one-color relations.

\subsection{Cyclicity} \label{subsec:cyclicityof2m} We begin to check the two-color relations, starting with the cyclicity of the $2m$-valent vertex.

When doing these computations below, if $\un{e} \subset \un{w}$, we
write $[\un{e}]$ for the identity endomorphism of the corresponding
summand of $\Lambda(B_{\un{w}})$.

\begin{ex}
For example, if we write
\[ \Lambda(\phi) \colon [\un{e}] \mapsto \alpha_s [\un{f}] + \alpha_t
  [\un{g}] \]
this means that the $\un{e}$ column of the matrix for $\Lambda(\phi)$
has only two nonzero entries, $\alpha_s$ in row $\un{f}$ and
$\alpha_t$ in row $\un{g}$.
  
\end{ex}

Cyclicity of the $2m$-valent vertex is equivalent to the following relation (together with its analogue after swapping $s$ and $t$):
\begin{equation} \label{eq:twistedcyclic}
\begin{array}{c}
  \tikz[xscale=0.45,yscale=0.45]{
  \draw[dashed] (-1,-2.5) to (4.5, -2.5); \draw[dashed] (-1,2) to (4.5,2);
\draw (0,1) rectangle (3.5,-1);
\node at (1.75,0) {$E_{s,t}$};
\draw[red] (.5,-1) to[out=-90,in=0] (0,-1.5) to[out=180,in=-90]
(-.5,2);
\draw[blue] (1,-1) to (1,-2.5);
\draw[red] (2.5,-1) to (2.5,-2.5);
\draw[blue] (3,-1) to (3,-2.5);
\node at (1.8,-1.75) {$\dots$};
} \end{array} = 
\begin{array}{c}
  \tikz[xscale=-0.45,yscale=0.45]{
    \draw[dashed] (-1,-2.5) to (4.5, -2.5); \draw[dashed] (-1,2) to (4.5,2);
\draw (0,1) rectangle (3.5,-1);
\node at (1.75,0) {$E_{t,s}$};
\draw[red] (.5,-1) to[out=-90,in=0] (0,-1.5) to[out=180,in=-90]
(-.5,2);
\draw[blue] (1,-1) to (1,-2.5);
\draw[red] (2.5,-1) to (2.5,-2.5);
\draw[blue] (3,-1) to (3,-2.5);
\node at (1.8,-1.75) {$\dots$};
} \end{array} .
\end{equation}

\begin{lem} $\Lambda$ preserves the relation
  \eqref{eq:twistedcyclic}. \label{lem:cyclicity} \end{lem}

\begin{proof} Let $\un{e}$ be a subexpression of the source, and consider the image of $[\un{e}]$ under the morphisms on
  the left and right hand side. Because the image of these morphisms
  is $\Lambda(B_s) = r_{\id} \oplus r_s$, $[\un{e}]$ is mapped to zero
  unless it has endpoint $\id$ or $s$.

\emph{Case 1:} $\un{e}$ is a subexpression for the identity. Then on the
left-hand side, $[\un{e}]$ is mapped as follows:
\[
[\un{e}] \mapsto \frac{1}{\a_s}[00\un{e}] - \frac{1}{\a_s}[11\un{e}]
\mapsto \frac{\pi_{s,t}}{\a_s} [0].
\]
(This is a composition of two maps, the red cup followed by $\id_{B_s}
\otimes E_{s,t}$. Applying the red cup yields the expression in the
middle by \eqref{eq:cup}. Afterwards, only the first term survives $\id_{B_s}
\otimes E_{s,t}$, since the second term applies $E_{s,t}$ to a
subexpression for $s$ rather than for $\id$.) Similarly, on the right-hand side we find:
\[
[\un{e}] \mapsto \frac{1}{\a_s}[\un{e}00] -
\frac{1}{\a_s}[\un{e}11] \mapsto \frac{\pi_{t,s}}{\a_s}[0].
\]
Thus $[\un{e}]$ has the same image on both sides, as $\pi_{s,t}= \pi_{t,s}$ (see \eqref{eq:pistispits}).

\emph{Case 2:} $\un{e}$ is a subexpression for $s$. On the left hand
side, $[\un{e}]$ is mapped as follows:
\[
[\un{e}] \mapsto \frac{1}{\a_s}[00\un{e}] - \frac{1}{\a_s}[11\un{e}]
\mapsto -\frac{s(\pi_{s,t})}{\a_s} [1].
\]
The $s$ in the numerator came from \eqref{eq:polyact}.
On the right-hand side we find:
\[
[\un{e}] \mapsto \frac{1}{s(\a_s)}[\un{e}00] -
\frac{1}{s(\a_s)}[\un{e}11] =  - \frac{1}{\a_s}[\un{e}00] +
\frac{1}{\a_s}[\un{e}11] \mapsto \frac{\pi_{t,s}}{\a_s}[1].
\]
Thus $[\un{e}]$ has the same image on both sides since
$-s(\pi_{s,t}) = \pi_{t,s}.$ \end{proof}

\subsection{Properties of the Jones-Wenzl projector} \label{subsec:JWmain}

We now turn to the Jones-Wenzl relation: 
\begin{equation} \label{Eq:JWreln}
\begin{array}{c}
  \tikz[xscale=0.45,yscale=0.45]{
      \draw[dashed] (-1,-2.5) to (4.5, -2.5); \draw[dashed] (-1,1.5) to (4.5,1.5);
\draw (0,1) rectangle (4,-1);
\node at (1.5,0) {$E_{s,t}$};
\draw[red] (.5,-1) to (.5,-2.5);
\draw[blue] (1,-1) to (1,-2);
  \draw[red] (1.5,-1) to (1.5,-2.5);
\draw[blue] (2,-1) to (2,-2.5);
%  \draw[red] (2.5,-1) to (2.5,-2.5);
\node[circle,fill,draw,inner sep=0mm,minimum size=1mm,color=blue] at (1,-2) {};
\node at (3,-1.75) {$\dots$};
  } \end{array}
=
\begin{array}{c}
  \tikz[xscale=0.45,yscale=0.45]{
    \draw[dashed] (-1,-2.5) to (4.5, -2.5); \draw[dashed] (-1,1.5) to  (4.5,1.5); % framing
\draw (0,1) rectangle (3.5,-1);
\node at (1.75,0) {$JW_{s,t}$};
\draw[red] (.5,-2.5) to[out=90,in=-120] (1,-1.5);
%\draw[blue] (1,-1) to (1,-2);
  \draw[red] (1.5,-2.5) to[out=90,in=-60] (1,-1.5);
  \draw[red] (1,-1.5) to (1,-1);
\draw[blue] (2,-1) to (2,-2.5);
%  \draw[red] (2.5,-1) to (2.5,-2.5);
%\node[circle,fill,draw,inner sep=0mm,minimum size=1mm,color=blue] at (1,-2) {};
\node at (3,-1.75) {$\dots$};
} \end{array}
\end{equation}
Here, $JW_{s,t}$ is a particular linear combination of degree $2$ diagrams consisting entirely of
univalent and trivalent vertices. We will need a number of basic facts about $JW_{s,t}$. Let us recall the construction of $JW_{s,t}$ so we may justify these
facts.

For this discussion we assume the reader is familiar with the two-colored Temperley-Lieb category $\TTL$ and the deformation retract functor $\Sigma \colon \TTL \to \HC$. The
best exposition for this can be found in the recent book \cite[\S
9.2]{SBook}; an older exposition can be found in \cite[\S 4.3 and
\S 5.3.2]{EDC}.

The functor $\Sigma$ sends a cup in $\TTL$ to a morphism in $\HC$ which we call a \emph{pitchfork}:
\begin{equation} \label{eq:redpitchfork}
            \begin{array}{c}
              \begin{tikzpicture}[xscale=0.5,yscale=.5]
          \draw[fill,red!20!white] (1,-1.5) rectangle (5.5,0);      
  \begin{scope}
   \clip (1,-1.5) rectangle (5.5,0);
    \draw[fill=blue!20!white] (3.25,0) circle (1cm);
 \end{scope}
 \draw[dashed] (1,-1.5) to (5.5, -1.5);
 \draw[dashed] (1,0) to (5.5,0);
  \end{tikzpicture}
  \end{array} \mapsto
            \begin{array}{c}
                  \begin{tikzpicture}[scale=0.5]
                  \draw[dashed] (1,-1.5) to (5.5, -1.5); \draw[dashed] (1,0) to (5.5,0);
                  \def\x{3.25}; \def\y{-1};
                  \draw[red!50!white] (2.5,0) to[out=-90,in=180] (\x,\y) to[out=0,in=-90] (4,0);
                  \draw[red!50!white] (\x,\y) -- (\x,-1.5);
                  \draw[blue] (\x,0) -- (\x,-0.5);
                  \node[circle,fill,draw,inner sep=0mm,minimum size=1mm,color=blue] at (\x,-0.5) {};
                  \end{tikzpicture}
            \end{array}
          \end{equation}  
The image of a cap we call an \emph{inverted pitchfork}:
\begin{equation}
            \begin{array}{c}
              \begin{tikzpicture}[xscale=0.5,yscale=-.5]
          \draw[fill,red!20!white] (1,-1.5) rectangle (5.5,0);      
  \begin{scope}
   \clip (1,-1.5) rectangle (5.5,0);
    \draw[fill=blue!20!white] (3.25,0) circle (1cm);
 \end{scope}
 \draw[dashed] (1,-1.5) to (5.5, -1.5);
 \draw[dashed] (1,0) to (5.5,0);
  \end{tikzpicture}
  \end{array} \mapsto
            \begin{array}{c}
                  \begin{tikzpicture}[xscale=0.5,yscale=-.5]
                  \draw[dashed] (1,-1.5) to (5.5, -1.5); \draw[dashed] (1,0) to (5.5,0);
                  \def\x{3.25}; \def\y{-1};
                  \draw[red!50!white] (2.5,0) to[out=-90,in=180] (\x,\y) to[out=0,in=-90] (4,0);
                  \draw[red!50!white] (\x,\y) -- (\x,-1.5);
                  \draw[blue] (\x,0) -- (\x,-0.5);
                  \node[circle,fill,draw,inner sep=0mm,minimum size=1mm,color=blue] at (\x,-0.5) {};
                  \end{tikzpicture}
            \end{array}
\end{equation}
All morphisms in the image of $\Sigma$ have degree $0$. The functor
$\Sigma$ is well-defined as long as the blue circle inside a red
region in $\TTL$ evaluates to $\a_{\reds}^{\vee}(\a_{\bluet})$, and
the red circle inside a blue region evaluates to
$\a_{\bluet}^{\vee}(\a_{\reds})$.
 
There is exactly one $(n,n)$ crossingless matching with no cups or caps: the identity matching. This is the unique crossingless matching which matches each of the $n$ bottom
boundary points to a top boundary point. Any other matching pairs two bottom points and two top points, and has both a cap on bottom and a cup on top.
Consequently, the image under $\Sigma$ of any non-identity two-colored crossingless matching has a pitchfork on top and an inverted pitchfork on bottom.

\begin{remark} Cups on top and caps on the bottom are just
  subdiagrams of a larger crossingless matching, and after applying
  $\Sigma$, the same can be said of pitchforks. When choosing a
  deformation retract of a crossingless matching to define $\Sigma$,
  one often chooses a retract where the pitchfork does not appear
  precisely as a subdiagram, but some slight modification thereof. For
  example, when taking the deformation retract of
  \[
            \begin{array}{c}
              \begin{tikzpicture}[xscale=0.5,yscale=.5]
          \draw[fill,blue!20!white] (1,-1.5) rectangle (5.5,0);      
  \begin{scope}
   \clip (1,-1.5) rectangle (5.5,0);
    \draw[fill=red!20!white] (3.25,0) circle (1.2cm);
   \draw[fill=blue!20!white] (3.25,0) circle (.6cm);
  \end{scope}
 \draw[dashed] (1,-1.5) to (5.5, -1.5);
 \draw[dashed] (1,0) to (5.5,0);
  \end{tikzpicture}
            \end{array}
            \]
            one obtains the diagram 
            \begin{equation} \label{dottedpitchfork} \begin{array}{c}
    \begin{tikzpicture}[scale=0.5] 
                  \draw[dashed] (1,-2) to (5.5, -2); \draw[dashed] (1,0) to (5.5,0);
% bottom dots
%      \foreach \x in {1.5,5} \node at (\x,-0.75) {$\dots$};
% % pitchfork:
       \def\x{3.25}; \def\y{-1};
       \draw[red!50!white] (2.5,0) to[out=-90,in=180] (\x,\y) to[out=0,in=-90] (4,0);
       \draw[white] (\x,\y) -- (\x,-1.5);
       \draw[blue] (\x,0) -- (\x,-0.5);
 \node[circle,fill,draw,inner
 sep=0mm,minimum size=1mm,color=blue] at (\x,-0.5) {};
  \draw[blue] (\x,-1.5) to (\x, -2);
  \draw[blue] (2,0) to[out=-90,in=150] (\x,-1.5) to[out=30,in=-90] (4.5,0);
    \end{tikzpicture} \end{array}. \end{equation}
The cap on top no longer gives rise to a subdiagram identical to the red pitchfork from \eqref{eq:redpitchfork}, but instead one finds a red pitchfork without a handle.

Of course, the red pitchfork without a handle is just the composition of a red pitchfork with a red univalent vertex, thanks to the unit relation \eqref{eq:unit}, so this diagram factors through the pitchfork. Similarly, another deformation retract might have several pitchforks which are grafted to the same handle, which can be accounted for by the one-color associativity relation.
Below, when we say that a diagram representing a morphism in $\HC$ has a \emph{pitchfork on top}, we mean that it factors as the composition of a pitchfork and another diagram. For
example, \eqref{dottedpitchfork} has a pitchfork on top in this sense. \end{remark}

Let us consider two-colored crossingless matchings where the leftmost region is colored $\reds$. The two-colored Jones-Wenzl projector $JW_{{}_{\reds} n}$, if it exists, is a particular linear combination of $(n,n)$ crossingless matchings. It is uniquely determined by the following properties in $\TTL$:
\begin{itemize} \item It becomes zero when any cap is placed on top (we say it is \emph{killed by caps}).
	\item It becomes zero when any cup is placed on bottom (we say it is \emph{killed by cups}).
	\item The coefficient of the identity crossingless matching is $1$.
\end{itemize}
The existence of such a linear combination is not guaranteed, and relies upon properties of the scalars $\a_{\reds}^{\vee}(\a_{\bluet})$ and $\a_{\bluet}^{\vee}(\a_{\reds})$ in the base ring $\Bbbk$. Similarly, one can discuss $JW_{n_{\bluet}}$. Jones-Wenzl projectors are discussed at great length in \S\ref{sec:JW}.

We let $\JWalt_{n,s,t}$  denote the image of $JW_{{}_{\reds} n}$ (if it exists) under $\Sigma$. Then $\JWalt_{n,s,t}$ is an endomorphism of
the alternating length $n+1$ object $B_s B_t B_s \cdots$ in $\HC$. The fact that Jones-Wenzl projectors are killed by cups
implies after applying $\Sigma$ that $\JWalt_{n,s,t}$ is \emph{killed
  by pitchforks on bottom}. That is
\begin{equation}
  \label{eq:pitchfork}
\begin{array}{c}
 \tikz[xscale=0.45,yscale=0.45]{
\draw (0,1) rectangle (3,-1);
\node at (1.5,0) {$\JWalt_{n,s,t}$};
\draw[red] (.5,-1) to (.5,-1.5);
\draw[blue] (1,-1) to (1,-1.5);
\draw[red] (2.5,-1) to (2.5,-1.5);
\node at (1.8,-1.25) {$\dots$};
\draw (0,-1.5) rectangle (3,-3.5);
\node at (1.5,-2.5) {$P$};
} \end{array} = 0
\end{equation}
whenever $P$ has a pitchfork on top.
Similarly, $\JWalt_{n,s,t}$ is \emph{killed by (inverted) pitchforks on top}; when postcomposing with a morphism that has an inverted pitchfork on bottom, you get zero. We call these properties \emph{death by pitchfork}.

The fact that Jones-Wenzl projectors have coefficient $1$ of the
identity, and all other diagrams have cups and caps, implies that
\begin{equation}
  \label{eq:3}
  \begin{array}{c}
  \begin{tikzpicture}[xscale=.5,yscale=.7]
     \draw[dashed] (-1.6,-1) to (1.6, -1);
     \draw[dashed] (-1.6,1) to (1.6,1);
     \draw[red] (-1.2,-1) to (-1.2,1);
     \draw[blue] (-.5,-1) to (-.5,1);
     \node at (.5,-.8) {$\dots$};
     \node at (.5,.8) {$\dots$};
     \draw[red] (1.2,-1) to (1.2,1);
     \draw[fill=white] (-1.4,-.6) rectangle (1.4,.6);
     \node at (0,0) {$\JWalt_{n,s,t}$};
\end{tikzpicture}
  \end{array}
  =
    \begin{array}{c}
  \begin{tikzpicture}[xscale=.5,yscale=.7]
     \draw[dashed] (-1.6,-1) to (1.6, -1);
     \draw[dashed] (-1.6,1) to (1.6,1);
     \draw[red] (-1.2,-1) to (-1.2,1);
     \draw[blue] (-.5,-1) to (-.5,1);
     \node at (.5,0) {$\dots$};
     \draw[red] (1.2,-1) to (1.2,1);
\end{tikzpicture}
    \end{array}
    +
  \begin{array}{c}
  \begin{tikzpicture}[xscale=.5,yscale=.7]
     \draw[dashed] (-1.6,-1) to (1.6, -1);
     \draw[dashed] (-1.6,1) to (1.6,1);
     \draw[red] (-1.2,-1) to (-1.2,1);
     \draw[blue] (-.5,-1) to (-.5,1);
     \node at (.5,-.8) {$\dots$};
     \node at (.5,.8) {$\dots$};
     \draw[red] (1.2,-1) to (1.2,1);
     \draw[fill=white] (-1.4,-.6) rectangle (1.4,.6);
     \node at (0,0) {$Q$};
\end{tikzpicture}
  \end{array}
\end{equation}
where $Q$ represents a linear combination of diagrams, and every diagram in $Q$ has both a pitchfork on top and an inverted pitchfork on bottom. This, combined with death by pitchfork, is the easiest way to observe that $\JWalt_{n,s,t}$ is an idempotent.

\begin{remark} \label{remark:colorparity} The pictures above are drawn
  as though $n$ were even. By definition, the source and target of
  $\JWalt_{n,s,t}$ should alternate {\color{red} red}, {\color{blue}
    blue}, etc. Whether the final strand is red or blue depends on the parity of $n$. \end{remark}
	
\begin{remark} Any Jones-Wenzl projector is fixed by the contravariant equivalence which applies a vertical flip to all diagrams. After all, the defining properties of the
Jones-Wenzl projector in $\TTL$ are invariant under applying the vertical flip. \end{remark}

As for any morphism in the image of the functor $\Sigma$, the
Jones-Wenzl projector commutes with trivalent vertices on the side, in
the sense that
\begin{equation} \label{JWcommutetri}         \begin{array}{c}
   \begin{tikzpicture}[xscale=.5,yscale=1]
     \draw[dashed] (-2,-1) to (1.6, -1);
     \draw[dashed] (-2,1) to (1.6,1);
     \draw[red] (-1.2,-1) to (-1.2,1);
     \draw[blue] (-.5,-1) to (-.5,1);
     \node at (.5,-.8) {$\dots$};
     \node at (.5,.8) {$\dots$};
     \draw[red] (1.2,-1) to (1.2,1);
     \draw[fill=white] (-1.4,-.6) rectangle (1.4,.6);
     \node at (0,0) {$\JWalt_{n,s,t}$};
     \draw[red] (-1.2,-.8) to[out=180,in=-90] (-1.8,1);
\end{tikzpicture}
        \end{array}
=
                \begin{array}{c}
  \begin{tikzpicture}[xscale=.5,yscale=1]
     \draw[dashed] (-2,-1) to (1.6, -1);
     \draw[dashed] (-2,1) to (1.6,1);
     \draw[red] (-1.2,-1) to (-1.2,1);
     \draw[blue] (-.5,-1) to (-.5,1);
     \node at (.5,-.8) {$\dots$};
     \node at (.5,.8) {$\dots$};
     \draw[red] (1.2,-1) to (1.2,1);
     \draw[fill=white] (-1.4,-.6) rectangle (1.4,.6);
     \node at (0,0) {$\JWalt_{n,s,t}$};
          \draw[red] (-1.2,.8) to[out=180,in=-90] (-1.8,1);
\end{tikzpicture}
        \end{array}\end{equation}
As a consequence of \eqref{JWcommutetri} and \eqref{eq:unit} we have
\begin{equation} \label{unitaroundJW}
        \begin{array}{c}
   \begin{tikzpicture}[xscale=.5,yscale=1]
     \draw[dashed] (-2,-1) to (1.6, -1);
     \draw[dashed] (-2,1) to (1.6,1);
     \draw[red] (-1.2,-1) to (-1.2,.8);
     \draw[blue] (-.5,-1) to (-.5,1);
     \node at (.5,-.8) {$\dots$};
     \node at (.5,.8) {$\dots$};
     \draw[red] (1.2,-1) to (1.2,1);
     \draw[fill=white] (-1.4,-.6) rectangle (1.4,.6);
     \node at (0,0) {$\JWalt_{n,s,t}$};
     \draw[red] (-1.2,-.8) to[out=180,in=-90] (-1.8,1);
           \node[circle,fill,draw,inner
 sep=0mm,minimum size=1mm,color=red] at (-1.2,.8) {};
\end{tikzpicture}
        \end{array}
=
                \begin{array}{c}
  \begin{tikzpicture}[xscale=.5,yscale=1]
     \draw[dashed] (-2,-1) to (1.6, -1);
     \draw[dashed] (-2,1) to (1.6,1);
     \draw[red] (-1.2,-1) to (-1.2,1);
     \draw[blue] (-.5,-1) to (-.5,1);
     \node at (.5,-.8) {$\dots$};
     \node at (.5,.8) {$\dots$};
     \draw[red] (1.2,-1) to (1.2,1);
     \draw[fill=white] (-1.4,-.6) rectangle (1.4,.6);
     \node at (0,0) {$\JWalt_{n,s,t}$};
\end{tikzpicture}
        \end{array}
\end{equation}

In this paper we are only interested in a particular Jones-Wenzl
morphism, where $n = m-1$ and $m = m_{s,t} < \infty$. One of the required assumptions for the category $\HC$ to be well-defined is that $JW_{{}_{\reds} (m-1)}$ and $JW_{{}_{\bluet} (m-1)}$ both exist. A second assumption is that these Jones-Wenzl projectors are \emph{rotatable}. We wish to spell out the rotatable property in the context of $\HC$, for the morphism
\[
\JWalt_{s,t} := \JWalt_{m_{st}-1,s,t}.
\]
However, $\Sigma$ is not a monoidal functor, and consequently rotation in $\TTL$ does not immediately translate to rotation in $\HC$. To discuss rotation more easily, we need to focus our attention on morphisms of degree $+2$ rather than degree $0$.

Define
\begin{equation}
    \begin{array}{c}
  \begin{tikzpicture}[xscale=.5,yscale=.7]
     \draw[dashed] (-1.6,-1) to (1.6, -1);
     \draw[dashed] (-1.6,1) to (1.6,1);
     \draw[red] (-1.2,-1) to (-1.2,0);
     \draw[blue] (-.7,-1) to (-.7,1);
     \node at (0,-.8) {$\dots$};
          \node at (0,.8) {$\dots$};
     \draw[blue] (.7,-1) to (.7,1);
     \draw[red] (1.2,0) to (1.2,1);
     \draw[fill=white] (-1.4,-.6) rectangle (1.4,.6);
     \node at (0,0) {$\JWalttwo_{s,t}$};
\end{tikzpicture}
    \end{array}
    =
    \begin{array}{c}
  \begin{tikzpicture}[xscale=.5,yscale=.7]
     \draw[dashed] (-1.6,-1) to (1.6, -1);
     \draw[dashed] (-1.6,1) to (1.6,1);
     \draw[red] (-1.2,-1) to (-1.2,.8);
     \draw[blue] (-.7,-1) to (-.7,1);
     \node at (0,-.8) {$\dots$};
          \node at (0,.8) {$\dots$};
     \draw[blue] (.7,-1) to (.7,1);
     \draw[red] (1.2,-.8) to (1.2,1);
     \draw[fill=white] (-1.4,-.6) rectangle (1.4,.6);
     \node at (0,0) {$\JWalt_{s,t}$};
      \node[circle,fill,draw,inner
      sep=0mm,minimum size=1mm,color=red] at (-1.2,.8) {};
      \node[circle,fill,draw,inner
      sep=0mm,minimum size=1mm,color=red] at (1.2,-.8) {};
    \end{tikzpicture}
  \end{array}
\end{equation}
Then $\JWalttwo_{s,t}$ is a degree $2$ map from $B_s B_t \cdots$ to $B_t B_s \cdots$, where these are the two reduced expressions of length $m-1$. Using \eqref{unitaroundJW} we also have
\begin{equation}
  \label{eq:4}
    \begin{array}{c}
  \begin{tikzpicture}[xscale=.5,yscale=.9]
     \draw[dashed] (-1.6,-1) to (1.6, -1);
     \draw[dashed] (-1.6,1) to (1.6,1);
     \draw[red] (-1.2,-1) to (-1.2,1);
     \draw[blue] (-.5,-1) to (-.5,1);
     \node at (.5,-.8) {$\dots$};
     \node at (.5,.8) {$\dots$};
     \draw[red] (1.2,-1) to (1.2,1);
     \draw[fill=white] (-1.4,-.6) rectangle (1.4,.6);
     \node at (0,0) {$\JWalt_{s,t}$};
\end{tikzpicture}
  \end{array}
      =
     \begin{array}{c}
  \begin{tikzpicture}[xscale=.5,yscale=.9]
     \draw[dashed] (-2,-1) to (2, -1);
     \draw[dashed] (-2,1) to (2,1);
     \draw[red] (-1.2,-1) to (-1.2,0);
     \draw[blue] (-.7,-1) to (-.7,1);
     \node at (0,-.8) {$\dots$};
          \node at (0,.8) {$\dots$};
     \draw[blue] (.7,-1) to (.7,1);
     \draw[red] (1.2,0) to (1.2,1);
     \draw[fill=white] (-1.4,-.6) rectangle (1.4,.6);
     \node at (0,0) {$\JWalttwo_{s,t}$};
     \draw[red] (1.2,.8) to[out=0,in=90] (1.8,.6) to (1.8,-1);
     \draw[red] (-1.2,-.8) to[out=180,in=-90] (-1.8,-.6) to (-1.8,1);
   \end{tikzpicture}
    \end{array}
\end{equation}

\begin{remark} One should think that there are two different ways to
  apply a deformation retract to the same linear combination of
  crossingless 
matchings. One produces a degree zero morphism on the planar strip, while the other produces a degree $+2$ morphism on the planar disk with a shorter boundary. The first process is
functorial, hence the functor $\Sigma$, while the second process is rotationally equivariant. See \cite[Equations (9.14) and (9.15)]{SBook} for more details. Hence $\JWalt_{s,t}$ and
$\JWalttwo_{s,t}$ are just two different kinds of deformation retract of the Jones-Wenzl projector. \end{remark}

The Jones-Wenzl projector $JW_{m-1}$ in $\TTL$ is \emph{rotatable} if its rotation is still orthogonal to all cups and caps, and thus is a scalar multiple of the Jones-Wenzl projector. If so, applying the degree $2$ deformation retract, we directly obtain
    \begin{equation}
      \label{eq:5}
           \begin{array}{c}
  \begin{tikzpicture}[xscale=.5,yscale=.9]
     \draw[dashed] (-2,-1) to (2, -1);
     \draw[dashed] (-2,1) to (2,1);
     \draw[blue] (-.7,-1) to (-.7,1);
     \node at (0,-.8) {$\dots$};
          \node at (0,.8) {$\dots$};
     \draw[blue] (.7,-1) to (.7,1);
     \draw[fill=white] (-1.4,-.6) rectangle (1.4,.6);
     \node at (0,0) {$\JWalttwo_{s,t}$};
     \draw[red] (1.2,.6) to[out=90, in=180] (1.5,.8) to[out=0,in=90] (1.8,.6) to (1.8,-1);
     \draw[red](-1.2,-.6) to[out=-90, in=0] (-1.5,-.8) to[out=180,in=-90] (-1.8,-.6) to (-1.8,1);
   \end{tikzpicture}
             \end{array}
             =
             \lambda_{s,t}
                 \begin{array}{c}
  \begin{tikzpicture}[xscale=.5,yscale=.9]
     \draw[dashed] (-1.6,-1) to (1.6, -1);
     \draw[dashed] (-1.6,1) to (1.6,1);
     \draw[blue] (-1.2,-1) to (-1.2,0);
     \draw[red] (-.7,-1) to (-.7,1);
     \node at (0,-.8) {$\dots$};
          \node at (0,.8) {$\dots$};
     \draw[red] (.7,-1) to (.7,1);
     \draw[blue] (1.2,0) to (1.2,1);
     \draw[fill=white] (-1.4,-.6) rectangle (1.4,.6);
     \node at (0,0) {$\JWalttwo_{t,s}$};
\end{tikzpicture}
    \end{array}
    \end{equation}
    for some scalar $\lambda_{s,t}$. One can prove that $\lambda_{s,t} = \lambda_{t,s}^{-1}$, so that
        \begin{equation}
      \label{eq:5.5}
           \begin{array}{c}
  \begin{tikzpicture}[xscale=.5,yscale=.9]
     \draw[dashed] (-2.5,-1.2) to (2.5, -1.2);
     \draw[dashed] (-2.5,1.2) to (2.5,1.2);
     \draw[blue] (-.7,0) to (-.7,1.2);
     \node at (0,-.8) {$\dots$};
          \node at (0,.8) {$\dots$};
     \draw[blue] (.7,-1.2) to (.7,0);
     \draw[fill=white] (-1.4,-.6) rectangle (1.4,.6);
     \node at (0,0) {$\JWalttwo_{s,t}$};
     \draw[red] (1.2,.6) to[out=90, in=180] (1.5,.8) to[out=0,in=90] (1.8,.6) to (1.8,-1.2);
     \draw[blue] (.7,.6) to[out=90, in=180] (1.5,1) to[out=0,in=90] (2.2,.6) to (2.2,-1.2);
     \draw[red](-1.2,-.6) to[out=-90, in=0] (-1.5,-.8) to[out=180,in=-90] (-1.8,-.6) to (-1.8,1.2); 
     \draw[blue] (-.7,-.6) to[out=-90, in=0] (-1.5,-1) to[out=180,in=-90] (-2.2,-.6) to (-2.2,1.2);
   \end{tikzpicture}
             \end{array}
             =
                 \begin{array}{c}
  \begin{tikzpicture}[xscale=.5,yscale=1.1]
     \draw[dashed] (-1.6,-1) to (1.6, -1);
     \draw[dashed] (-1.6,1) to (1.6,1);
     \draw[red] (-1.2,-1) to (-1.2,0);
     \draw[blue] (-.7,-1) to (-.7,1);
     \node at (0,-.8) {$\dots$};
          \node at (0,.8) {$\dots$};
     \draw[blue] (.7,-1) to (.7,1);
     \draw[red] (1.2,0) to (1.2,1);
     \draw[fill=white] (-1.4,-.6) rectangle (1.4,.6);
     \node at (0,0) {$\JWalttwo_{s,t}$};
\end{tikzpicture}
    \end{array}
  \end{equation}
  In this case, we say that
  the Jones-Wenzl projector is \emph{invariant under rotation by two
    strands}. See \S\ref{subsec:rotinvce} for the proofs.

One can compute that whenever \eqref{eq:5} holds for some scalar $\lambda_{s,t}$, that scalar must be a particular invertible two-colored quantum number, see Lemma
\ref{lem:rotinvce} for details. By definition, $\lambda_{s,t} = 1$ when the realization is \emph{balanced}, c.f. Definition \ref{defn:balanced}. In the balanced case, we say that the Jones-Wenzl projector is \emph{invariant under rotation by one strand with color swap}. Rotation
invariance of the Jones-Wenzl projector is a crucial compatibility between \eqref{Eq:JWreln} and the cyclicity of the $2m$-valent vertex. For the rest of this chapter we assume the realization is balanced. We treat the unbalanced case in \S\ref{sec-unbalanced}.

Combining rotation invariance with death by pitchfork, we also see that $\JWalttwo_{s,t}$ is killed by placing pitchforks on the ``side,'' for example
\begin{equation}
  \begin{array}{c}
    \begin{tikzpicture}[xscale=.7,yscale=.9]
      \begin{scope}
          \clip (.3,0) circle (2.2);
     % \draw[dashed] (-1.6,-1) to (1.6, -1);
     % \draw[dashed] (-1.6,1) to (1.6,1);
     \draw[red] (-1.2,-3) to (-1.2,0);
     \draw[blue] (-.7,-3) to (-.7,3);
     \draw[red] (.2,-3) to (.2,3);
     \node at (-.2,-.8) {$\dots$};
          \node at (-.2,.8) {$\dots$};
     \draw[blue] (.7,-.6) to[out=-90,in=180] (1.4,-1.2) to[out=0,in=-90] (2,0) to[out=90,in=0] (1.4,1.2) to[out=180,in=90] (.7,.6);
     \draw[red] (1.2,0) to (1.2,.8);
     \draw[fill=white] (-1.4,-.6) rectangle (1.4,.6);
     \node at (0,0) {$\JWalttwo_{s,t}$};
           \node[circle,fill,draw,inner
           sep=0mm,minimum size=1mm,color=red] at (1.2,.8) {};
           \draw[blue] (2,0) to (4,0);
         \end{scope}
         \draw[dashed] (.3,0) circle (2.2);
\end{tikzpicture}
    \end{array}
    =
    0
  \end{equation}
  
Another crucial property of this special Jones-Wenzl projector is that
  \begin{equation}
    \label{eq:6}
        \begin{array}{c}
  \begin{tikzpicture}[xscale=.5,yscale=1]
     \draw[dashed] (-1.6,-1) to (1.6, -1);
     \draw[dashed] (-1.6, 1) to (1.6, 1);
     \draw[red] (-1.2,-.8) to (-1.2,.8);
     \draw[blue] (-.7,-.8) to (-.7, .8);
     \node at (0,-.7) {$\dots$};
          \node at (0,.7) {$\dots$};
     \draw[blue] (.7,-.8) to (.7, .8);
     \draw[red] (1.2,-.8) to (1.2, .8);
     \draw[fill=white] (-1.4,-.6) rectangle (1.4,.6);
     \node at (0,0) {$\JWalt_{s,t}$};
      \node[circle,fill,draw,inner
      sep=0mm,minimum size=1mm,color=red] at (-1.2,.8) {};
      \node[circle,fill,draw,inner
      sep=0mm,minimum size=1mm,color=red] at (1.2,.8) {};
      \node[circle,fill,draw,inner
      sep=0mm,minimum size=1mm,color=blue] at (-.7,.8) {};
            \node[circle,fill,draw,inner
            sep=0mm,minimum size=1mm,color=blue] at (.7,.8) {};
                  \node[circle,fill,draw,inner
      sep=0mm,minimum size=1mm,color=blue] at (-.7,-.8) {};
            \node[circle,fill,draw,inner
      sep=0mm,minimum size=1mm,color=blue] at (.7,-.8) {};
      \node[circle,fill,draw,inner
      sep=0mm,minimum size=1mm,color=red] at (1.2,-.8) {};
            \node[circle,fill,draw,inner
      sep=0mm,minimum size=1mm,color=red] at (-1.2,-.8) {};
    \end{tikzpicture}
        \end{array}
        =
        \begin{array}{c}
  \begin{tikzpicture}[xscale=.5,yscale=1]
     \draw[dashed] (-1.6,-1) to (1.6, -1);
     \draw[dashed] (-1.6,1) to (1.6,1);
     \draw[red] (-1.2,-.8) to (-1.2,0);
     \draw[blue] (-.7,-.8) to (-.7, .8);
     \node at (0,-.7) {$\dots$};
          \node at (0,.7) {$\dots$};
     \draw[blue] (.7,-.8) to (.7, .8);
     \draw[red] (1.2,0) to (1.2, .8);
     \draw[fill=white] (-1.4,-.6) rectangle (1.4,.6);
     \node at (0,0) {$\JWalttwo_{s,t}$};
                       \node[circle,fill,draw,inner
      sep=0mm,minimum size=1mm,color=blue] at (-.7,-.8) {};
            \node[circle,fill,draw,inner
      sep=0mm,minimum size=1mm,color=blue] at (.7,-.8) {};
      \node[circle,fill,draw,inner
      sep=0mm,minimum size=1mm,color=red] at (-1.2,-.8) {};
                             \node[circle,fill,draw,inner
      sep=0mm,minimum size=1mm,color=blue] at (-.7,.8) {};
            \node[circle,fill,draw,inner
      sep=0mm,minimum size=1mm,color=blue] at (.7,.8) {};
      \node[circle,fill,draw,inner
      sep=0mm,minimum size=1mm,color=red] at (1.2,.8) {};
\end{tikzpicture}
        \end{array}
        = \pi_{s,t}
      \end{equation}
(the notation $\pi_{s,t}$ was introduced in \S\ref{subsec-pos-roots}). This was proven in \cite[Corollary 4.15]{EDC} under certain assumptions, and we prove it more generally in Lemma \ref{lem:evaluation} below.

\begin{remark} \label{rmk:onecolorforalldots} The proof of this result comes from counting the red and the blue regions in each crossingless matching. After
applying $\Sigma$ and composing with many dots as above, each region is transformed to a barbell, and evaluates to the polynomial $\alpha_{{\color{red} s}}$ or $\alpha_{{\color{blue} t}}$. In
particular, the relation \eqref{eq:6} is a consequence only of the one color relations. Death by pitchfork is also a consequence only of the one-color relations. \end{remark}

Finally, to apply the functor $\Lambda$ most efficiently to Jones-Wenzl projectors, we will twist them using adjunction to consider maps whose target is the monoidal identity. Let
\begin{equation} JW_{s,t} =            \begin{array}{c}
  \begin{tikzpicture}[xscale=.5,yscale=.9]
     \draw[dashed] (-1.8,-1.2) to (3.1, -1.2);
     \draw[dashed] (-1.8,1.6) to (3.1,1.6);
     \node at (0,-.8) {$\dots$};
          \node at (0,.8) {$\dots$};
          \draw[blue] (.7,-1.2) to (.7,0);
          \draw[red] (-1.2,0) to (-1.2,-1.2);
          \draw[blue] (-.7,0) to (-.7,-1.2);
     \draw[fill=white] (-1.4,-.6) rectangle (1.2,.6);
     \node at (0,0) {$\JWalttwo_{s,t}$};
          \draw[blue] (-.7,.6) to[out=90,in=180] (1.5,1.3) to[out=0,in=90] (2.8,.6) to (2.8,-1.2);
     \draw[red] (.9,.6) to[out=90, in=180] (1.2,.8) to[out=0,in=90] (1.5,.6) to (1.5,-1.2);
     \draw[blue] (.4,.6) to[out=90, in=180] (1.2,1) to[out=0,in=90] (1.9,.6) to (1.9,-1.2);
     \node at (2.4,0) {$\dots$};
     % \draw[red](-1.2,-.6) to[out=-90, in=0] (-1.5,-.8) to[out=180,in=-90] (-1.8,-.6) to (-1.8,1.2); 
     % \draw[blue] (-.7,-.6) to[out=-90, in=0] (-1.5,-1) to[out=180,in=-90] (-2.2,-.6) to (-2.2,1.2);
   \end{tikzpicture}
             \end{array}
       \end{equation}
       and
       \begin{equation} \label{eq:JWprimedefn}
  \begin{array}{c}
  \tikz[xscale=0.45,yscale=0.45]{
      \draw[dashed] (-1,-2.5) to (4.5, -2.5); \draw[dashed] (-1,2.5) to (4.5,2.5);
\draw (0,1) rectangle (4,-1);
  \node at (2,0) {$JW_{s,t}'$};
  \draw[red] (.5,-1) to (.5,-2.5);
%\draw[red] (.5,-1.5) to[out=180,in=0] (0,-1.5) to[out=180,in=-90] (-1,0) to[out=90,in=180] (1.5,2.5) to[out=0,in=90] (3.5,0) to[out=-90,in=90] (3.5,-2.5);
\draw[blue] (1,-1) to (1,-2.5);
  \draw[red] (1.5,-1) to (1.5,-2.5);
\draw[blue] (2,-1) to (2,-2.5);
%  \draw[red] (2.5,-1) to (2.5,-2.5);
%\node[circle,fill,draw,inner sep=0mm,minimum size=1mm,color=blue] at (1,-2) {};
    \node at (2.8,-1.75) {\small $\dots$};
      \draw[red] (3.5,-1) to (3.5,-2.5);
    } \end{array}
  :=
\begin{array}{c}
  \tikz[xscale=0.45,yscale=0.45]{
      \draw[dashed] (-1,-2.5) to (4.5, -2.5); \draw[dashed] (-1,2.5) to (4.5,2.5);
\draw (0,1) rectangle (3,-1);
  \node at (1.5,0) {$JW_{s,t}$};
  \draw[red] (.5,-1) to (.5,-2.5);
\draw[red] (.5,-1.5) to[out=180,in=0] (0,-1.5) to[out=180,in=-90] (-1,0) to[out=90,in=180] (1.5,2) to[out=0,in=90] (3.5,0) to[out=-90,in=90] (3.5,-2.5);
\draw[blue] (1,-1) to (1,-2.5);
  \draw[red] (1.5,-1) to (1.5,-2.5);
\draw[blue] (2,-1) to (2,-2.5);
%  \draw[red] (2.5,-1) to (2.5,-2.5);
%\node[circle,fill,draw,inner sep=0mm,minimum size=1mm,color=blue] at (1,-2) {};
\node at (2.8,-1.75) {\small $\dots$};
  } \end{array}
=
  \begin{array}{c}
  \begin{tikzpicture}[xscale=.5,yscale=.7]
     \draw[dashed] (-1.8,-1.2) to (3.6, -1.2);
     \draw[dashed] (-1.8,2) to (3.6,2);
     \node at (0,-.8) {$\dots$};
          \node at (0,.8) {$\dots$};
          \draw[red] (.9,-.6) to (.9,-.9);
          \node[circle,fill,draw,inner sep=0mm,minimum size=1mm,color=red] at (.9,-.9) {};
          \draw[blue] (.4,-1.2) to (.4,0);
          \draw[red] (-1.2,0) to (-1.2,-1.2);
          \draw[blue] (-.7,0) to (-.7,-1.2);
     \draw[fill=white] (-1.4,-.6) rectangle (1.2,.6);
     \node at (0,0) {$\JWalt_{s,t}$};
     \draw[red] (-1.2,.6) to[out=90,in=180] (1.5,1.7) to[out=0,in=90] (3.3,.6) to (3.3,-1.2);
     \draw[blue] (-.7,.6) to[out=90,in=180] (1.5,1.4) to[out=0,in=90] (2.8,.6) to (2.8,-1.2);
     \draw[red] (.9,.6) to[out=90, in=180] (1.2,.8) to[out=0,in=90] (1.5,.6) to (1.5,-1.2);
     \draw[blue] (.4,.6) to[out=90, in=180] (1.2,1) to[out=0,in=90] (1.9,.6) to (1.9,-1.2);
     \node at (2.4,0) {$\dots$};
     % \draw[red](-1.2,-.6) to[out=-90, in=0] (-1.5,-.8) to[out=180,in=-90] (-1.8,-.6) to (-1.8,1.2); 
     % \draw[blue] (-.7,-.6) to[out=-90, in=0] (-1.5,-1) to[out=180,in=-90] (-2.2,-.6) to (-2.2,1.2);
   \end{tikzpicture}
             \end{array}
           \end{equation}
          (We leave
           it as an exercise to the reader to check the equality of
           the second two diagrams, using \eqref{unitaroundJW}.)
Various features of $JW_{s,t}$ and $JW'_{s,t}$ follow immediately from the properties of $\JWalt_{s,t}$ and $\JWalttwo_{s,t}$ listed above. For example, they are both killed by all pitchforks on bottom, and
\begin{equation} \label{eq:alldots}
\begin{array}{c}
 \tikz[xscale=0.45,yscale=0.45]{
\draw (0,1) rectangle (3,-1);
\node at (1.5,0) {$JW'_{s,t}$};
\draw[red] (.5,-1) to (.5,-2);
\draw[blue] (1,-1) to (1,-2);
\draw[red] (2.5,-1) to (2.5,-2);
\node[circle,fill,draw,inner sep=0mm,minimum size=1mm,color=red] at (.5,-2) {};
\node[circle,fill,draw,inner sep=0mm,minimum size=1mm,color=blue] at (1,-2) {};
\node[circle,fill,draw,inner sep=0mm,minimum size=1mm,color=red] at (2.5,-2) {};
\node at (1.8,-1.75) {$\dots$};
} \end{array}
= \pi_{s,t}.
\end{equation}
As another example, $JW_{s,t}$ is invariant under rotation by two strands.
 \begin{equation} \label{Eq:rotrot}
\begin{array}{c}
  \tikz[xscale=0.45,yscale=0.45]{
      \draw[dashed] (-1,-2.5) to (4.5, -2.5); \draw[dashed] (-1,3.5) to (4.5,3.5);
\draw (0,1) rectangle (3,-1);
\node at (1.5,0) {$JW_{s,t}$};
\draw[red] (.5,-1) to[out=-90,in=0] (0,-1.5) to[out=180,in=-90] (-1,0) to[out=90,in=180] (1.5,2.5) to[out=0,in=90] (3.5,0) to[out=-90,in=90] (3.5,-2.5);
\draw[blue] (1,-1) to[out=-90,in=0] (0,-2) to[out=180,in=-90] (-1.5,0) to[out=90,in=180] (1.5,3) to[out=0,in=90] (4,0) to[out=-90,in=90] (4,-2.5);
  \draw[red] (1.5,-1) to (1.5,-2.5);
\draw[blue] (2,-1) to (2,-2.5);
%  \draw[red] (2.5,-1) to (2.5,-2.5);
%\node[circle,fill,draw,inner sep=0mm,minimum size=1mm,color=blue] at (1,-2) {};
\node at (2.8,-1.75) {\small $\dots$};
  } \end{array}
=
\begin{array}{c}
  \tikz[xscale=0.45,yscale=0.45]{
      \draw[dashed] (-1,-2.5) to (4.5, -2.5); \draw[dashed] (-1,3.5) to (4.5,3.5);
\draw (0,1) rectangle (3,-1);
\node at (1.5,0) {$JW_{s,t}$};
\draw[red] (.5,-1) to (.5,-2.5);
\draw[blue] (1,-1) to (1,-2.5);
  \draw[red] (2,-1) to (2,-2.5);
\draw[blue] (2.5,-1) to (2.5,-2.5);
%  \draw[red] (2.5,-1) to (2.5,-2.5);
%\node[circle,fill,draw,inner sep=0mm,minimum size=1mm,color=blue] at (1,-2) {};
\node at (1.55,-1.75) {\small $\dots$};
  } \end{array}
\end{equation}

\begin{remark} One nice feature of $JW_{s,t}$ and $JW'_{s,t}$ is that their coloration does not depend on the parity of $m_{st}$ (as opposed to $\JWalt_{s,t}$, see Remark \ref{remark:colorparity}). \end{remark}

\subsection{Checking the two-color relations}

There are a number of versions of the Jones-Wenzl relation \eqref{Eq:JWreln} we will use, which are all equivalent only using cyclicity. Rotating one red strand of \eqref{Eq:JWreln} around
  (analogously to  \eqref{Eq:rotrot}) we obtain
\begin{equation} \label{eq:JWlem} 
\begin{array}{c}
  \tikz[xscale=0.45,yscale=0.45]{
        \draw[dashed] (-1,-2.5) to (4.5, -2.5); \draw[dashed] (-1,2) to (4.5,2);
\draw (0,1) rectangle (3.5,-1);
\node at (1.75,0) {$E_{s,t}$};
\draw[red] (.5,-1) to (.5,-2.5);
\draw[blue] (1,-1) to (1,-2.5);
\draw[red] (2.5,-1) to (2.5,-2.5);
\draw[blue] (3,-1) to (3,-2);
\node[circle,fill,draw,inner sep=0mm,minimum size=1mm,color=blue] at (3,-2) {};
\node at (1.8,-1.75) {$\dots$};
} \end{array}
=
\begin{array}{c}
  \tikz[xscale=0.45,yscale=0.45]{
        \draw[dashed] (-1,-2.5) to (4.5, -2.5); \draw[dashed] (-1,2) to (4.5,2);
\draw (0,1) rectangle (3,-1);
\node at (1.5,0) {$JW'_{s,t}$};
\draw[red] (.5,-1) to (.5,-2.5);
\draw[blue] (1,-1) to (1,-2.5);
\draw[red] (2.5,-1) to (2.5,-2.5);
\node at (1.8,-1.75) {$\dots$};
} \end{array},
\end{equation}
which is the version we check in the following lemma.

\begin{lem} \label{lem:JW} $\Lambda$ preserves the relation \eqref{eq:JWlem}. \end{lem}

\begin{proof}
  For any subexpression $\un{e}$ of $(s,t,\dots, s)$, $[e]$ is killed
  by both sides of the equation unless $\un{e}$ is a subexpression for
  the identity. It follows from the definition of $\Lambda$
  (more precisely, from \eqref{eq:ddot} and \eqref{eq:E}) that
  for any such subexpression, under the morphism of the
  left-hand side of the lemma we have
\[
[\un{e}] \mapsto \pi_{s,t}[\emptyset].
\]
So we are done if we can verify the same statement on the right hand
side.

Under the RHS of \eqref{eq:JWlem}, let us write 
\[
[\un{e}] \mapsto g_{\un{e}}[\emptyset]
\]
for some $g_{\un{e}} \in Q$. We wish to show that $g_{\un{e}} = \pi_{s,t}$ for all subexpressions $\un{e}$ for the identity.

First we check this when $\un{e} = (0,0,\dots,0)$. Using \eqref{eq:ddot}, the LHS of \eqref{eq:alldots} will (after applying $\Lambda$) send $[\emptyset]$ to $g_{\un{e}} [\emptyset]$ as well. Clearly the RHS of \eqref{eq:alldots} will send $[\emptyset]$ to $\pi_{s,t} [\emptyset]$. Thus $g_{\un{e}} = \pi_{s,t}$ if \eqref{eq:alldots} is preserved by $\Lambda$. Now $JW'_{s,t}$ is built from univalent and trivalent vertices. Moreover, \eqref{eq:alldots} is the consequence of the one-color relations, see Remark \ref{rmk:onecolorforalldots}. Since we have already checked that $\Lambda$ preserves the one-color relations, \eqref{eq:alldots} is already proven to hold after applying $\Lambda$.

Now, let $P$ (for pitchfork) denote a morphism of the form
\[
  \begin{array}{c}
    \begin{tikzpicture}[scale=0.5]
            \draw[dashed] (1,-1.5) to (5.5, -1.5); \draw[dashed] (1,0) to (5.5,0);
% bottom dots
      \foreach \x in {1.5,5} \node at (\x,-0.75) {$\dots$};
% % pitchfork:
       \def\x{3.25}; \def\y{-1};
       \draw[red!50!white] (2.5,0) to[out=-90,in=180] (\x,\y) to[out=0,in=-90] (4,0);
       \draw[red!50!white] (\x,\y) -- (\x,-1.5);
       \draw[blue] (\x,0) -- (\x,-0.5);
 \node[circle,fill,draw,inner
      sep=0mm,minimum size=1mm,color=blue] at (\x,-0.5) {};
    \end{tikzpicture}
  \end{array}
\text{or}
  \begin{array}{c}
    \begin{tikzpicture}[scale=0.5]
                  \draw[dashed] (1,-1.5) to (5.5, -1.5); \draw[dashed] (1,0) to (5.5,0);
% bottom dots
      \foreach \x in {1.5,5} \node at (\x,-0.75) {$\dots$};
% % pitchfork:
       \def\x{3.25}; \def\y{-1};
       \draw[red!50!white] (2.5,0) to[out=-90,in=180] (\x,\y) to[out=0,in=-90] (4,0);
       \draw[white] (\x,\y) -- (\x,-1.5);
       \draw[blue] (\x,0) -- (\x,-0.5);
 \node[circle,fill,draw,inner
      sep=0mm,minimum size=1mm,color=blue] at (\x,-0.5) {};
    \end{tikzpicture}
  \end{array},
\]
perhaps with colors interchanged, and with additional vertical strands (identity maps) to the left and right. Death by pitchfork implies that 
\begin{equation}
  \label{eq:pitchforkagain}
\begin{array}{c}
 \tikz[xscale=0.45,yscale=0.45]{
\draw (0,1) rectangle (3,-1);
\node at (1.5,0) {$JW'_{s,t}$};
\draw[red] (.5,-1) to (.5,-1.5);
\draw[blue] (1,-1) to (1,-1.5);
\draw[red] (2.5,-1) to (2.5,-1.5);
\node at (1.8,-1.25) {$\dots$};
\draw (0,-1.5) rectangle (3,-3.5);
\node at (1.5,-2.5) {$P$};
} \end{array} = 0.
\end{equation}
Death by pitchfork is also a consequence of the one-color relations,
so we already know that $\Lambda$ kills the LHS of
\eqref{eq:pitchforkagain}.

Let $\un{e} = ( \un{e}_1, 1, 0, 0, \un{e}_2)$ and $\un{e}' = ( \un{e}_1, 0, 0, 1, \un{e}_2)$ be subexpressions for the identity which differ only in the three strands meeting the pitchfork $P$, which is a pitchfork of the first kind pictured above. Let $\un{f} = (\un{e}_1, 1,\un{e}_2)$ be the corresponding subexpression of the source of $P$. Let $w$ be the endpoint of the partial subexpression $\un{e}_1$. Then using \eqref{eq:ddot} and \eqref{eq:peace} we see that \eqref{eq:pitchfork} sends
\[ [\un{f}] \mapsto w(\frac{1}{\alpha_s}) [\un{e}'] - w(\frac{1}{\alpha_s}) [\un{e}] \mapsto \frac{1}{w(\alpha_s)}(g_{\un{e}'} - g_{\un{e}}) [\emptyset]. \]
The first arrow is the action of $P$, and the second arrow the action of $JW_{s,t}$. 
Since the result must be zero, and $R$ is a domain, we deduce that
\[ g_{\un{e}} = g_{\un{e}'}. \]

Using very similar arguments, we deduce
that if two subexpressions $\un{e}$ and $\un{e}'$ are related by one
of the following moves
\begin{gather*}
( \dots, 1, 0, 0, \dots ) \leftrightarrow ( \dots, 0, 0, 1, \dots ) \\
( \dots, 1, 0, 1, \dots ) \leftrightarrow ( \dots, 0, 0, 0, \dots )
\end{gather*}
then their images $g_{\un{e}} [\emptyset]$ and $g_{\un{e}'} [\emptyset]$ under $\Lambda(JW_{s,t})$ agree. One checks easily
that any
subexpression of $(s,t,\dots, s)$ for the identity may be 
transformed via a sequence of these moves to the
subexpression $(0,0,\dots, 0)$, whence the result. \end{proof}

\begin{remark}
The proof above used the ``pre-composition trick'': to understand the images of various subexpressions $\un{e}$, we pre-compose with a map which isolates these subexpressions. For example, the image of $[00\cdots 0]$ can be isolated using precomposition with the all-dots map, as in \eqref{eq:alldots}. Precomposing with pitchforks isolates pairs of subexpressions.  
\end{remark}

The following is 2-color associativity:
\begin{equation} \label{eq:2assocequality}
\begin{array}{c}
  \tikz[scale=0.8]{
              \draw[dashed] (1.5,-2) to (4.5, -2); \draw[dashed] (1.5,1) to (4.5,1);
\draw[color=red] (2.1,-1) to[out=90,in=-150] (3,0) to (3,1); \draw[color=red] (3.9,-2) to[out=90,in=-30] (3,0);
\draw[red] (2.5,-2) to (2.1,-1) to (1.5,-2);
\draw[color=blue] (2.1,1) to[out=-90,in=150] (3,0) to (3,-2);
  \draw[color=blue] (3.9,1) to[out=-90,in=30] (3,0);}
\end{array}
=
\begin{array}{c}
  \tikz[scale=0.7]{
                \draw[dashed] (1,-1) to (4.5, -1); \draw[dashed] (1,2.5) to (4.5, 2.5);
\draw[color=red] (2,-1) to[out=90,in=-150] (3,0);% to (3,1);
\draw[color=red] (3.9,-1) to[out=90,in=-30] (3,0);
\draw[color=blue] (2.1,1.5) to[out=-90,in=150] (3,0) to (3,-1);
  \draw[color=blue] (3.9,1) to[out=-90,in=30] (3,0);
\draw[blue] (3.9,1) to[out=90,in=-30] (3,2);
\draw[blue] (2.1,1.5) to (3,2) to (3,2.5);
\draw[blue] (2.1,1.5) to[out=150,in=-90] (1.2,2.5);
\draw[red] (3,0) to[out=90,in=-30] (2.1,1.5);
\draw[red] (1.5,-1) to[out=90,in=-150] (2.1,1.5);
\draw[red] (2.1,1.5) to (2.1,2.5);
}
\end{array}
\end{equation}
(We draw the relation for $m_{st} = 3$. The following lemma
and its proof holds in general.)

\begin{lem} $\Lambda$ preserves \eqref{eq:2assocequality}. \label{lem:2assoc}
\end{lem}

The proof the lemma will be via the ``precomposition trick''. However
in order to apply it we need the following:

\begin{lem} \label{lem:12last}
  Precomposing \eqref{eq:2assocequality} with a dot on either
  the first, second or last strand yields a relation which is
  already a consequence of the other relations: the one-color,
  cyclicity, and Jones-Wenzl relations.
\end{lem}

\begin{remark}
  In fact, the lemma remains true for any dot placed on the lower
  strands. However the proof is more involved, and we will only need
  the easier version stated here.
\end{remark}

\begin{proof}
First we claim that 
placing a dot below a $2m$-valent vertex gives one term where ``the
dot pulls straight through the $2m$-valent vertex,'' and all the
remaining terms have pitchforks. That is,
\begin{equation} \label{mosttermsP}
  \begin{array}{c}
    \begin{tikzpicture}[scale=0.4]
      \draw[dashed] (-2,-1.5) to (2, -1.5); \draw[dashed] (-2,1.5) to (2,1.5);
      \node[circle,fill,draw,inner sep=0mm,minimum size=1mm,color=red] at (-1,-1.2) {};
  \draw[color=red] (-1,-1.2) to (-1,-.7);
  \draw[color=blue] (0,-1.5) to (0,-.7);
\draw[color=red] (1,-1.5) to (1,-.7);
\draw[color=blue] (-1,1.5) to (-1,.7);
  \draw[color=red] (0,1.5) to (0, .7);
  \draw[color=blue] (1,1.5) to (1,.7);
  \draw (-1.5,-.7) rectangle (1.5,.7);
   \node at (0,0) {$G_{s,t}$};
\end{tikzpicture}
  \end{array}
  =
    \begin{array}{c}
      \begin{tikzpicture}[scale=0.4]
              \draw[dashed] (-2,-1.5) to (2, -1.5); \draw[dashed] (-2,1.5) to (2,1.5);
%   \draw[color=red] (-1,-1.2) to (-1,-.7);
%   \draw[color=blue] (0,-1.5) to (0,-.7);
% \draw[color=red] (1,-1.5) to (1,-.7);
\draw[color=blue] (-1,1.5) to (-1,-1.5);
  \draw[color=red] (0,1.5) to (0, -1.5);
  \draw[color=blue] (1,1.5) to (1,.7);
%  \draw (-1.5,.7) rectangle (1.5,.7);
      \node[circle,fill,draw,inner sep=0mm,minimum size=1mm,color=blue] at (1,.7) {};
    \end{tikzpicture}
  \end{array}
      +
          \begin{array}{c}
            \begin{tikzpicture}[scale=0.4]
                    \draw[dashed] (-2,-1.5) to (2, -1.5); \draw[dashed] (-2,1.5) to (2,1.5);
%   \draw[color=red] (-1,-1.2) to (-1,-.7);
%   \draw[color=blue] (0,-1.5) to (0,-.7);
% \draw[color=red] (1,-1.5) to (1,-.7);
\draw[color=blue] (-1,1.5) to (-1,-1.5);
  \draw[color=red] (0,1.5) to (0, -1.5);
  \draw[color=blue] (1,1.5) to (1,.7);
  \draw[fill=white] (-1.5,-.7) rectangle (1.5,.7);
  \node at (0,0) {$P$};
%      \node[circle,fill,draw,inner sep=0mm,minimum size=1mm,color=blue] at (1,.7) {};
    \end{tikzpicture}
  \end{array}.
\end{equation}
Here $P$ denotes a linear combination of diagrams, and each diagram
has a pitchfork on top. Moreover, if the rightmost blue strand is
twisted down like so,
\begin{equation} \label{mosttermsPdown}
  \begin{array}{c}
    \begin{tikzpicture}[scale=0.4]
      \draw[dashed] (-2,-1.5) to (2.5, -1.5); \draw[dashed] (-2,1.5) to (2.5,1.5);
      \node[circle,fill,draw,inner sep=0mm,minimum size=1mm,color=red] at (-1,-1.2) {};
  \draw[color=red] (-1,-1.2) to (-1,-.7);
  \draw[color=blue] (0,-1.5) to (0,-.7);
\draw[color=red] (1,-1.5) to (1,-.7);
\draw[color=blue] (-1,1.5) to (-1,.7);
  \draw[color=red] (0,1.5) to (0, .7);
  \draw[color=blue] (1,.7) to[out=90,in=180] (1.4,1.2) to[out=0,in=90] (1.8,-1.5);
  \draw (-1.5,-.7) rectangle (1.5,.7);
   \node at (0,0) {$G_{s,t}$};
\end{tikzpicture}
  \end{array}
  =
    \begin{array}{c}
      \begin{tikzpicture}[xscale=0.4,yscale=-.4]
              \draw[dashed] (-2,-1.5) to (2, -1.5); \draw[dashed] (-2,1.5) to (2,1.5);
%   \draw[color=red] (-1,-1.2) to (-1,-.7);
%   \draw[color=blue] (0,-1.5) to (0,-.7);
% \draw[color=red] (1,-1.5) to (1,-.7);
\draw[color=blue] (-1,1.5) to (-1,-1.5);
  \draw[color=red] (0,1.5) to (0, -1.5);
  \draw[color=blue] (1,1.5) to (1,.7);
%  \draw (-1.5,.7) rectangle (1.5,.7);
      \node[circle,fill,draw,inner sep=0mm,minimum size=1mm,color=blue] at (1,.7) {};
    \end{tikzpicture}
  \end{array}
      +
          \begin{array}{c}
            \begin{tikzpicture}[xscale=0.4,yscale=-.4]
                    \draw[dashed] (-2,-1.5) to (2, -1.5); \draw[dashed] (-2,1.5) to (2,1.5);
%   \draw[color=red] (-1,-1.2) to (-1,-.7);
%   \draw[color=blue] (0,-1.5) to (0,-.7);
% \draw[color=red] (1,-1.5) to (1,-.7);
\draw[color=blue] (-1,1.5) to (-1,-1.5);
  \draw[color=red] (0,1.5) to (0, -1.5);
  \draw[color=blue] (1,1.5) to (1,.7);
  \draw[fill=white] (-1.5,-.7) rectangle (1.5,.7);
  \node at (0,0) {$P'$};
%      \node[circle,fill,draw,inner sep=0mm,minimum size=1mm,color=blue] at (1,.7) {};
    \end{tikzpicture}
  \end{array}.
\end{equation}
then each diagram in $P'$ has a pitchfork on the bottom.
The formulas \eqref{mosttermsP} and \eqref{mosttermsPdown} are very
useful in combination with ``death by pitchfork'' for diagrams with
multiple $2m$-valent vertices.

In fact, \eqref{mosttermsP} and \eqref{mosttermsPdown} can both be deduced from \eqref{eq:3}; the term where the dot pulls through
corresponds to the identity diagram in the original Jones-Wenzl
projector. To get from \eqref{eq:3} to \eqref{mosttermsP} and
\eqref{mosttermsPdown} is a relatively straightforward exercise in diagrammatics and adjunction. First, absorb the dot into the $2m$-valent vertex using \eqref{eq:JWlem}, and define $JW'_{s,t}$ using the last diagram in \eqref{eq:JWprimedefn}. We leave the rest to the reader.

Now precompose the RHS of \eqref{eq:2assocequality} with a dot morphism on the first, second, or last strand. Resolve the $2m$-valent vertex which meets this dot by
\eqref{mosttermsP}. Because $2m$-valent vertices satisfy death by pitchfork, every diagram in $P$ is orthogonal to the other $2m$-valent vertex, so only the term where the dot pulls straight through will survive. It is now
relatively easy to verify that precomposing \eqref{eq:2assocequality} with a dot morphism on the first, second, or last strand yields a relation which follows from the previous
relations; we leave this as an exercise to the reader.
\end{proof}

\begin{remark} The coloration in \eqref{mosttermsP} and \eqref{mosttermsPdown} depends on the parity of $m_{st}$, and they have been drawn above as though $m_{st}$ were odd. If $m_{st}$ were even then the dot on the RHS of \eqref{mosttermsP} would be red rather than blue. \end{remark}

\begin{proof}[Proof of Lemma \ref{lem:2assoc}] Suppose that $\un{e}$
  is a subexpression of $(s, s, t, s,
  \dots)$. We check as usual that $[\un{e}]$ has the same image under
  the morphisms presented on the left-hand and right-hand sides, for
  all choices of $\un{e}$.
  
First we check the case $\un{e} = (1,1,\dots,1)$. We will use \eqref{eq:110} and \eqref{eq:111}. On the left hand side its image is
\[
[111 \dots 11] \mapsto [01 \dots 11] \mapsto [11 \dots 10].
\]
On the right hand side its image is
\[
[11 \dots 111] \mapsto [11 \dots 111] \mapsto [11 \dots 111] \mapsto
[111 \dots 10].
\]

We claim that all other cases can be solved using the precomposition
trick, and Lemma \ref{lem:12last}. Indeed, by Lemma \ref{lem:12last}
and the fact that we have already checked that $\Lambda$ satisfies the
one-color, cyclicity, and Jones-Wenzl relations, we deduce that if
$\un{e}$ has a zero in its first, second or last position, then the
images under both sides of \eqref{eq:2assocequality} agree.

Finally, note that both sides of \eqref{eq:2assocequality} are killed
by most pitchforks on bottom. Each pitchfork has a dot morphism in the
middle of its three strands, and this central dot can appear on any
strand but the first, second, or last. As in the proof of Lemma
\ref{lem:JW}, this allows us to equate the images of various
subexpressions under the transformations
\begin{gather*}
( \dots, 1, 0, 0, \dots ) \leftrightarrow ( \dots, 0, 0, 1, \dots ) \\
( \dots, 1, 0, 1, \dots ) \leftrightarrow ( \dots, 0, 0, 0, \dots )
\end{gather*}
of the source, where the central zero is on one of the strands except
the first, second, or last. Using these transformations, we can get
from any subexpression $\un{e}$ with a zero anywhere to some
subexpression with a zero in the first, second, or last position. The
lemma follows.
\end{proof}

\subsection{Zamolodchikov (or 3-color) relations}

The remaining relations in $\HC$ are associated with finitary rank $3$
parabolic subgroups. The possible types of relations are: $I_2(m)
\times A_1$, $A_3$, $B_3$, and (the forbidden) $H_3$. These relations
can be found in \cite[\S 1.4.3]{EW}.

That the relation of type $I_2(m) \times A_1$ is preserved by $\Lambda$ is immediate. Let $\{s,t\}$
be the simple reflections in $I_2(m)$, and $u$ be the commuting simple
reflection in $A_1$. The relation is an equality of two morphisms from
$(u,s,t,s,\ldots)$ to $(t,s,t,\ldots,u)$ which only use vertices
$G_{s,t}$, $G_{u,s}$, and $G_{u,t}$. If $e \in \{0,1\}$ and $\un{f}
\subset (s,t,s,\ldots)$ and $\un{f}' \subset (t,s,t,\ldots)$, then we
claim that after applying $\Lambda$ both sides send
\[ [e \un{f}] \mapsto \sum_{\un{f}'} (G_{s,t})_{\un{f}}^{\un{f}'} [\un{f'} e]. \]
The proof is relatively straightforward using Example \ref{ex:m2} and the fact that $u$ fixes $\pi_{s,t}$ and $\zeta(\un{f}')$ for any $\un{f}'$.

The relations of type $A_3$ and $B_3$ are complicated, but (unlike the
two-color relations) they do not come in families, so one need not
provide a general argument. Both relations 
have been checked by computer. For example, the $B_3$ relation says
that two diagrams are equal; both diagrams have 9 strands on bottom
and top, and consist of thirteen $2m$-valent 
vertices. Checking that this relation is preserved by $\Lambda$
amounts to checking that two different thirteen-fold products of
sparse $2^9 \times 2^9$ matrices are equal.

Computer computation with the formulas in this paper is possible with
the Magma package ASLoc:
\[
\text{https://github.com/joelgibson/ASLoc}
\]
The reader is referred to \cite{GJW} for more detail on the
implementation. The reader wishing to check
the Zamolodchikov relations in types $A_3$ and $B_3$ should go to the above link, install the
package, and read the section on ``Verifying relations''.

\begin{remark} \label{rmk:ZamSubtle} The reader may be consoled that the relations of type $A_3$ and $B_3$ can also be checked with the precomposition trick, but the details are technical and not worth
repeating here. One can prove that placing a dot below these relations yields an equality which already follows from the other relations. Thus the diagrams must be equal on any
subexpression with a zero. Meanwhile, both sides send $[11\ldots 1]\mapsto [11 \ldots 1]$ because each $2m$-valent vertex $G_{s,t}$ does the same.

Note that there is considerable subtlety in the Zamolodchikov relations, and sweeping it under the rug entirely would be a disservice. The fact that precomposition with a dot gives
an equality should be viewed as rather surprising! For example, the Zamolodchikov relation in type $A_3$ or $B_3$ can be stated as follows. There is a particular reduced expression
for the longest element, and there are two different sequences of braid relations one can apply to get an ``antipodal'' reduced expression for the longest element. Turning these
sequences of braid relations into sequences of $2m$-valent vertices, the two resulting diagrams are equal. However, this statement of equality is false if we chose the wrong reduced
expression to begin with, and one can check this by precomposing with a dot and confirming that the diagrams are \emph{not} equal. \end{remark}

\subsection{The status of type $H_3$} \label{subsec:H3} 
This section is an informal discussion of the current status of the
Hecke category in types containing $H_3$. For further discussion see
\cite[\S 1.4.3]{EW}.

Given a reduced expression, and a sequence of braid relations to
another reduced expression, one obtains a diagram by transforming each
braid relation into a $2m$-valent vertex. 
Such a diagram will be called a \emph{rex move}. By construction, rex
moves induce the identity map $[11\ldots 1]\mapsto [11 \ldots 1]$. By
applying the functor to algebraic 
Soergel bimodules, we obtain morphisms between Bott-Samelson bimodules
which we also call rex moves. We assume for this discussion that
Soergel bimodules categorify the Hecke 
algebra and obey the Soergel hom formula, since this is the behavior
that we wish the diagrammatic Hecke category to mimic. 

Let $X$ and $Y$ be two rex moves with the same source and target, both
reduced expressions for an element $w \in W$. Let $P(X,Y)$ be the
property that $X-Y$ is a linear combination 
of double leaves which factor through lower terms. Using the Soergel
Hom formula, one can deduce a priori that the $P(X,Y)$ holds for
morphisms between Soergel bimodules. Given 
$P(X,Y)$ for all comparable $X$ and $Y$, \cite{EW}  gives a diagrammatic proof that double leaves span the diagrammatic
Hecke category. Moreover, all (non-trivial) cycles in reduced expression graphs come
from the longest element in a rank 3 parabolic subgroup (this has been
noticed by several authors, and is the subject of
\cite{EWdiag}). Thanks to this one can prove as in \cite{EW} that
$P(X_0,Y_0)$ for a single pair $X_0$
and $Y_0$ (certain rex moves between reduced expressions for the
longest element, which are not homotopic as paths in the reduced
expression graph) is sufficient to deduce $P(X,Y)$ for all pairs $X$
and $Y$. The goal of the Zamolodchikov relation is to impose $P(X_0,
Y_0)$ for one such pair. 

In types $I_2(m) \times A_1$ and $A_3$ and $B_3$, one is lucky:
there exists a pair $X_0$ and $Y_0$ such that $X_0 - Y_0 = 0$. No
lower terms are needed for this difference of 
rex moves. As noted in Remark \ref{rmk:ZamSubtle}, this requires a
special choice of $X_0$ and $Y_0$, and there is no a priori reason it
should be true. 

In type $H_3$, it was checked by computer that there is no pair $X_0$
and $Y_0$ such that $X_0 - Y_0 = 0$. (It somewhat non-trivial to check that this is the case. Proceeding naively
involves the multiplication of many sparse matrices with $2^{15} =
32768$ rows and columns, which is far beyond our hardware's capabilities. However, for each pair $X_0$ and $Y_0$ we were able to find a
block of this sparse matrix for which $X_0 - Y_0$ was nonzero.)

To define the Hecke category, we
would need to find a linear combination $L$ of double leaves through
lower terms, such that $X_0 - Y_0 = L$, and impose this as the $H_3$
Zamolodchikov relation. Any relation of this form is sufficient to
imply that double leaves span the diagrammatic category. However, to
prove that double leaves are linearly independent one requires a
functor (either the functor to bimodules, or the functor $\Lambda$
defined herein), and for this functor to be well-defined, the relation
$X_0 - Y_0 = L$ must hold. Picking some $X_0$ and $Y_0$,
we were unable to discover the correct linear combination $L$ so that
$X_0 - Y_0 = L$ is preserved by $\Lambda$. This remains an interesting
computational challenge.

Such a relation is all that is missing for a definition of the
diagrammatic Hecke category in types containing $H_3$.

\section{Heuristics for the definition of the localization functor} \label{sec:heuristics}

Let us give some details on where the definition of $\Lambda$ came from, by reminding the reader of some technology developed in \cite[\S 5.4 and following]{EW}.

We continue to use $B_{\reds}$ to refer both to a Bott-Samelson
bimodule, and to a generating object in the diagrammatic Hecke
category. Similarly, we use $R$ to refer both to the regular $R$-bimodule, and the
monoidal unit of the Hecke category.

\subsection{Mixed calculus and the splitting of $B_s$}

Firstly, the category $\Omega_Q W$ (or rather, its monoidal strictification) was given a diagrammatic presentation by generators and relations in \cite{EWdiag}. Instead of having
objects $r_w$ for each $w \in W$, the objects are sequences $\un{w} = (s_1, \ldots, s_d)$ of simple reflections, meant to embody the object $r_{s_1} \ot \cdots \ot r_{s_d}$. Below
we draw the object $r_s$ (or rather, its identity map) as a squiggly line, of the same color as the solid line which represents $B_s$:
\[
  \id_{B_{\color{red} s}} :=
\begin{array}{c}
  \begin{tikzpicture}
        \draw[dashed] (-0.5,1) to (.5,1); \draw[dashed] (-.5,0) to (.5,0);
\draw [red, decorate] %decoration={snake,amplitude=.4mm,segment
  %length=2mm,post length=1mm}]
  (0,0) -- (0,1); 
  % \draw [red] (0,0.5) -- (0,1)
\end{tikzpicture} \end{array}, 
\qquad
\id_{r_{\color{red} s}} :=
\begin{array}{c}
  \begin{tikzpicture}
            \draw[dashed] (-0.5,1) to (.5,1); \draw[dashed] (-.5,0) to (.5,0);
\draw [red, decorate, decoration={snake,amplitude=.4mm,segment
  length=1mm,post length=0.5mm}]
% [red, decorate, decoration={snake,amplitude=.4mm,segment
%  length=2mm,post length=1mm}]
(0,0) -- (0,1); 
%\draw [red] (0,0.5) -- (0,1);
\end{tikzpicture}
  \end{array}.
\]
Morphisms in the category are linear combinations of diagrams built
out of squiggly cups and caps, and squiggly $2m$-valent
vertices. These diagrams are subject to the relations spelled out in
\cite[\S 4.1]{EW} (see also \cite[\S 6]{EWdiag}). In particular, squiggly cups
and caps define isomorphisms between the unit and $r_s^{\otimes 2}$.

Then in \cite[\S 5.4]{EW} we developed a diagrammatic calculus which mixed the objects $B_s$ and $r_s$. For sake of discussion we call this the \emph{mixed calculus}. This is a model for the monoidal subcategory of $R$-bimodules containing both the standard bimodules $R_s$ and the Bott-Samelson bimodules $B_s$. In addition to the diagrams appearing in $\HC$ and in $\Omega_Q W$ individually, we also have new generating morphisms
\begin{equation} \label{eq:newmixedmaps}
\begin{array}{c}
  \begin{tikzpicture}
            \draw[dashed] (-0.5,1) to (.5,1); \draw[dashed] (-.5,0) to (.5,0);
\draw [red, decorate, decoration={snake,amplitude=.4mm,segment
  length=1mm,post length=0.5mm}] (0,0) -- (0,0.5); 
\draw [red] (0,0.5) -- (0,1);
\end{tikzpicture} \end{array}
\quad \text{and} \quad
\begin{array}{c}
  \begin{tikzpicture}
            \draw[dashed] (-0.5,1) to (.5,1); \draw[dashed] (-.5,0) to (.5,0);
\draw [red] (0,0) -- (0,0.5); 
\draw [red, decorate, decoration={snake,amplitude=.4mm,segment
  length=1mm,post length=0.5mm}] (0,0.5) -- (0,1);
\end{tikzpicture}
\end{array},
\end{equation}
both of degree $+1$.

\begin{remark} In the category of $R$-bimodules there are important non-split short exact sequences
	\[ 0 \to R_s(-1) \to B_s \to R_{\id}(+1) \to 0 \]
and
	\[ 0 \to R_{\id}(-1) \to B_s \to R_s(+1) \to 0 \]
which describe how $B_s$ is filtered by standard/costandard bimodules. The maps between $B_s$ and $R_{\id}$ are maps between Bott-Samelson bimodules, so they appear already in $\HC$ as the two dot maps. The maps between $B_s$ and $R_s$ are new in this mixed calculus, and they are the new maps of \eqref{eq:newmixedmaps}. After localization to $Q$, these two short exact sequences  ``split each other'' up to the invertible element $\alpha_s$ (see \cite[\S 3.6]{EW}).  \end{remark}

Inside the mixed calculus we have the following relations (these are
not all of them, just enough to illustrate a point):
\begin{equation} \label{eq:locrel1}
\begin{array}{c}
  \begin{tikzpicture}[scale=0.6]
    \draw[dashed] (-0.5,1.5) to (.5,1.5); \draw[dashed] (-.5,0) to (.5,0);
\draw [red, decorate, decoration={snake,amplitude=.4mm,segment
  length=1mm,post length=.5mm}]  (0,0) -- (0,0.6); 
\draw [red] (0,0.6) -- (0,0.9);
\draw [red, decorate, decoration={snake,amplitude=.4mm,segment
  length=1mm,post length=.5mm}]  (0,0.9) -- (0,1.5); 
\end{tikzpicture}
\end{array}
 =
{\color{red}\a}
 \begin{array}{c}
\begin{tikzpicture}[scale=0.6]
    \draw[dashed] (-0.5,1.5) to (.5,1.5); \draw[dashed] (-.5,0) to
    (.5,0);
  \draw [red, decorate, decoration={snake,amplitude=.4mm,segment
  length=1mm,post length=.5mm}]  (0,0) -- (0,1.5); 
\end{tikzpicture}
\end{array},
\qquad
\begin{array}{c}
\tikz[xscale=0.35,yscale=0.35]{
\draw[dashed] (2,1.5) to (4,1.5); \draw[dashed] (2,-1) to (4,-1);
\draw[color=red] (3,-.5) to (3,1);
\node[circle,fill,draw,inner sep=0mm,minimum size=1mm,color=red] at (3,-.5) {};
\node[circle,fill,draw,inner sep=0mm,minimum size=1mm,color=red] at (3,1) {};
}
\end{array}
=
\begin{array}{c}
\tikz[xscale=0.35,yscale=0.35]{
\draw[dashed] (2,1.5) to (4,1.5); \draw[dashed] (2,-1) to (4,-1);
\node[color=red] at (3,0.25) {$\a$};
}
\end{array},
\end{equation}
% SECOND EQUATION
\begin{equation}
\label{eq:locreldot}
\begin{array}{c}
  \begin{tikzpicture}[scale=0.6]
    \draw[dashed] (-.5,1.5) to (.5,1.5); \draw[dashed] (-.5,0) to (.5,0);
\draw [red, decorate, decoration={snake,amplitude=.4mm,segment
  length=1mm,post length=.5mm}]  (0,0) -- (0,0.7); 
\draw [red] (0,0.7) -- (0,1);
\node [circle,fill=red,inner sep=0pt,minimum size=2mm] at (0,1) {};
\end{tikzpicture}
\end{array}
= 0 =
\begin{array}{c}
  \begin{tikzpicture}[scale=0.6]
    \draw[dashed] (-.5,1) to (.5,1); \draw[dashed]  (-.5,-.5) to (.5,-.5);
\draw [red] (0,0) -- (0,0.3);
\draw [red, decorate, decoration={snake,amplitude=.4mm,segment
  length=1mm,post length=.5mm}] (0,0.3) -- (0,1); 
\node [circle,fill=red,inner sep=0pt,minimum size=2mm] at (0,0) {};
\end{tikzpicture}
\end{array},
\end{equation}
% THIRD EQUATION
\begin{equation} \label{eq:locidempotenttimesalpha}
{\color{red} \a}
\begin{array}{c}
  \begin{tikzpicture}[yscale=0.6,xscale=0.6]
    \draw[dashed] (-.5,2) to (.5,2); \draw[dashed] (-.5,0) to (.5,0);
    \draw[red] (0,0) to (0,2);
  \end{tikzpicture}
\end{array}
=
\begin{array}{c}
  \begin{tikzpicture}[yscale=0.6,xscale=0.6]
    \draw[dashed] (-.5,2) to (.5,2); \draw[dashed] (-.5,0) to (.5,0);
    \draw[red] (0,0) to (0,0.6);
\node [circle,fill=red,inner sep=0pt,minimum size=2mm] at (0,0.6) {};
    \draw[red] (0,1.4) to (0,2);
\node [circle,fill=red,inner sep=0pt,minimum size=2mm] at (0,1.4) {};
  \end{tikzpicture}
\end{array}
+ 
\begin{array}{c}
  \begin{tikzpicture}[yscale=0.6,xscale=0.6]
    \draw[dashed] (-.5,2) to (.5,2); \draw[dashed] (-.5,0) to (.5,0);
\draw[red] (0,0) to (0,0.6);
\draw [red, decorate, decoration={snake,amplitude=.4mm,segment
  length=2mm,post length=1mm}] (0,0.6) -- (0,1.4); 
 \draw[red] (0,1.4) to (0,2);
  \end{tikzpicture}
\end{array}.
\end{equation}
\begin{equation} \label{eq:pullf}
 \begin{array}{c}
\begin{tikzpicture}[scale=0.6]
    \draw[dashed] (-0.8,1.5) to (.8,1.5); \draw[dashed] (-.8,0) to
    (.8,0);
  \draw [red, decorate, decoration={snake,amplitude=.4mm,segment
    length=1mm,post length=.5mm}]  (0,0) -- (0,1.5);
  \node at (-.6,0.75) {$f$};
\end{tikzpicture}
 \end{array}
 =
 \begin{array}{c}
\begin{tikzpicture}[scale=0.6]
    \draw[dashed] (-0.8,1.5) to (.8,1.5); \draw[dashed] (-.8,0) to
    (.8,0);
  \draw [red, decorate, decoration={snake,amplitude=.4mm,segment
    length=1mm,post length=.5mm}]  (0,0) -- (0,1.5);
  \node at (.6,0.75) {$sf$};
\end{tikzpicture}
 \end{array}
 \end{equation}
 
Now pass to the localization $Q$ where $\frac{1}{\alpha}$ is defined. Combining the relations above we see that
\begin{equation} \label{eq:locidempotent}
\begin{array}{c}
  \begin{tikzpicture}[yscale=0.6,xscale=0.6]
        \draw[dashed] (-.5,2) to (.5,2); \draw[dashed] (-.5,0) to (.5,0);
    \draw[red] (0,0) to (0,2);
  \end{tikzpicture}
\end{array}
=
\frac{1}{\color{red} \a}
\begin{array}{c}
  \begin{tikzpicture}[yscale=0.6,xscale=0.6]
        \draw[dashed] (-.5,2) to (.5,2); \draw[dashed] (-.5,0) to (.5,0);
    \draw[red] (0,0) to (0,0.6);
\node [circle,fill=red,inner sep=0pt,minimum size=2mm] at (0,0.6) {};
    \draw[red] (0,1.4) to (0,2);
\node [circle,fill=red,inner sep=0pt,minimum size=2mm] at (0,1.4) {};
  \end{tikzpicture}
\end{array}
+ \frac{1}{\color{red} \a}
\begin{array}{c}
  \begin{tikzpicture}[yscale=0.6,xscale=0.6]
        \draw[dashed] (-.5,2) to (.5,2); \draw[dashed] (-.5,0) to (.5,0);
\draw[red] (0,0) to (0,0.6);
\draw [red, decorate, decoration={snake,amplitude=.4mm,segment
  length=2mm,post length=1mm}] (0,0.6) -- (0,1.4); 
 \draw[red] (0,1.4) to (0,2);
  \end{tikzpicture}
\end{array}
\end{equation}
is actually a decomposition of $\id_{B_s}$ into the sum of two orthogonal idempotents. More precisely, set
\begin{equation} \label{eq:incprojdefined}
i_s = \begin{array}{c}
        \begin{tikzpicture}[scale=0.7]
                  \draw[dashed] (-.5,1) to (.5,1); \draw[dashed] (-.5,0) to (.5,0);
\draw [red, decorate, decoration={snake,amplitude=.4mm,segment
  length=1mm,post length=0.5mm}] (0,0) -- (0,0.5); 
\draw [red] (0,0.5) -- (0,1);
\end{tikzpicture} \end{array}, \quad
p_2 = \frac{1}{{\color{red} \a}}
\begin{array}{c}
  \begin{tikzpicture}[scale=0.7]
            \draw[dashed] (-.5,1) to (.5,1); \draw[dashed] (-.5,0) to (.5,0);
\draw [red] (0,0) -- (0,0.5); 
\draw [red, decorate, decoration={snake,amplitude=.4mm,segment
  length=1mm,post length=0.5mm}] (0,0.5) -- (0,1);
\end{tikzpicture}
\end{array}, \quad i_{\id} = \begin{array}{c}
\tikz[xscale=0.35,yscale=-0.35]{
%\draw[dashed] (3,0) circle (1cm);
\draw[dashed] (2,-1) to (4,-1); \draw[dashed] (2,1) to (4,1);
\draw[color=red] (3,-1) to (3,0);
\node[circle,fill,draw,inner sep=0mm,minimum size=1mm,color=red] at (3,0) {};}
\end{array}, \quad
p_{\id} = \frac{1}{{\color{red} \a}}
\begin{array}{c}
\tikz[scale=0.35]{
%\draw[dashed] (3,0) circle (1cm);
\draw[dashed] (2,-1) to (4,-1); \draw[dashed] (2,1) to (4,1);
\draw[color=red] (3,-1) to (3,0);
\node[circle,fill,draw,inner sep=0mm,minimum size=1mm,color=red] at (3,0) {};}
\end{array}.
\end{equation}
Then these are the inclusion and projection maps of complementary direct summands:
\[ p_s i_s = \id_{r_s}, \quad  p_{\id} i_{\id} = \id_{r_{\id}}, \]
\[ p_s i_{\id} = 0, \quad p_{\id} r_s = 0, \]
\[ \id_{B_s} = i_s p_s + i_{\id} p_{\id}. \] 
This makes explicit the direct sum decomposition $B_s \ot_R Q \cong r_s \oplus r_{\id}$, and we will use these projection and inclusion maps below. 

\begin{remark} While we will not recall all the relations in the mixed calculus, we wish to warn the reader of an important technicality which appears in the mixed calculus but not in $\HC$ or in $\Omega_Q W$ itself. Namely, the new morphisms of \eqref{eq:newmixedmaps} are \emph{not} mates under adjunction. Instead, this is only true up to a sign. We have relations without signs
\begin{gather}
  \label{eq:loccap3}
\begin{array}{c}
  \begin{tikzpicture}[scale=0.6]
                      \draw[dashed] (-1.5,-1) to (.5,-1); \draw[dashed] (-1.5,1) to (.5,1);
  \draw [red, decorate, decoration={snake,amplitude=.4mm,segment
  length=1mm,post length=0.5mm}] (0,1) -- (0,0); 
\draw[red] (0,0) .. controls (0,-1) and (-1,-1) ..  (-1,0) to (-1,1);
\end{tikzpicture}
  \end{array}
=
\begin{array}{c}
  \begin{tikzpicture}[scale=0.6]
                          \draw[dashed] (-1.5,-1) to (.5,-1); \draw[dashed] (-1.5,1) to (.5,1);
 \draw [red, decorate, decoration={snake,amplitude=.4mm,segment
  length=1mm,post length=.5mm}] (0,1) to (0,0.3) .. controls (0,-1) and (-1,-1) ..
(-1,0.3); 
\draw [red] (-1,0) -- (-1,1);
\end{tikzpicture}
  \end{array} \quad \text{and} \quad
% \label{eq:loccap4}
\begin{array}{c}
  \begin{tikzpicture}[scale=0.6]
                          \draw[dashed] (-1.5,-1) to (.5,-1); \draw[dashed] (-1.5,1) to (.5,1);
  \draw [red, decorate, decoration={snake,amplitude=.4mm,segment
  length=1mm,post length=0.5mm}] (0,-1) -- (0,0); 
\draw[red] (0,0) .. controls (0,1) and (-1,1) ..  (-1,0) to (-1,-1);
\end{tikzpicture}
  \end{array}
=
\begin{array}{c}
  \begin{tikzpicture}[scale=0.6]
                          \draw[dashed] (-1.5,-1) to (.5,-1); \draw[dashed] (-1.5,1) to (.5,1);
 \draw [red, decorate, decoration={snake,amplitude=.4mm,segment
  length=1mm,post length=.5mm}] (0,-1) to (0,-0.3) .. controls (0,1) and (-1,1) ..
(-1,-0.3); 
\draw [red] (-1,0) -- (-1,-1);
\end{tikzpicture}
  \end{array} .
\end{gather}
%\draw [red, decorate, decoration={snake,amplitude=.4mm,segment
%  length=2mm,post length=1mm}] (0,0.6) -- (0,1.4); 
and \emph{with} signs
\begin{gather}
  \label{eq:loccap1}
\begin{array}{c}
  \begin{tikzpicture}[scale=0.6]
                          \draw[dashed] (-.5,-1) to (1.5,-1); \draw[dashed] (-.5,1) to (1.5,1);
  \draw [red, decorate, decoration={snake,amplitude=.4mm,segment
  length=1mm,post length=0.5mm}] (0,1) -- (0,0); 
\draw[red] (0,0) .. controls (0,-1) and (1,-1) ..  (1,0) to (1,1);
\end{tikzpicture}
  \end{array}
=
- \begin{array}{c}
    \begin{tikzpicture}[scale=0.6]
                            \draw[dashed] (-.5,-1) to (1.5,-1);  \draw[dashed] (-.5,1) to (1.5,1);
 \draw [red, decorate, decoration={snake,amplitude=.4mm,segment
  length=1mm,post length=.5mm}] (0,1) to (0,0.3) .. controls (0,-1) and (1,-1) ..
(1,0.3); 
\draw [red] (1,0) -- (1,1);
\end{tikzpicture}
  \end{array} \quad \text{and} \quad
\begin{array}{c}
  \begin{tikzpicture}[scale=0.6]
                          \draw[dashed] (-.5,-1) to (1.5,-1); \draw[dashed] (-.5,1) to (1.5,1);
  \draw [red, decorate, decoration={snake,amplitude=.4mm,segment
  length=1mm,post length=0.5mm}] (0,-1) -- (0,0); 
\draw[red] (0,0) .. controls (0,1) and (1,1) ..  (1,0) to (1,-1);
\end{tikzpicture}
  \end{array}
=
-\begin{array}{c}
   \begin{tikzpicture}[scale=0.6]
                           \draw[dashed] (-.5,-1) to (1.5,-1);  \draw[dashed] (-.5,1) to (1.5,1);
 \draw [red, decorate, decoration={snake,amplitude=.4mm,segment
  length=1mm,post length=.5mm}] (0,-1) to (0,-0.3) .. controls (0,1) and (1,1) ..
(1,-0.3); 
\draw [red] (1,0) -- (1,-1);
\end{tikzpicture}
 \end{array}.
\end{gather}
(Note that both pairs of relations have a vertical symmetry.)

As a consequence the new morphisms \eqref{eq:newmixedmaps} are mates under adjunction only up to a sign:
\begin{equation} \label{eq:locmates}
\begin{array}{c}
  \begin{tikzpicture}[scale=0.7]
    \draw[dashed] (-.5,0) to (1.5,0);     \draw[dashed] (-.5,2.2) to (1.5,2.2); 
    \draw [red, decorate, decoration={snake,amplitude=.4mm,segment
  length=1mm,post length=.5mm}] (0,0) .. controls (0,2.5) and (0.6,2) .. (0.6,1);
    \draw [red] (0.6,1) .. controls (0.6,0) and (1.2,0) .. (1.2,2.2);
  \end{tikzpicture}
\end{array}
= -
\begin{array}{c}
  \begin{tikzpicture}[scale=0.7]
    \draw[dashed] (-.5,0) to (.5,0);     \draw[dashed] (-.5,2.2) to (.5,2.2); 
\draw [red, decorate, decoration={snake,amplitude=.4mm,segment
  length=1mm,post length=.5mm}] (0,0) -- (0,1.2); 
\draw [red] (0,1.2) -- (0,2.2);
\end{tikzpicture}
\end{array}
\quad \text{and} \quad
\begin{array}{c}
  \begin{tikzpicture}[scale=0.7]
    \draw[dashed] (-.5,0) to (1.5,0);     \draw[dashed] (-.5,2.2) to (1.5,2.2); 
    \draw [red] (0,0) .. controls (0,2.5) and (0.6,2) .. (0.6,1);
    \draw [red, decorate, decoration={snake,amplitude=.4mm,segment
  length=1mm,post length=.5mm}]   (0.6,1) .. controls (0.6,0) and (1.2,0) .. (1.2,2.2);
  \end{tikzpicture}
\end{array}
= 
\begin{array}{c}
  \begin{tikzpicture}[scale=0.7]
        \draw[dashed] (-.5,0) to (.5,0);     \draw[dashed] (-.5,2.2) to (.5,2.2); 
\draw [red] (0,0) -- (0,0.8); 
\draw [red, decorate, decoration={snake,amplitude=.4mm,segment
  length=1mm,post length=.5mm}]  (0,0.8) -- (0,2.2);
\end{tikzpicture}
\end{array}
\end{equation}
(Another pair of relations is obtained by performing a vertical flip.)
\end{remark}

\subsection{Localizing morphisms}

Here we elaborate on the content of \cite[\S 5.5]{EW}.

For a word $\un{w} = (s_1, \ldots, s_d)$ in $S$ and a subexpression $\un{e} \subset \un{w}$, we take the tensor product of the inclusion and projection maps of
\eqref{eq:incprojdefined} to produce the inclusion and projection maps for the summand $r_{\un{w}^{\un{e}}} \subset B_{\un{w}}$ (see \eqref{eq:can_decomp}). These were denoted
$i(\un{e})$ and $p(\un{e})$ in \S\ref{subsec:defnobjects}.

\begin{ex}
  For example, if $\un{w} = ( { \color{red} s }, { \color{blue} t }, {
  \color{red} s }, { \color{green} u })$ and $\un{e} = (0,1,1,0)$ then
\[ \begin{array}{c}
\begin{tikzpicture}[scale=0.7]
\draw rectangle (0,0) rectangle (4,2);
\node [rectangle, fill=white] at (2,0) {$\un{e}$};
\node [rectangle, fill=white] at (2,2) {$\un{w}$};
\node at (2,1) {$i(\un{e})$};
\end{tikzpicture}
\end{array}
  = 
\begin{array}{c}
  \begin{tikzpicture}[scale=0.7]
        \draw[dashed] (-.5,0) to (3.5,0);     \draw[dashed] (-.5,-2) to (3.5,-2); 
\draw [red] (0,0) -- (0,-1);
\node [circle,fill=red,inner sep=0pt,minimum size=2mm] at (0,-1) {};
\draw [blue] (1,0) -- (1,-1);
\draw [blue, decorate, decoration={snake,amplitude=.4mm,segment
  length=2mm,post length=1mm}] (1,-1) -- (1,-2); 
\draw [color=red] (2,0) -- (2,-1);
\draw [red, decorate, decoration={snake,amplitude=.4mm,segment
  length=2mm,post length=1mm}] (2,-1) -- (2,-2); 
\draw [color=green] (3,0) -- (3,-1);
\node [circle,fill=green,inner sep=0pt,minimum size=2mm] at (3,-1) {};
\end{tikzpicture}
\end{array}
\]
and 
\[ \begin{array}{c}
\begin{tikzpicture}[scale=0.7]
\draw rectangle (0,0) rectangle (4,2);
\node [rectangle, fill=white] at (2,0) {$\un{w}$};
\node [rectangle, fill=white] at (2,2) {$\un{e}$};
\node at (2,1) {$p(\un{e})$};
\end{tikzpicture}
\end{array}
  = 
\begin{array}{c}
  \begin{tikzpicture}[scale=0.7]
            \draw[dashed] (-.5,0) to (3.5,0);     \draw[dashed] (-.5,2) to (3.5,2); 
\draw [red] (0,0) -- (0,1);
\node [circle,fill=red,inner sep=0pt,minimum size=2mm] at (0,1) {};
\node [red] at (0,1.5) {$1/\a_s$};
\draw [blue] (1,0) -- (1,1);
\draw [blue, decorate, decoration={snake,amplitude=.4mm,segment
  length=2mm,post length=1mm}] (1,1) -- (1,2); 
\draw [color=red] (2,0) -- (2,1);
\draw [red, decorate, decoration={snake,amplitude=.4mm,segment
  length=2mm,post length=1mm}] (2,1) -- (2,2); 
\draw [color=green] (3,0) -- (3,1);
\node [circle,fill=green,inner sep=0pt,minimum size=2mm] at (3,1) {};
\node [green] at (3,1.5) {$1/\a_u$};
\end{tikzpicture}
\end{array}
\]
\end{ex}
In \cite{EW} and \cite{EWdiag} we have proven that every morphism in the mixed calculus whose boundary consists only of squiggly lines is itself a $Q$-linear combination of
squiggly diagrams (i.e. one does not need the new morphisms from \eqref{eq:newmixedmaps} or any diagrams from $\HC$, but only squiggly cups, caps, and $2m$-valent vertices.)
Moreover, any two squiggly diagrams with the same boundary are equal. Squiggly diagrams are isomorphisms, and any squiggly diagram represents the same isomorphism; we refer to any morphism represented by a squiggly diagram as the canonical isomorphism from its source to its target.

If $\psi \colon B_{\un{w}} \to B_{\un{x}}$ is a morphism between Bott-Samelson objects, and $\un{e} \subset \un{w}$ and $\un{f} \subset \un{x}$, then the entry in column $\un{e}$
and row $\un{f}$ of $\Lambda(\psi)$ should be the coefficient of the canonical isomorphism in the composition $p(\un{f}) \circ \psi \circ i(\un{e})$:
\[
\begin{array}{c}
\begin{tikzpicture}[xscale=0.7, yscale=1]
\draw (0,-1) rectangle (4,-2);
\node [rectangle, fill=white] at (2,-2) {$\un{e}$};
\node at (2,-1.5) {$i(\un{e})$};
\draw (0,1) rectangle (4,-1);
\node at (2,0) {$\psi$};
\draw (0,1) rectangle (4,2);
\node [rectangle, fill=white] at (2, 1) {$\underline{x}$};
\node [rectangle, fill=white] at (2,2) {$\un{f}$};
\node at (2,1.5) {$p(\un{f})$};
\node [rectangle, fill=white] at (2, -1) {$\underline{w}$};
\end{tikzpicture}
\end{array} = \Lambda(\psi)_{\un{e}}^{\un{f}} \begin{array}{c}
\begin{tikzpicture}[scale=0.7]
\draw rectangle (0,0) rectangle (4,2);
\node [rectangle, fill=white] at (2,0) {$\un{e}$};
\node [rectangle, fill=white] at (2,2) {$\un{f}$};
\node at (2,1) {$\can$};
\end{tikzpicture}
\end{array}
\]
Here $\can$ denotes the canonical
isomorphism.

\begin{remark}
If the reader desires a mnemonic, they should remember that all the complicated business happens on top of morphisms! This is because we have included the factors $\frac{1}{\alpha}$ into the projection maps, not the inclusion maps.
\end{remark}

We can compute these matrix entries explicitly using the mixed calculus, as we do in several examples below.

\begin{ex} The dot map $B_s \to R = r_{\id}$ from \eqref{eq:ddot} should be precomposed with the two inclusion maps $i_{\id}$ and $i_s$ to determine its matrix entries after applying $\Lambda$. We have
\[ \begin{array}{c}
\tikz[xscale=0.4,yscale=0.6]{
\draw[dashed] (2,1.5) to (4,1.5); \draw[dashed] (2,-1) to (4,-1);
\draw[color=red] (3,-.5) to (3,1);
\node[circle,fill,draw,inner sep=0mm,minimum size=1mm,color=red] at (3,-.5) {};
\node[circle,fill,draw,inner sep=0mm,minimum size=1mm,color=red] at (3,1) {};
}
\end{array}
=
\begin{array}{c}
\tikz[xscale=0.4,yscale=0.6]{
\draw[dashed] (2,1.5) to (4,1.5); \draw[dashed] (2,-1) to (4,-1);
\node[color=red] at (3,0.25) {$\a$};
}
\end{array}\qquad \text{and} \quad \begin{array}{c}
                                     \begin{tikzpicture}[yscale=.6]
\draw[dashed] (-.5,-1) to (.5,-1); \draw[dashed] (-.5,1.5) to (.5,1.5);
\draw [red, decorate, decoration={snake,amplitude=.4mm,segment
  length=1mm,post length=.5mm}]  (0,-1) -- (0,-.3); 
\draw [red] (0,-.3) -- (0,.5);
\node [circle,fill=red,inner sep=0pt,minimum size=2mm] at (0,.5) {};
\end{tikzpicture}
\end{array}
= 0. \]
This gives us the matrix $\left( \begin{array}{cc} {\color{red} \a} & 0 \end{array} \right)$ from \eqref{eq:ddot}. \end{ex}

\begin{ex} The dot map $R \to B_s$ from \eqref{eq:idot} should be postcomposed with the two projection maps $p_{\id}$ and $p_s$ to determine its matrix entries after applying
$\Lambda$. The calculation is almost identical to the previous example (upside-down) except that $p_{\id}$ and $p_s$ have an additional factor of $\frac{1}{{\color{red} \a}}$,
whence the matrix in \eqref{eq:idot}. \end{ex}
	
For the trivalent vertices, the reader is encouraged to use relation \eqref{eq:locidempotenttimesalpha} to prove that
\begin{equation}
\begin{array}{c}
\tikz[xscale=0.35,yscale=0.6]{
\draw[dashed] (1.5,1) to (4.5,1); \draw[dashed] (1.5,-1.5) to (4.5,-1.5);
  \draw[color=red] (3,-1.5) to (3,-.5);
  \draw[color=red] (3,-.5) to[out=170,in=-90] (2,0);
  \draw[color=red] (3,-.5) to[out=10,in=-90] (4,0);
  \draw[color=red] (3,-1.5) to (3,-.5);
%\node[circle,fill,draw,inner sep=0mm,minimum size=1mm,color=red] at (3,-.5) {};
%\node[circle,fill,draw,inner sep=0mm,minimum size=1mm,color=red] at
%(3,1) {};
  \draw [red, decorate, decoration={snake,amplitude=.4mm,segment
  length=1mm,post length=.5mm}]  (2,0) to[out=90,in=90] (4,0); 
}
\end{array}
=
\begin{array}{c}
\tikz[xscale=0.35,yscale=0.6]{
\draw[dashed] (1.5,1) to (4.5,1); \draw[dashed] (1.5,-1.5) to (4.5,-1.5);
  \draw[color=red] (3,-1.5) to (3,-.5);
  % \draw[color=red] (3,-.5) to[out=170,in=-90] (2,0);
  % \draw[color=red] (3,-.5) to[out=10,in=-90] (4,0);
  \draw[color=red] (3,-1.5) to (3,-.5);
\node[circle,fill,draw,inner sep=0mm,minimum size=1mm,color=red] at (3,-.5) {};
%\node[circle,fill,draw,inner sep=0mm,minimum size=1mm,color=red] at
%(3,1) {};
  % \draw [red, decorate, decoration={snake,amplitude=.4mm,segment
  % length=1mm,post length=.5mm}]  (2,0) to[out=90,in=90] (4,0); 
}
\end{array}.
\end{equation}
From this it becomes easy to compute the matrices associated to trivalent vertices, see e.g. the proof of \cite[(5.29)]{EW}.

When a $2m$-valent vertex is precomposed with $i(\un{e})$ and $\un{e}$ contains a zero somewhere, then the $2m$-valent vertex meets a dot, and we can apply the Jones-Wenzl relation \eqref{Eq:JWreln}. Thus we can continue our computation using only the one-color relations. If instead it is precomposed with $i(11\cdots 1)$ then we can use \cite[(5.28)]{EW}, which gives the relationship between solid and squiggly $2m$-valent vertices. In formulas \cite[(5.28)]{EW} says that
\[ G_{s,t} \circ i(\un{11\cdots 1}) = i(\un{11 \cdots 1}) \circ \can, \]
where in this case $\can$ is the squiggly version of $G_{s,t}$. Consequently
\[ p(\un{11 \cdots 1}) \circ G_{s,t} \circ i(\un{11 \cdots 1}) = \can \]
and $(G_{s,t})_{(11 \cdots 1)}^{(11 \cdots 1)} = 1$, agreeing with \eqref{eq:111}.

%\comm{Got to here. 31/5}

\section{On the error(s) in our previous work} \label{sec:error}

Our intention in \cite{EW} was to define the diagrammatic category
$\HC$ and prove that it behaved nicely without needing to use the
functor $\FC$ to Soergel bimodules or any theorems about Soergel
bimodules.

Suppose however that there is a functor $\FC$ from $\HC$ to Soergel bimodules. Then one should have the following commutative diagram:
\begin{equation}
\begin{array}{c}
  \begin{tikzpicture}[scale=1.2,descr/.style={fill=white}]
\node (HC) at (0,2) {$\HC$};
\node (OmQ) at (0,0) {$\Omega_Q W$};
\node (SBim) at (2,2) {$\SBim$};
\node (StdBim) at (2,0) {$\StdBim_Q$};

\path [->,>=angle 90]
	(HC) edge node[right,descr] {$\Lambda$} (OmQ)
	(SBim) edge node[right,descr] {$(-) \otimes_R Q$} (StdBim)
	(HC) edge node[above, descr] {$\FC$} (SBim)
	(OmQ) edge node[above,descr] {$\FC_{\std}$} (StdBim);
      \end{tikzpicture}
\end{array}
  \end{equation}
The functor $\FC_{\std}$ was discussed in the introduction; it is
always a well-defined faithful functor, and it is an equivalence when
the realization is faithful. The functor $(-)\otimes_R Q$ is also
always well-defined, and it is faithful when hom spaces between
Soergel bimodules are free\footnote{ In fact, for faithfulness of
  $(-)\otimes_R Q$, we only need that hom spaces are torsion free. Torsion freeness is often much easier to
  establish. However for applications an essential role
  is played by Soergel's hom formula, which only makes sense
  when hom spaces are free.}
as right $R$-modules. One could use this
commutative diagram to attempt to define either $\FC$ or $\Lambda$ in
the presence of the other. For example, if $\Lambda$ is well-defined,
then $\FC_{\std} \circ \Lambda$ is a functor $\HC \to \StdBim$. If hom
spaces between Soergel bimodules are free then $(-)\otimes_R Q$ is
faithful, and one can lift $\FC_{\std} \circ \Lambda$ to a functor
$\FC$ if and only if the image of $\FC_{\std} \circ \Lambda$ lies
inside the image of $(-) \otimes_R Q$, i.e. if the map between
$Q$-bimodules can be defined without the use of fractions. This lift
is unique if it exists.

\begin{remark} 
  In \cite{SB} Soergel proves that hom spaces between
    Soergel bimodules are free when the realization is reflection faithful
  and the base ring $\Bbbk$ is an infinite field for which $2 \ne 0$.
 The assumption that the field is infinite can easily be removed (it is only
 used to identify functions with elements in the symmetric algebra). The other
 assumptions, particularly reflection faithfulness, appear more
 serious, and we do not know in what generality to expect them to be
 true. This is one of the motivations for studying the diagrammatic
 incarnation of the Hecke category, rather than Soergel bimodules. For
 an exciting recent approach which is more bimodule theoretic, see \cite{Abe}.
\end{remark}

In \cite{EW} we did not define the functor $\Lambda$ explicitly as we have done in this paper. Instead we defined the mixed calculus, and used it to argue that $\Lambda$ could
be defined explicitly (as in \S\ref{sec:heuristics}) using the mixed calculus. In our own unpublished computations we had computed the matrices associated to the univalent,
trivalent, and $2m$-valent vertices for $m \le 5$, and we checked all the relations involving these generators; this much we asserted in \cite{EW}. At that time, we did not
know the formula \eqref{eq:2} for the $2m$-valent vertex in general, and as such, we could not have checked the two color relations explicitly for all $m$. 

Instead, we argued in \cite{EW} that the two color relations were
satisfied by referencing the earlier work \cite{EDC}, which defined
and studied $\HC$ for dihedral groups. In doing so, we implicitly (and unwittingly) defined $\Lambda$ as the unique lift of $(-)\otimes_R Q \circ \FC$, where the functor $\FC$ was also
defined on the $2m$-valent vertex in \cite{EDC}. The reason this argument is faulty is that $\FC$ was only defined under certain assumptions on the realization (which we recall
below), while we claimed that $\Lambda$ was defined in greater generality. We also made an error in \cite[\S 5.3]{EW} in that we only assumed Demazure Surjectivity when we claimed
that $\FC$ was well-defined, forgetting that additional assumptions were used in \cite{EDC}.

To set the record straight, here is the list of assumptions that should be made about realizations in order to use various theorems, now that this paper has constructed $\Lambda$
in generality.

\subsection{Definition of $\HC$}

Which assumptions do we need for $\HC$ to be well defined? All the generators and relations in the category make sense without any assumptions, with the exception of the
Jones-Wenzl relation. For this relation to make sense we need the existence of a particular rotatable Jones-Wenzl projector in the two-colored Temperley-Lieb category, one for each finite dihedral parabolic subgroup. If these exist, then $\HC$ is well defined.

The existence and rotatability of Jones-Wenzl projectors is a relatively mild constraint on a realization, and is primarily a condition on the base ring $\Bbbk$. For example,
suppose that the restriction of the realization to a finite dihedral parabolic subgroup generated by $\{\reds,\bluet\}$ (with $m = m_{st} < \infty$) is a faithful representation of
that dihedral subgroup. Then $JW_{(m-1)_{\reds}}$ and $JW_{(m-1)_{\bluet}}$ are defined and rotatable so long as $\Bbbk$ is a field, as we prove in Theorem \ref{thm:rotatableoverfield}. Thus if a realization over a field is \emph{dihedrally faithful} (meaning that it is faithful after restriction to
any dihedral parabolic subgroup), then $\HC$ is well defined.

However, giving precise algebraic conditions for the existence of Jones-Wenzl projectors, and for their rotatability, is extremely subtle! Conditions have been
put forth in previous works \cite{EDC, ELib}, with some handwaving arguments that we no longer believe are justified (though many of these arguments suffice in all but the most
degenerate realizations). In \S\ref{sec:JW} we do the best we can at the moment, providing necessary conditions on $\Bbbk$. In particular, in \S\ref{subsec:rotinvce} we prove
(generalizing the argument of \cite{EDC}) that only even-balanced realizations can admit rotatable Jones-Wenzl projectors. Thus even-balancedness is a prerequisite to the
well-definedness of $\HC$.

\begin{remark} The condition that a Jones-Wenzl projector is negligible is much more common in the literature. If a projector is rotatable it is also negligible, but not vice
versa, see Remark \ref{rmk:negligiblevs}. The paper \cite{EDC} often
works under the assumption that the realization is dihedrally faithful
(as defined above). Under this assumption, when the base ring is a domain, the Jones-Wenzl projector in question is both
rotatable and negligible, and the realization is even-balanced (we show this in \S\ref{sec:JW}). In papers that followed, which did not make the dihedrally faithful assumption,
sometimes the rotatability assumption was erroneously replaced with weaker assumptions like negligibility or even-balancedness. \end{remark}

\begin{remark} In this paper, we give examples of realizations, all of which are not dihedrally faithful, and all of which have negligible Jones-Wenzl projectors $JW_{m-1}$, where: \begin{itemize}
\item The realization is not even-balanced, and the Jones-Wenzl projector is not rotatable (Example \ref{ex:funky3redux}).
\item The realization is even-balanced, and the Jones-Wenzl projector is not rotatable (Examples \ref{evenbalancedbad} and \ref{ex:bad6}). 
\item The realization is even-balanced, and the Jones-Wenzl projector is rotatable (Example  \ref{ex:notdihedralfaithfulstillrotatable}).
\end{itemize}
Because non-dihedrally faithful realizations with rotatable Jones-Wenzl projectors do exist, we make no assumptions of dihedral faithfulness in this paper.	\end{remark}

\begin{remark} Although one might attempt to define $\HC$ even in the
  absence of a rotatable Jones-Wenzl projector, it fails to satisfy
  enough crucial properties that we do not believe it deserves the
  name ``Hecke category.'' For example, its Grothendieck group is not
  isomorphic to the the Hecke algebra. \end{remark}
%See Remark \ref{dontcallmethat} for more discussion. 

\begin{remark} Note that Demazure surjectivity is not a necessary requirement for the well-definedness of $\HC$, or for the construction of the functor $\Lambda$ in this paper.
\end{remark}

\subsection{The Soergel categorification theorem}
Which assumptions on the realization do we need to assert that double leaves form a basis for $\HC$, and hence that the Soergel Categorification Theorem holds for $\HC$? The
argument that double leaves are linearly independent (because they remain so after applying $\Lambda$) requires no additional assumptions. For our diagrammatic arguments that
double leaves span, we use polynomial forcing extensively, and this requires that $R$ is a Frobenius extension over $R^s$ (with Frobenius trace $\partial_s$) for each simple
reflection $s$. This is equivalent to the assumption of Demazure
Surjectivity.

\subsection{Existence of $\FC$} \label{sec:existence}
Which assumptions on the realization do we need to define $\FC$? We
certainly need Demazure surjectivity, because the image of one of the
dots under $\FC$ requires the existence of 
dual bases for $R$ over $R^s$. Beyond this, the essential difficulty
is to provide an explicit formula for the image of the $2m$-valent
vertex directly.

Recent progress on this issue has been made by Abe \cite{Abe3}. Rather
than try to define an explicit formula for the morphism between Bott-Samelson
bimodules (which seems very hard), he instead checks that a morphism
given by the explicit formulas of Lemma \ref{lem:E} preserves the
Bott-Samelson bimodules. (In other words, he checks that a morphism
between localizations arises as the localization of a morphism.) His
proof requires \cite[Assumption 1.1]{Abe3}, which in the language of
two coloured quantum numbers (see \S\ref{subsec:2quantum}) is the
requirement that
\[
 {m_{st} \brack k}_s =  {m_{st} \brack k}_t = 0
\]
$k = 1, 2, \dots, m_{st}-1$ (and all $s, t \in S$). This assumption is
natural (and holds in all Kac-Moody realizations, as is easily
checked).

The original approach to defining the image of the $2m$-valent was
found in \cite{EDC}. There, the $2m$-valent vertex is described as an
explicit composition of more elementary 
morphisms between singular Soergel bimodules. This description has
many uses, but minimizing assumptions on the realization is not one of
them. These elementary morphisms are built 
using various Frobenius extension structures, and require that the
ring extensions $R^{s,t} \subset R^s, R^t \subset R$ are Frobenius
extensions. The underlying assumption\footnote{In \cite{EDC} one also
  makes an assumption called $(\star)$ from \cite[Definition
  1.7]{ESW}. This assumption is used to ensure that the diagrammatics
  which describe singular Soergel bimodules is well-behaved, but is
  not strictly necessary for defining the image of the $2m$-valent
  vertex.} is \emph{dihedral Demazure surjectivity}, which states that
$R$ is a Frobenius extension over $R^{s,t}$ with Frobenius trace
$\partial_{s,t}$. Dihedral Demazure surjectivity implies 
 that the action of the finite dihedral group generated by $\{s,t\}$
 is faithful, because if it factors through a smaller dihedral group
 then $\partial_{s,t} = 0$.

\section{On two-colored Jones-Wenzl projectors} \label{sec:JW}

In this chapter we develop some of the facts we need about Jones-Wenzl projectors; in particular, the two-colored Jones-Wenzl projectors which appear in the two-colored
Temperley-Lieb category $\TTL$. The literature on Jones-Wenzl projectors is quite rich, but less so in the two-colored case. Many of the facts in this chapter were stated in \cite{EDC} but without proof, by analogy with the uncolored case. We will prove the facts we need, and clearly
state some additional properties as conjectures\footnote{Most of these conjectures are within reach, but we did not feel it was appropriate for this paper.}. Starting in this chapter, we develop the needed technology to address the unbalanced case. We hope that this
chapter fills some gaps in the literature.

\subsection{Two-colored quantum numbers} \label{subsec:2quantum}

Let $A = \ZM[x_s,x_t]$ be the integral polynomial ring in two variables. We define the \emph{two-colored quantum numbers}, which are elements of $A$, as follows. First set $[2]_s = x_s$ and $[2]_t = x_t$ and $[1]_s = [1]_t = 1$. Now define the remaining quantum numbers inductively by the formulas
\begin{equation} \label{eq:recursivedefnquantumnumber} [n+1]_t = [n]_s [2]_t - [n-1]_t, \qquad [n+1]_s = [n]_t [2]_s - [n-1]_s. \end{equation}
For example, $[3]_s = [3]_t = [2]_s [2]_t - 1$, and $[4]_s = [2]_s^2 [2]_t - 2[2]_s$. One can also use these same formulas to define $[n]_s$ and $[n]_t$ for $n \le 0$, and one can easily confirm that
\begin{equation} [0]_s = [0]_t = 0, \end{equation}
\begin{equation} \label{eq:negnumbers} [-n]_s = -[n]_s, \qquad [-n]_t = -[n]_t. \end{equation}
Indeed, these formulas hold for $n = 0$ and $n=1$, and follow by induction since both $[n]$ and $-[-n]$ satisfy the recursive formulas \eqref{eq:recursivedefnquantumnumber}.

A straightforward induction proves that
\begin{equation} [k]_s = [k]_t \quad \text{whenever $k$ is odd}, \end{equation}
so we may merely write it as $[k]$ rather than $[k]_s$ or $[k]_t$.

Many facts about ordinary quantum numbers will generalize easily to the two-colored case. For example, if $k$ and $n$ are odd one can prove that
\begin{subequations} \label{subeq:ktimesn}
\begin{equation} [k][n] = [k+n-1] + [k+n-3] + \cdots + [k-n+3] + [k-n+1]. \end{equation}
Meanwhile, if $k$ is even and $n$ is odd then
\begin{equation} [k]_s [n] = [k+n-1]_s + [k+n-3]_s + \cdots + [k-n+3]_s + [k-n+1]_s. \end{equation}
Finally, if both $k$ and $n$ are even then
\begin{equation} [k]_s [n]_t = [k+n-1] + [k+n-3] + \cdots + [k-n+3] + [k-n+1]. \end{equation}
\end{subequations}
All these results can be proven by induction en masse. An additional consequence is that $[k]_s$ divides $[n]_s$ whenever $k$ divides $n$. For example, if $k$ is even then $[2k]_s = [k]_s ([k+1] - [k-1])$, and if $k$ is odd then $[2k]_s = [k]([k+1]_s - [k-1]_s)$. We leave these elementary properties as exercises for the reader.

Because of the divisibility of quantum numbers, fractions like $\frac{[8]_s}{[2]_s}$ can be viewed as polynomials in $A$ as well. In fact, if $k$ is even then
\begin{equation} \label{kover2} \frac{[k]_s}{[2]_s} = \frac{[k]_t}{[2]_t} = [k-1]-[k-3]+[k-5]-\cdots \pm 1. \end{equation}
Thus the dependence of an even quantum number on the coloring is no more than a factor of either $[2]_s$ or $[2]_t$.

\begin{defn} Define the \emph{two-colored quantum binomial coefficients} as
\begin{equation} {n \brack k}_s := \frac{[n]_s [n-1]_s \cdots [n-k+1]_s}{[1]_s [2]_s \cdots [k]_s} \in A. \end{equation}
\end{defn}

Because $A$ is a UFD, one can use standard arguments to prove that binomial coefficients are indeed elements of $A$.

\subsection{Two-colored quantum numbers modulo $[m]$} \label{subsec:2quantummodulo}

Let $\Bbbk$ be an $A$-algebra. The \emph{two-colored quantum numbers} in $\Bbbk$ are the specializations of the two-colored quantum numbers in $A$. For any realization of a Coxeter
group with base ring $\Bbbk$, and any pair of simple reflections $s
\ne t \in S$, we obtain an $A$-algebra structure on $\Bbbk$ by
specializing \begin{equation} \label{eq:realizationspecialization} x_s
  \mapsto -a_{st} = - \a_s^\vee(\a_t),
\qquad x_t \mapsto -a_{ts} = -\a_t^\vee(\a_s). \end{equation} One of the requirements of a realization is that both $[m]_s$ and $[m]_t$ specialize to $0$ when $m = m_{st} <
\infty$.

Let us discuss the consequences of imposing the relation $[m]_s = [m]_t = 0$ on the generic ring $A$. Let $A_m = A / ([m]_s,[m]_t)$. Note that $A_m$ is not a domain and not local when $m$ is not prime. For example, $[10]_s$ has three irreducible factors: $[2]_s$, $[5]$, and $\frac{[10]_s}{[2]_s [5]} = [5] - 2[3] + 2$. In particular, we can specialize further from $A_m$ to $A_k$ whenver $k$ divides $m$.

In the same way one proves \eqref{eq:negnumbers} one can deduce that
\begin{equation} \label{eq:cancelthem} [m+k]_s = -[m-k]_s, \qquad [m+k]_t = -[m-k]_t. \end{equation}
Using the inductive definition of quantum numbers, but working down from $m$ rather than up from $0$, one can prove that
\begin{equation} \label{eq:m-k} [m-1]_s [k] = [m-k]_s \text{ for } k \text{ odd,} \quad [m-1]_s [k]_t = [m-k]_t \text{ for } k \text { even.} \end{equation}
One can also prove this by applying \eqref{subeq:ktimesn} and cancelling terms with \eqref{eq:cancelthem}.
In particular (taking $k = m-1$) we have
\begin{equation} \label{eq:lambdainverse} [m] = 0, m \text{ odd} \implies [m-1]_s [m-1]_t = 1, \end{equation}
\begin{equation} \label{eq:lambdainverseeven} [m]_s = [m]_t = 0, m \text{ even} \implies [m-1]^2 = 1. \end{equation}
Either way, $[m-1]_s$ and $[m-1]_t$ are invertible.

\begin{remark} If $m$ is odd then $[2]_s$ and $[2]_t$ are invertible, since $[2]_s$ divides $[m-1]_s$. Whenever $[2]_s$ and $[2]_t$ are invertible, $\frac{[2]_s}{[2]_t}$ is an invertible scalar factor which gives the ``ratio'' between $[2k]_s$ and $[2k]_t$ for any $k \ge 0$, in the sense that $[2k]_t \frac{[2]_s}{[2]_t} = [2k]_s$, see \eqref{kover2}. \end{remark}
	
\subsection{Faithful and unbalanced realizations} \label{subsec:unbalanced}

\begin{defn} A realization is called \emph{dihedrally faithful} if, for every $s \ne t \in S$ with $m_{st} \ne \infty$, the action on $\hg^*$ of the finite dihedral group generated by $\{s,t\}$ is faithful. \end{defn}

\begin{defn} \label{defn:balanced} A realization is called \emph{balanced for the pair
    $\{s,t\}$} if $[m-1]_s = [m-1]_t = 1$. The realization is called
  \emph{balanced} if it is balanced for every pair which generates a
  finite dihedral subgroup. It is \emph{even-balanced} if it is balanced
  for every pair with $m_{st}$ even, and \emph{odd-balanced} if it is balanced for every pair with $m_{st}$ odd.
\end{defn}

\begin{remark} \label{rmk:funkydihedral} Note that a balanced realization is not necessarily dihedrally faithful. For example, if $a_{st} = a_{ts} = -1$ then $[2] = 1$ and $[3] =
0$, so one has a balanced faithful realization if $m=3$. Continuing, $[4] = [5] = -1$ and $[6]=0$ so one has an unbalanced non-faithful realization if $m=6$. Continuing, $[7] = [8]
= 1$ and $[9]=0$ so one has a balanced non-faithful realization if $m=9$. \end{remark}

It was proven in \cite[Claim 3.5, and page 63]{EDC} that $[k]_s = [k]_t = 0$ implies that $(st)^k$ acts by the identity on the span of $\alpha_s$ and $\alpha_t$. In particular, if a realization is dihedrally faithful then either $[k]_s \ne 0$ or $[k]_t \ne 0$ for any $k < m_{st}$. (Of course $[m]_s = [m]_t = 0$ for $m=m_{st}$, by the definition of a realization). It was also claimed in \cite{EDC} that the converse holds when $\Bbbk$ is a domain. This result has a slight error, because it neglects to carefully study the case when $[2]_s = 0$ but $[2]_t \ne 0$. For completeness, we prove the correct result here, after a lemma.

\begin{lem} \label{lem:weirdcasequantums} Suppose that $[2]_s = 0$. Then for all integers $k$ we have
\begin{equation} \label{weirdcasequantums} [2k+1] = (-1)^k, \qquad [2k]_s = 0, \qquad [2k]_t = (-1)^{k-1} k [2]_t. \end{equation}
Note the appearance of the ordinary integer $k$ as a factor of $[2k]_t$.
\end{lem}

\begin{proof} These equations are easily shown to satisfy the recursion \eqref{eq:recursivedefnquantumnumber}. \end{proof}
	
\begin{remark} One is often interested in the case when $[2]_s = 0$ but $[2]_t \ne 0$, as happens for example in type $B_2$ in characteristic $2$, or in type $G_2$ in characteristic $3$. Thus Lemma \ref{lem:weirdcasequantums} is relevant in small characteristic Kac-Moody realizations. \end{remark}

\begin{thm} \label{thm:whenfaithful} Consider a realization of a dihedral group over the base ring $\Bbbk$. If $[k]_s = [k]_t = 0$ then $(st)^k$ acts as the identity on the span of $\alpha_s$ and $\alpha_t$. If $\Bbbk$ is a domain, and $(st)^k$ acts as the identity on the span of $\alpha_s$ and $\alpha_t$, then exactly one of the following options holds \begin{itemize}
	\item $[k]_s = [k]_t = 0$, or
	\item the characteristic of $\Bbbk$ is $2$, $[2]_s = [2]_t = 0$, and $k$ is odd,
	\item the characteristic of $\Bbbk$ is $2$, exactly one of $[2]_s$ and $[2]_t$ is zero, and $k$ is equivalent to $2$ modulo $4$.
\end{itemize}
\end{thm}

\begin{proof} We recall a computation from \cite[Claim 3.5 and page 63]{EDC}. The action of $(st)^k$ on the span of
$\a_s$ and $\a_t$ is given by the matrix
\begin{equation} \label{actionofstk} (st)^k = \left( \begin{array}{cc} [2k+1] & -[2k]_s \\ \left[2k\right]_t & -[2k-1] \end{array} \right). \end{equation}
If $[k]_s = [k]_t = 0$ then one can use \eqref{eq:cancelthem} to deduce that $[2k]_s = [2k]_t = 0$, and $[2k+1] = 1 = -[2k-1]$. Hence the matrix above is the identity matrix.

Conversely, suppose that the matrix above is the identity, and that $\Bbbk$ is a domain. Then $[2k]_s = [2k]_t = 0$, and $[2k-1] = -1$. Using \eqref{eq:m-k} with $2k$ replacing $m$, we deduce (regardless of whether $k$ is even or odd, c.f. \eqref{eq:m-keven}) that
\begin{equation} [k]_s = [2k-1][k]_s = -[k]_s, \qquad [k]_t = [2k-1][k]_t = -[k]_t, \end{equation}
and
\begin{gather} [k+1]_t = [2k-1][k-1]_t = -[k-1]_t, \\ [k+1]_s = [2k-1] [k-1]_s = -[k-1]_s. \end{gather}
In particular, if $2 \ne 0$ then $[k]_s = [k]_t = 0$.

For the rest of the proof we assume $2 = 0$. We also have
\begin{equation} [2]_s [k]_t = [k+1]_s + [k-1]_s = 0, \qquad [2]_t [k]_s = [k+1]_t + [k-1]_t = 0. \end{equation}
If $[2]_s \ne 0$ and $[2]_t \ne 0$ then $[k]_s = [k]_t = 0$.

Suppose that $[2]_s = [2]_t = 0$. If $k$ is even then $[k]_s = [k]_t = 0$. If $k$ is odd, then $[k]_s = [k]_t = 1 \ne 0$ by \eqref{weirdcasequantums} in characteristic $2$, but nonetheless the matrix from \eqref{actionofstk} is the identity. 

The only remaining case is that exactly one of $[2]_s$ and $[2]_t$ is zero. Without loss of generality, $[2]_s = 0$ and $[2]_t \ne 0$. By \eqref{weirdcasequantums} in characteristic $2$, $[2k]_t = k [2]_t = 0$, so $k = 2 \ell$ is even. Then $[k]_t = (-1)^{\ell - 1} \ell [2]_t$, which (in characteristic $2$) is zero if and only if $\ell$ is even. Thus $[k]_t$ is nonzero if and only if $k$ is equivalent to $2$ modulo $4$.
\end{proof}

\subsection{Two-colored Temperley-Lieb}

Let $\TTL = \TTL_A$ denote the two-colored Temperley-Lieb category with base ring $A$. We remind the reader of the definition.

\begin{defn} The objects of of $\TTL$ are associated with finite alternating sequences $(..., \reds, \bluet, \reds, \bluet, \ldots)$. For an integer $n \ge 0$ we let ${}_{\reds} n$ denote the sequence of length $n+1$ which begins with $\reds$, and $n_{\reds}$ the sequence of length $n+1$ which ends in $\reds$, and similarly for $\bluet$. Thus 
\[ {}_{\reds} 2 = (\reds,\bluet,\reds) = 2_{\reds} \quad \text{ and } \quad {}_{\reds} 3 = (\reds, \bluet, \reds, \bluet) = 3_{\bluet}. \]
We visualize this object as a sequence of $n$ dots on a line; removing these points from the line we have $n+1$ connected components or regions, which are colored with alternating colors. 

There are no morphisms between $n_{\reds}$ and $n'_{\bluet}$ in either direction, or between ${}_{\reds} n$ and ${}_{\bluet} n'$, for any $n, n'$. If $n$ and $m$ have the same parity then
$\Hom_{\TTL}(n_{\reds},m_{\reds})$ has an $A$-basis given by two-colored $(n,m)$-crossingless matchings which have the color $\reds$ on the far right. Recall that a two-colored
crossingless matching is a crossingless matching where the regions are colored either $\reds$ or $\bluet$, and adjacent regions have different colors. By convention, an $(n,m)$-crossingless matching has $n$ boundary points on bottom and $m$ boundary points on top.

Composition of morphisms is as in the Temperley-Lieb category. A blue circle inside a red region evaluates to $-x_{\reds}$ while a red circle inside a blue region evaluates to
$-x_{\bluet}$. \end{defn}

For any commutative $A$-algebra $\base$ we define $\TTL_{\base}$ to be the base change of $\TTL_A$ from $A$ to $\base$. Morphism spaces still have a basis given by crossingless
matchings.

\begin{defn} We let $\TTL_{\base}^\Kar$ denote the Karoubi envelope of the additive envelope of $\TTL_{\base}$. \end{defn}	

For any $\base$, the category $\TTL_{\base}$ is an \emph{object adapted cellular category}, as originally defined in \cite{ELauda}. See \cite[Chapter 11]{SBook} for background reading on
object adapted cellular categories (the definition is found in \cite[11.3.1]{SBook}). The object adapted structure on the Temperley-Lieb category is discussed in \cite[Example
11.31]{SBook}. We assume the reader is familiar with this background material below. We use the following general result about object adapted cellular categories, whose proof
morally derives from \cite[Theorem 6.25]{EW}, and was generalized in \cite[Proposition 2.24]{ELauda} and \cite[Chapter 11.3.4]{SBook}.

\begin{thm} \label{thm:TTLOACC} Suppose that $\base$ is an $A$-algebra, a commutative domain, and a Henselian local ring. \begin{enumerate}
\item $\TTL_{\base}$ is an object-adapted cellular category with its cellular basis of crossingless matchings.
\item In any decomposition of the identity of $n_{\reds}$ (resp. $n_{\bluet}$) into indecomposable orthogonal idempotents
\[ 1_{n_{\reds}} = e_1 + \cdots + e_d, \]
there is a unique idempotent in the sum for which the coefficient of the identity (in the basis of crossingless matchings) is nonzero. This coefficient is $1$.

For each $n \ge 0$ we pick such an indecomposable idempotent and denote it $e_{n_{\reds}}^{\base}$ (resp. $e_{n_{\bluet}}^{\base}$).
\item Let $V_{n_{\reds}}^{\base}$ denote the image of $e_{n_{\reds}}^{\base}$, an object in $\TTL_{\base}^\Kar$, and similarly for $V_{n_{\bluet}}^{\base}$. Up to isomorphism, the indecomposable objects in $\TTL_{\base}^\Kar$ are precisely enumerated by $\{V_{n_{\reds}}^{\base},V_{n_{\bluet}}^{\base}\}_{n \ge 0}$.
\item For any idempotent $e \in \End(n_{\reds})$, its image $V$ in $\TTL_{\base}^\Kar$ decomposes as
\begin{equation} \label{eq:pluslowerterms} V \cong (V_{n_{\reds}}^{\base})^{\oplus \delta} \oplus
    \bigoplus_{0 \le k = n - 2l} (V_{k_{\reds}}^{\base})^{\oplus \mu_{k}} \end{equation}
for some multiplicites $\mu_k$. The multiplicity $\delta$ of $V_{n_{\reds}}^{\base}$ is equal to zero if $e$ has zero coefficient of the identity, and is equal to $1$ if $e$ has nonzero coefficient of the identity.
\end{enumerate}
\end{thm}

The summand $V_{n_{\reds}}^{\base}$ inside $n_{\reds}$ is called the \emph{top summand}, and $e_{n_{\reds}}^{\base}$ the \emph{top idempotent}, because all other summands in $n_{\reds}$ are appear as summands of $k_{\reds}$ for $k < n$. While the top idempotent itself may depend on the choice of decomposition of $1_{n_{\reds}}$, its image $V_{n_{\reds}}^{\base}$ is independent of this choice up to isomorphism.

\begin{proof} The theorem is a straightforward adaption of \cite[Proposition 2.24]{ELauda} and \cite[Example 11.31]{SBook}. \end{proof}

Note that the behavior of $e_{n_{\reds}}^{\base}$ and $V_{n_{\reds}}^{\base}$ is quite dependent on the choice of the base ring $\base$. For example, $2_{\reds}$ has a nontrivial top summand (the complement of a summand isomorphic to $0_{\reds}$) when $[2]_s$ is invertible, and is indecomposable otherwise.

\subsection{Two-colored Jones-Wenzl projectors}

Until otherwise stated we fix an $A$-algebra $\base$ as in Theorem \ref{thm:TTLOACC}, and omit $\base$ from the notation for idempotents and summands.

\begin{defn} Fix $n \ge 0$. If $e_{n_{\reds}}$ is a top idempotent which satisfies the following property, then we call it a \emph{Jones-Wenzl projector} and denote it $JW_{n_{\reds}}$:
\begin{equation} \label{eq:orthokprecomp} e_{n_{\reds}} \circ
  \Hom(k_{\reds},n_{\reds}) = 0, \quad \text{ for all } 0 \le k <
  n. \end{equation}
\end{defn}

We now state some properties of Jones-Wenzl projectors. To be walked through the proofs of these facts, see \cite[Exercise 9.25]{SBook}.

Whether $JW_{n_{\reds}}$ exists depends on the choice of $\base$. If it exists, then it is unique. Note that \eqref{eq:orthokprecomp}, which states that the image of
$JW_{n_{\reds}}$ is orthogonal to all the objects $k_{\reds}$ for $k < n$, is equivalent to the condition that $JW_{n_{\reds}}$ is orthogonal to the objects $V_{k_{\reds}}$ for $k
< n$, using \eqref{eq:pluslowerterms}.

Every $(n,n)$-crossingless matching except for the identity has a cap on bottom and a cup on top. Similarly, every $(k,n)$-crossingless matching for $k < n$ has a cup on
top. For an endomorphism $f \in \End(n_{\reds})$, being orthogonal to $k_{\reds}$ for $k < n$ as in \eqref{eq:orthokprecomp} is equivalent to being orthogonal to all cups which
could be placed below, i.e. \[ f \circ \text{cup} = 0,\] which we call \emph{death by cups}. So $e_{n_{\reds}}$ is a Jones-Wenzl projector if and only if it is satisfies death by
cups. Conversely, any endomorphism which satisfies death by cups and whose coefficient of the identity equals $1$ is an idempotent, and hence is a Jones-Wenzl projector.

Jones-Wenzl projectors are also killed by all caps on top, and death by caps can replace death by cups as an equivalent definition of a Jones-Wenzl projector. In other words,
\begin{equation} \label{eq:orthokpostcomp} \Hom(n_{\reds},k_{\reds}) \circ e_{n_{\reds}} = 0, \quad \text{ for all } 0 \le k < n. \end{equation} When $JW_{n_{\reds}}$
exists, then any endomorphism of $n_{\reds}$ which satisfies death by caps or death by cups is a scalar multiple of $JW_{n_{\reds}}$, and this scalar is determined by the
coefficient of the identity.

Using death by caps and/or cups, it is easy to prove that the endomorphism ring of the image of $JW_{n_{\reds}}$ (namely $V_{n_{\reds}}$) is precisely $\base$. All non-identity
crossingless matchings in $\End(n_{\reds})$ are killed when composed with $JW_{n_{\reds}}$.

By swapping the roles of $\reds$ and $\bluet$, the same things can be said about the morphism $JW_{n_{\bluet}}$. Of course, $JW_{{}_{\reds} n}$ is equal to $JW_{n_{\reds}}$ or
$JW_{n_{\bluet}}$, depending on the parity of $n$.

The horizontal reflection of a Jones-Wenzl projector still satisfies the same conditions. Thus the horizontal reflection of $JW_{n_{\reds}}$ is $JW_{{}_{\reds} n}$. Similarly,
the vertical reflection still satisfies death by cups instead of death by caps, so it is still a Jones-Wenzl projector. Thus the vertical reflection of $JW_{n_{\reds}}$ is
$JW_{n_{\reds}}$. The rotation of a Jones-Wenzl projector by 180 degrees is thus also a Jones-Wenzl projector (composing these two reflections).

We say that \emph{$JW_n$ exists} when both $JW_{n_{\reds}}$ and $JW_{n_{\bluet}}$ exist. It is possible for $JW_n$ to exist even when $JW_k$ does not exist for various $k < n$.
Clearly $JW_0$ and $JW_1$ exist. It is easy to confirm by hand that $JW_2$ exists if and only if $[2]_s$ and $[2]_t$ are invertible, and a reasonable exercise to confirm that
$JW_3$ exists if and only if $[3]$ is invertible. It might be that $[2]_s = [2]_t = 0$ and thus $[3] = -1$, in which case $JW_3$ exists but $JW_2$ does not.

\subsection{Two-colored Jones-Wenzl projectors: the generic case}

When all quantum numbers are invertible then all Jones-Wenzl projectors exist. This is a familiar story, which we elaborate on here mostly to pin down the colorings on quantum
numbers precisely.

\begin{defn}
Let $n \ge 1$. The \emph{partial trace} $\ptr(f)$ of an endomorphism
$f$ of $n_{\reds}$ is an endomorphism of $(n-1)_{\bluet}$ defined as
follows.
\begin{equation}
  \label{eq:ptrace}
  \ptr(f) = 
    \begin{array}{c}
      \begin{tikzpicture}[xscale=.6,yscale=1]
        \begin{scope}
          \clip (-1.8,1) rectangle (2.5,-1);
  \draw[fill=blue!20!white] (.5,2) rectangle (3,-2);
  \draw[fill=red!20!white] (1,.6) to[out=90,in=180] (1.5,.8)
  to[out=0,in=90] (2,0) to[out=-90,in=0] (1.5,-.8) to[out=180,in=-90]
  (1,-.6) to (1,.6);
  \draw[fill=purple!20!white] (-1,2) rectangle (-2,-2);        
  \draw[fill=white] (-1.4,.6) rectangle (1.4,-.6);
  \node at (0,0) {$f$};
  \node at (-.25,.8) {$\dots$};
  \node at (-.25,-.8) {$\dots$};
  \end{scope}
  \draw[dashed] (-1.8,1) to (2.5,1);
  \draw[dashed] (-1.8,-1) to (2.5,-1);
\end{tikzpicture}
      \end{array}
\end{equation}
(The purplish color on the left is meant to represent either red or blue, depending on parity.) Let $\ptr_1(f)$ denote the coefficient of the identity when $\ptr(f)$ is expanded in the basis of
crossingless matchings.
\end{defn}

Now suppose that $JW_n$ exists, and consider $\ptr(JW_{n_{\reds}})$. This morphism is clearly orthogonal to all cups and caps. If $JW_{n-1}$ exists then \begin{equation}
\label{ptrvsptr1} \ptr(JW_{n_{\reds}}) = \ptr_1(JW_{n_{\reds}}) \cdot JW_{(n-1)_{\bluet}}. \end{equation} Conversely, if $\ptr_1(JW_{n_{\reds}})$ is invertible then
$JW_{(n-1)_{\bluet}}$ exists, by this same formula. On the other hand, if $JW_{n-1}$ exists and $\ptr_1(JW_{n_{\reds}}) = 0$ then
$\ptr(JW_{n_{\reds}}) = 0$ by \eqref{ptrvsptr1}. If $JW_{n-1}$ does not exist, it is possible that $\ptr_1(JW_{n_{\reds}}) = 0$ but $\ptr(JW_{n_{\reds}}) \ne 0$, see Remark
\ref{rmk:ptrptr1notequiv}.

Assuming that $JW_n$ exists and $\ptr_1(JW_{n_{\reds}})$ is invertible,
$JW_{(n+1)_{\bluet}}$ also exists, and can be defined by the recursive
formula:
\begin{equation}
  \label{JWrecursion}
                  \begin{array}{c}
        \begin{tikzpicture}[xscale=.3,yscale=0.25]
          \begin{scope}
            \clip (-3,3.5) rectangle (3,-3.5);
            \draw[fill=blue!20!white] (0,4) rectangle (3,-4);
            \draw[fill=red!20!white] (1,4) rectangle (2,-4);
         %    \draw[fill=red!20!white] (1,4) to (1,1) to[out=-90,in=180] (1.5,0.5) to[out=0,in=-90] (2,1) to (2,4) to (1,4);
         % \draw[fill=red!20!white] (1,-4) to (1,-1) to[out=90,in=180] (1.5,-0.5) to[out=0,in=90] (2,-1) to (2,-4) to (1,-4);            
         \draw[fill=purple!20!white] (-4,4) rectangle (-2,-4);
         \draw[fill=white] (-2.5,2) rectangle (2.5,-2);
         \node at (0,0) {$JW_{(n+1)_{\color{blue}t}}$};
%         \node at (-1,0) {$\dots$};
         \node at (-1,2.7) {$\dots$};
         \node at (-1,-2.7) {$\dots$};
       \end{scope}
       \draw[dashed] (-3,3.5) to (3,3.5);
       \draw[dashed] (-3,-3.5) to (3,-3.5);
     \end{tikzpicture}
                  \end{array}
                  =
            \begin{array}{c}
        \begin{tikzpicture}[xscale=.3,yscale=0.25]
          \begin{scope}
            \clip (-3,3.5) rectangle (2.5,-3.5);
            \draw[fill=blue!20!white] (0,4) rectangle (3,-4);
            \draw[fill=red!20!white] (1,4) rectangle (2,-4);
         %    \draw[fill=red!20!white] (1,4) to (1,1) to[out=-90,in=180] (1.5,0.5) to[out=0,in=-90] (2,1) to (2,4) to (1,4);
         % \draw[fill=red!20!white] (1,-4) to (1,-1) to[out=90,in=180] (1.5,-0.5) to[out=0,in=90] (2,-1) to (2,-4) to (1,-4);            
         \draw[fill=purple!20!white] (-4,4) rectangle (-2,-4);
         \draw[fill=white] (-2.5,2) rectangle (1.5,-2);
         \node at (-.5,0) {$JW_{n_{\color{red}s}}$};
%         \node at (-1,0) {$\dots$};
         \node at (-1,2.7) {$\dots$};
         \node at (-1,-2.7) {$\dots$};
       \end{scope}
       \draw[dashed] (-3,3.5) to (2.5,3.5);
       \draw[dashed] (-3,-3.5) to (2.5,-3.5);
     \end{tikzpicture}
      \end{array}
        - \ptr_1(JW_{n_{\color{red}s}})^{-1}
      \begin{array}{c}
        \begin{tikzpicture}[xscale=.3,yscale=0.25]
          \begin{scope}
            \clip (-3,3.5) rectangle (2.5,-3.5);
            \draw[fill=blue!20!white] (0,4) rectangle (3,-4);
                     \draw[fill=red!20!white] (1,4) to (1,1) to[out=-90,in=180] (1.5,0.5) to[out=0,in=-90] (2,1) to (2,4) to (1,4);
         \draw[fill=red!20!white] (1,-4) to (1,-1) to[out=90,in=180] (1.5,-0.5) to[out=0,in=90] (2,-1) to (2,-4) to (1,-4);            
         \draw[fill=purple!20!white] (-4,4) rectangle (-2,-4);
         \draw[fill=white] (-2.5,3) rectangle (1.5,1);
         \draw[fill=white] (-2.5,-3) rectangle (1.5,-1);
         \node at (-.5,2) {$JW_{n_{\color{red}s}}$};
         \node at (-.5,-2) {$JW_{n_{\color{red}s}}$};
         \node at (-1,0) {$\dots$};
         \node at (-1,3.25) {$\dots$};
         \node at (-1,-3.25) {$\dots$};
       \end{scope}
       \draw[dashed] (-3,3.5) to (2.5,3.5);
       \draw[dashed] (-3,-3.5) to (2.5,-3.5);
     \end{tikzpicture}
      \end{array}
\end{equation}
Meanwhile, if $JW_n$ exists but $\ptr_1(JW_{n_{\reds}})$ is not invertible, one can deduce
that $JW_{(n+1)_{\bluet}}$ can not exist. It would have to be defined by a formula just like \eqref{JWrecursion}, but there is no choice of coefficients which would work.

The standard way which one computes $\ptr(JW_n)$ is to assume that $JW_{n-1}$ exists, and that $JW_n$ is constructed by \eqref{JWrecursion} for $n-1$. This allows one to deduce that
\begin{equation} \ptr(JW_{n_{\reds}}) = (-[2]_{\bluet} + \ptr_1(JW_{(n-1)_{\bluet}})^{-1}) \cdot JW_{(n-1)_{\bluet}} \end{equation}
giving the recursive formula
\begin{equation} \label{ptr1recursive} \ptr_1(JW_{n_{\reds}}) = (-[2]_{\bluet} + \ptr_1(JW_{(n-1)_{\bluet}})^{-1}). \end{equation}
Thus assuming that all $JW_k$ are defined for $0 \le k<n$, one can solve the recursion and compute that
\begin{equation} \label{ptr1solved} \ptr_1(JW_{n_{\reds}}) = \frac{-[n+1]_t}{[n]_s}. \end{equation}

Let us summarize.

\begin{thm} \label{thm:JWgeneric} Suppose that $\base$ is an commutative $A$-algebra for which $[k]_s$ and $[k]_t$ are invertible for all $k \ge 1$. Then $JW_n$ exists for all $n$, and
  \begin{equation} \label{JWrecursionredux}
%  \label{JWrecursion}
                  \begin{array}{c}
        \begin{tikzpicture}[xscale=.3,yscale=0.25]
          \begin{scope}
            \clip (-3,3.5) rectangle (3,-3.5);
            \draw[fill=blue!20!white] (0,4) rectangle (3,-4);
            \draw[fill=red!20!white] (1,4) rectangle (2,-4);
         %    \draw[fill=red!20!white] (1,4) to (1,1) to[out=-90,in=180] (1.5,0.5) to[out=0,in=-90] (2,1) to (2,4) to (1,4);
         % \draw[fill=red!20!white] (1,-4) to (1,-1) to[out=90,in=180] (1.5,-0.5) to[out=0,in=90] (2,-1) to (2,-4) to (1,-4);            
         \draw[fill=purple!20!white] (-4,4) rectangle (-2,-4);
         \draw[fill=white] (-2.5,2) rectangle (2.5,-2);
         \node at (0,0) {$JW_{(n+1)_{\color{blue}t}}$};
%         \node at (-1,0) {$\dots$};
         \node at (-1,2.7) {$\dots$};
         \node at (-1,-2.7) {$\dots$};
       \end{scope}
       \draw[dashed] (-3,3.5) to (3,3.5);
       \draw[dashed] (-3,-3.5) to (3,-3.5);
     \end{tikzpicture}
                  \end{array}
                  =
            \begin{array}{c}
        \begin{tikzpicture}[xscale=.3,yscale=0.25]
          \begin{scope}
            \clip (-3,3.5) rectangle (2.5,-3.5);
            \draw[fill=blue!20!white] (0,4) rectangle (3,-4);
            \draw[fill=red!20!white] (1,4) rectangle (2,-4);
         %    \draw[fill=red!20!white] (1,4) to (1,1) to[out=-90,in=180] (1.5,0.5) to[out=0,in=-90] (2,1) to (2,4) to (1,4);
         % \draw[fill=red!20!white] (1,-4) to (1,-1) to[out=90,in=180] (1.5,-0.5) to[out=0,in=90] (2,-1) to (2,-4) to (1,-4);            
         \draw[fill=purple!20!white] (-4,4) rectangle (-2,-4);
         \draw[fill=white] (-2.5,2) rectangle (1.5,-2);
         \node at (-.5,0) {$JW_{n_{\color{red}s}}$};
%         \node at (-1,0) {$\dots$};
         \node at (-1,2.7) {$\dots$};
         \node at (-1,-2.7) {$\dots$};
       \end{scope}
       \draw[dashed] (-3,3.5) to (2.5,3.5);
       \draw[dashed] (-3,-3.5) to (2.5,-3.5);
     \end{tikzpicture}
      \end{array}
+        \frac{[n+1]_t}{[n]_s}
      \begin{array}{c}
        \begin{tikzpicture}[xscale=.3,yscale=0.25]
          \begin{scope}
            \clip (-3,3.5) rectangle (2.5,-3.5);
            \draw[fill=blue!20!white] (0,4) rectangle (3,-4);
                     \draw[fill=red!20!white] (1,4) to (1,1) to[out=-90,in=180] (1.5,0.5) to[out=0,in=-90] (2,1) to (2,4) to (1,4);
         \draw[fill=red!20!white] (1,-4) to (1,-1) to[out=90,in=180] (1.5,-0.5) to[out=0,in=90] (2,-1) to (2,-4) to (1,-4);            
         \draw[fill=purple!20!white] (-4,4) rectangle (-2,-4);
         \draw[fill=white] (-2.5,3) rectangle (1.5,1);
         \draw[fill=white] (-2.5,-3) rectangle (1.5,-1);
         \node at (-.5,2) {$JW_{n_{\color{red}s}}$};
         \node at (-.5,-2) {$JW_{n_{\color{red}s}}$};
         \node at (-1,0) {$\dots$};
         \node at (-1,3.25) {$\dots$};
         \node at (-1,-3.25) {$\dots$};
       \end{scope}
       \draw[dashed] (-3,3.5) to (2.5,3.5);
       \draw[dashed] (-3,-3.5) to (2.5,-3.5);
     \end{tikzpicture}
      \end{array}
  \end{equation} and \eqref{ptr1solved} holds.
\end{thm}

\begin{remark} Note that there are two-colored analogs of other
  formulas for the Jones-Wenzl formula as well. For example, there is
  a two-colored single clasp expansion:  
      \begin{equation} \label{eq:singleclasp}
%  \label{JWrecursion}
                  \begin{array}{c}
        \begin{tikzpicture}[xscale=.45,yscale=.4]
          \begin{scope}
            \clip (-3,3.5) rectangle (3,-3.5);
            \draw[fill=blue!20!white] (0,4) rectangle (3,-4);
            \draw[fill=red!20!white] (1,4) rectangle (2,-4);
         %    \draw[fill=red!20!white] (1,4) to (1,1) to[out=-90,in=180] (1.5,0.5) to[out=0,in=-90] (2,1) to (2,4) to (1,4);
         % \draw[fill=red!20!white] (1,-4) to (1,-1) to[out=90,in=180] (1.5,-0.5) to[out=0,in=90] (2,-1) to (2,-4) to (1,-4);            
         \draw[fill=purple!20!white] (-4,4) rectangle (-2,-4);
         \draw[fill=white] (-2.5,2) rectangle (2.5,-2);
         \node at (0,0) {$JW_{(n+1)_{\color{blue}t}}$};
%         \node at (-1,0) {$\dots$};
         \node at (-1,2.7) {$\dots$};
         \node at (-1,-2.7) {$\dots$};
       \end{scope}
       \draw[dashed] (-3,3.5) to (3,3.5);
       \draw[dashed] (-3,-3.5) to (3,-3.5);
     \end{tikzpicture}
                  \end{array}
                  =
            \begin{array}{c}
        \begin{tikzpicture}[xscale=.45,yscale=.4]
          \begin{scope}
            \clip (-3,3.5) rectangle (2.5,-3.5);
            \draw[fill=blue!20!white] (0,4) rectangle (3,-4);
            \draw[fill=red!20!white] (1,4) rectangle (2,-4);
         %    \draw[fill=red!20!white] (1,4) to (1,1) to[out=-90,in=180] (1.5,0.5) to[out=0,in=-90] (2,1) to (2,4) to (1,4);
         % \draw[fill=red!20!white] (1,-4) to (1,-1) to[out=90,in=180] (1.5,-0.5) to[out=0,in=90] (2,-1) to (2,-4) to (1,-4);            
         \draw[fill=purple!20!white] (-4,4) rectangle (-2,-4);
         \draw[fill=white] (-2.5,2) rectangle (1.5,-2);
         \node at (-.5,0) {$JW_{n_{\color{red}s}}$};
%         \node at (-1,0) {$\dots$};
         \node at (-1,2.7) {$\dots$};
         \node at (-1,-2.7) {$\dots$};
       \end{scope}
       \draw[dashed] (-3,3.5) to (2.5,3.5);
       \draw[dashed] (-3,-3.5) to (2.5,-3.5);
     \end{tikzpicture}
      \end{array}
+        \sum_{a = 1}^{n}\frac{[a]_{\color{purple} u}}{[n+1]_{\bluet}}
      \begin{array}{c}
        \begin{tikzpicture}[xscale=.6,yscale=.55]
          \begin{scope}
            \clip (-3.5,1.5) rectangle (2.5,-3.5);
            % right blue square:
            \draw[fill=blue!20!white] (.5,4) rectangle (3,-4);
             % \draw[fill=red!20!white] (1,4) to (1,1)
             %         to[out=-90,in=180] (1.5,0.5) to[out=0,in=-90]
             %         (2,1) to (2,4) to (1,4);
          % purple cup on top labelled by a
          \draw[fill=purple!20!white] (-1.5,2) to (-1.5,1.5)
          to[out=-90,in=180] (-1,.5) to[out=0,in=-90] (-.5,1.5) to
          (-.5,2);
          % bottom red cup:
         \draw[fill=red!20!white] (1,-4) to (1,-1) to[out=90,in=180] (1.5,-0.5) to[out=0,in=90] (2,-1) to (2,-4) to (1,-4);            
         \draw[fill=purple!20!white] (-4,4) rectangle (-2.5,-4);
%         \draw[fill=white] (-2.5,3) rectangle (1.5,1);
         \draw[fill=white] (-3,-3) rectangle (1.5,-1);
%         \node at (-.5,2) {$JW_{n_{\color{red}s}}$};
         \node at (-.5,-2) {$JW_{n_{\color{red}s}}$};
         \node at (-1.87,1.25) {$\dots$};
          \node at (.12,1.25) {$\dots$};
         \node at (-1,-.25) {$\dots$};
%         \node at (-1,3.25) {$\dots$};
         \node at (-1,-3.25) {$\dots$};
       \end{scope}
       \node at (-1.5,2.05) {$a$};
       \node[color=white] at (-1.5,-3.72) {$a$};
       \draw[dashed] (-3.5,1.5) to (2.5,1.5);
       \draw[dashed] (-3.5,-3.5) to (2.5,-3.5);
     \end{tikzpicture}
      \end{array}
    \end{equation}
In the $a$-th diagram in the sum there is a cup connecting strands $a$ and $a+1$ on top, and the coefficient $[a]$ in the numerator has the same color as the interior of this cup, which alternates between blue and red. We expect there is a two-colored analog of Morrison's closed formula \cite{MorrisonJW} as well. \end{remark}

\subsection{Two-colored Jones-Wenzl projectors: from generic to special}

There are a number of special properties of Jones-Wenzl projectors that one can deduce from \eqref{JWrecursionredux} and \eqref{ptr1solved}. Unfortunately, for a base ring $\base$
where some quantum numbers are not invertible, one can no longer crank the induction to define $JW_n$, and these properties become difficult to prove. What if $[113] = 0$ so that
$JW_{113}$ is not defined; then $[226]_s = 0$ and $[225]$ is invertible and we might hope that $JW_{225}$ exists; do our computations using \eqref{JWrecursionredux} in the generic
case still tell us anything about $JW_{225}$? The goal of this section is to deduce properties of any existant Jones-Wenzl projector from the generic case.

Let $\Bbbk$ be an $A$-algebra which is a commutative domain and a local Henselian ring. In particular, it contains elements $[2]_s$ and $[2]_t$ which are the images of $x_s$ and $x_t$. Let $A_{\Bbbk} = \Bbbk[x_s,x_t]$, made an $A$-algebra via $x_s \mapsto x_s$ and $x_t \mapsto x_t$. Now let $\widehat{A}$ denote the completion of $A_{\Bbbk}$ at the ideal $(x_s - [2]_s,x_t - [2]_t)$.  If we let $z_u = x_u - [2]_u$ for $u = s, t$ then we
have an isomorphism
\begin{equation}
  \label{eq:7}
\widehat{A} = \Bbbk[[z_s, z_t]].  
\end{equation}
The ring $\widehat{A}$ is also a commutative domain, and is also local
and Henselian \cite{MathOverflowHensel, HenselRef}. Now let $Q$ denote the field of fractions of $\widehat{A}$. 

It might be the case that some quantum numbers vanish in $\Bbbk$. Regardless, we claim that all two-colored quantum numbers are invertible in $Q$. After all, the natural maps
$A_{\Bbbk}$ to $\widehat{A}$ and finally to $Q$ are all injective. The map $A \to A_{\Bbbk}$ is not necessarily injective, as $\Bbbk$ might have finite characteristic. But the
two-colored quantum numbers are all monic polynomials (with distinct leading monomials), and they are sent to (distinct) nonzero elements of $A_{\Bbbk}$. Since $Q$ is a field, all
nonzero elements are invertible.

Thus we can apply Theorem \ref{thm:TTLOACC} for $\base \in \{\Bbbk, \widehat{A}, Q\}$, and we can apply Theorem \ref{thm:JWgeneric} for $\base = Q$.

We have functors of extension of scalars
\begin{equation}
  \label{eq:ext}
  \begin{array}{c}
  \begin{tikzpicture}
    \node (l) at (-3,0) {$\TTL_\Bbbk$};
    \node (m) at (0,0) {$\TTL_{\widehat{A}}$};
    \node (r) at (3,0) {$\TTL_{Q}$};
    \draw[->] (m) to node[above] {$(-)\otimes_{\widehat{A}} \Bbbk$} (l);
    \draw[->] (m) to node[above] {$(-)\otimes_{\widehat{A}}Q$} (r);
  \end{tikzpicture}
    \end{array}.
  \end{equation}
These extend (by the universal property) to functors between the corresponding Karoubi envelopes.
  
\begin{prop} \label{prop:lift and split}
  We have isomorphisms:
  \begin{enumerate}
  \item $V_{n_{\reds}}^{\widehat{A}} \otimes \Bbbk \cong V_{n_{\reds}}^\Bbbk$;
  \item $V_{n_{\reds}}^{\widehat{A}} \otimes Q \cong V_{n_{\reds}}^{Q} \oplus
    \bigoplus_{0 \le k = n - 2l} (V_{k_{\reds}}^{Q})^{\oplus \mu_{k}}$.
  \end{enumerate}
  Similar statements hold for $n_{\bluet}$. In other words, the idempotent $e_{n_{\reds}}^{\widehat{A}}$ specializes over $\Bbbk$ to be $e_{n_{\reds}}^{\Bbbk}$, and specializes over $Q$ to be $e_{n_{\reds}}^Q$ plus idempotents corresponding to lower cells.
\end{prop}

\begin{proof} Property (1) is a consequence of idempotent lifting: every idempotent in $\End(n_{\reds})$ over $\Bbbk$ lifts to an idempotent over $\widehat{A}$. Hence indecomposable idempotents lift to indecomposable idempotents. The idempotent $e_{n_{\reds}}^{\Bbbk}$ thus lifts to an idempotent over $\widehat{A}$ which satisfies the conditions to be $e_{n_{\reds}}^{\widehat{A}}$. Property (2) follows directly from the last part of Theorem \ref{thm:TTLOACC}. \end{proof}
	
We will not use Property (2), we merely wish to emphasize that the specialization of $e_{n_{\reds}}^{\widehat{A}}$ to $Q$ is not necessarily $e_{n_{\reds}}^{Q}$, because it need
not be indecomposable.

\begin{thm} \label{thm:generaltospecial}
  Suppose that $JW^\Bbbk_{n_{\reds}}$ exists. Then it lifts to $JW^{\widehat{A}}_{n_{\reds}}$, which specializes to $JW^Q_{n_{\reds}}$.
\end{thm}

\begin{proof}
  Suppose that $JW^\Bbbk_{n_{\reds}}$ exists. By Proposition \ref{prop:lift and split}, it lifts to an idempotent $e_{n_{\reds}}^{\widehat{A}}$, and
\begin{equation}
  V_{n_{\reds}}^{\Bbbk}  = V_{n_{\reds}}^{\widehat{A}}
  \otimes_{\widehat{A}} \Bbbk.\label{eq:lift1}  
\end{equation}
For any $0 \le k < n$ we have 
  \[
\Hom_{\TTL_{\widehat{A}}}(V^{\widehat{A}}_{k_{\reds}},
V^{\widehat{A}}_{n_{\reds}}) \otimes_{\widehat{A}} \Bbbk =
\Hom_{\TTL_\Bbbk}(V^\Bbbk_{k_{\reds}},
V_{n_{\reds}}^\Bbbk) = 0
\]
where the the first equality uses (1) in Proposition \ref{prop:lift
  and split} and second equality again uses \eqref{eq:orthokprecomp}. As $\Hom_{\TTL_{\widehat{A}}}(V^{\widehat{A}}_{k_{\reds}},
V^{\widehat{A}}_{n_{\reds}})$ is certainly finitely-generated over
$\widehat{A}$, we conclude that
\begin{equation}
  \label{eq:JW hom vanishing R}
  \Hom_{\TTL_{\widehat{A}}}(V^{\widehat{A}}_{k_{\reds}},
V^{\widehat{A}}_{n_{\reds}})  = 0 \qquad \text{for any $0 \le k < n$}
\end{equation}
by Nakayama's lemma. Hence $e_{n_{\reds}}^{\widehat{A}}$ satisfies \eqref{eq:orthokprecomp} and is a Jones-Wenzl projector.

Now let $f$ be the specialization of $e_{n_{\reds}}^{\widehat{A}}$ to $Q$. Property \eqref{eq:orthokprecomp} (i.e. death by cups) for $e_{n_{\reds}}^{\widehat{A}}$ specializes to
\eqref{eq:orthokprecomp} for $f$. This in turn implies that the endomorphism ring of the image of $f$ is one-dimensional over $Q$ (spanned by the identity map), so $f$ is an
indecomposable idempotent. Hence $f = e_{n_{\reds}}^Q$.
\end{proof}

The upshot of Theorem \ref{thm:generaltospecial} is that it implies that, whenever $JW_n^{\Bbbk}$ exists, it has the same ``coefficients'' as $JW_n^Q$, in the sense that both come
from the same idempotent $JW_n^{\widehat{A}}$.

\begin{ex} \label{ex:funky3conclusion} Using \eqref{JWrecursionredux}, we can compute that $JW_{3_{\bluet}}^Q$ is a linear combination of five crossingless matchings: the identity with coefficient $1$, two
others with coefficient $\frac{1}{[3]}$, one with coefficient $\frac{[2]_s}{[3]}$, and one with coefficient $\frac{[2]_t}{[3]}$. More precisely, all these coefficients should be
interpreted as rational functions in $x_s$ and $x_t$. (See \cite[equation (A.3)]{EDC}, matching notation via $x = x_s$ and $y = x_t$.) Now suppose that
$JW_{3_{\reds}}^{\Bbbk}$ exists. Since the idempotent $JW_{3_{\reds}}^Q$ lifts to $JW_{3_{\reds}}^{\widehat{A}}$, the ring $\widehat{A}$ must have some element which specializes to
$\frac{1}{[3]}$, so that $[3] = x_s x_t - 1$ must be invertible in $\widehat{A}$ (by the injectivity of $\widehat{A} \to Q$). In turn, this implies that its specialization $[3]$ is
invertible in $\Bbbk$. The exact same formula for the coefficients of the Jones-Wenzl projector must hold over $\Bbbk$ as well. This is true even in the unusual situations where
$[2]_s = 0$ and/or $[2]_t = 0$ in $\Bbbk$, where $JW_2$ need not exist, so we can not deduce the coefficients immediately from \eqref{JWrecursionredux}. Consequently,
\eqref{ptr1solved} still holds over $\Bbbk$, telling us that
$\ptr_1(JW_{3_{\bluet}}) = -\frac{[4]_s}{[3]}$. In particular,
$\ptr_1(JW_{3_{\bluet}}) = 0$ when $[2]_s = 0$. \end{ex}

\subsection{Rotatable Jones-Wenzl projectors} \label{subsec:rotinvce}

In this section, we fix an $A$-algebra $\Bbbk$ which is a commutative domain and a local Henselian ring. When omitted, the base ring of any idempotents is $\Bbbk$.

Let us assume that $JW_n$ exists. We now ask whether it is possible
that the (counterclockwise) rotation of $JW_{n_{\reds}}$ by one strand
is equal to a scalar multiple of $JW_{n_{\bluet}}$:
\begin{equation}
  \label{eq:JWrot}
  \begin{array}{c}
  \begin{tikzpicture}[scale=.5]
    \begin{scope}
         \clip (-3,2) rectangle (3,-2);
    \draw[fill=blue!20!white] (-3.3,3) rectangle (3.3,-3);
             \draw[fill=red!20!white] (-.9,3) rectangle (.9,-3);
             \draw[fill=white] (-.3,3) rectangle (.3,-3);
             \draw[fill=red!20!white] (1.5,3) to (1.5,-1)
         to[out=-90,in=180] (2,-1.5) to[out=0,in=-90] (2.5,-1) to
         (2.5,3) to (1.5,3);
         \draw[fill=red!20!white] (-1.5,-3) to (-1.5,1)
         to[out=90,in=0] (-2,1.5) to[out=180,in=90] (-2.5,1) to
         (-2.5,-3) to (-1.5,-3);
         \draw[fill=white] (-1.9,1) rectangle (1.9,-1);
         \node at (0,0) {$JW_{n_{\color{red} s}}$};
         \node at (0,1.4) {$\dots$}; 
         \node at (0,-1.4) {$\dots$};
         \end{scope}
         \draw[dashed] (-3,2) to (3,2);
         \draw[dashed] (-3,-2) to (3,-2);
       \end{tikzpicture}
     \end{array}
     =
     ?
     \begin{array}{c}
  \begin{tikzpicture}[scale=.5]
    \begin{scope}
         \clip (-2.3,2) rectangle (2.3,-2);
    \draw[fill=blue!20!white] (-3.3,3) rectangle (3.3,-3);
             \draw[fill=red!20!white] (-1.5,3) rectangle (1.5,-3);
             \draw[fill=blue!20!white] (-.9,3) rectangle (.9,-3);
             \draw[fill=white] (-.3,3) rectangle (.3,-3);
         \draw[fill=white] (-1.9,1) rectangle (1.9,-1);
         \node at (0,0) {$JW_{n_{\color{blue} t}}$};
         \node at (0,1.4) {$\dots$}; 
         \node at (0,-1.4) {$\dots$};
         \end{scope}
         \draw[dashed] (-2.3,2) to (2.3,2);
         \draw[dashed] (-2.3,-2) to (2.3,-2);
       \end{tikzpicture}
       \end{array}
\end{equation}

\begin{lem} \label{lem:rotatable} Suppose that $JW_n$ exists. The rotation of $JW_{n_{\reds}}$ by one strand, as pictured in \eqref{eq:JWrot}, is a scalar multiple of $JW_{n_{\bluet}}$ if and only if $\ptr(JW_{n_{\reds}}) = 0$. \end{lem}

\begin{proof} Being equal to a scalar multiple of a Jones-Wenzl projector is equivalent to death by caps. The rotation of a Jones-Wenzl projector automatically killed by all but
the rightmost cap. It is killed by the rightmost cap if and only if the original diagram has zero partial trace. \end{proof}

\begin{defn} When $\ptr(JW_{n_{\reds}}) = 0$ and $\ptr(JW_{n_{\bluet}}) = 0$ we say that $JW_n$ is \emph{rotatable}. \end{defn}
	
The following result implies that the appropriate Jones-Wenzl projector is rotatable for most dihedrally-faithful realizations. In most cases this is proven in \cite{EDC}, but the edge cases are treated carefully here.

\begin{thm} \label{thm:rotatableoverfield} Suppose that $m = m_{st} < \infty$, that $(st)$ has order exactly $m$ on the span of $\a_s$ and $\a_t$, and that the base ring is a field. Then $JW_{m-1}$ exists and is
rotatable. \end{thm}

\begin{proof} By Theorem \ref{thm:whenfaithful}, either $[k]_s \ne 0$ or $[k]_t \ne 0$ for all $k < m$. In particular, $m$ is the first positive integer such that $[m]_s = [m]_t = 0$.

Suppose that $[k]_s, [k]_t \ne 0$ for all $k < m$. Since the base ring is a field, $[k]_s$ and $[k]_t$ are invertible for all $k<m$. Hence we can apply the inductive formulas to construct all $JW_k$ for $k < m$, and we can use \eqref{ptr1solved}. By definition of a realization, $[m]_s = [m]_t = 0$, so $\ptr_1(JW_{m-1}) = 0$ regardless of coloring. Because $JW_{m-2}$ exists, \eqref{ptrvsptr1} implies that $\ptr(JW_{m-1}) = 0$, thus $JW_{m-1}$ is rotatable by definition. This handles the case when $m=2$, so we assume $m > 2$.

In the remaining case, $[k]_s = 0$ but $[k]_t \ne 0$, or vice versa, for some $k < m$. Choose $d$ minimal such that this situation arises for $k=d$. Clearly $d$ is even, because if $d$ is odd then $[d]_s = [d]_t$. If $d > 2$ then $[2]_s$ and $[2]_t$ are nonzero, hence invertible. But then \eqref{kover2} implies that $[d]_s$ and $[d]_t$ are related by a unit, a contradiction since one is zero and the other isn't. The only remaining possibility is that $d=2$.

Now we treat the case when $[2]_s = 0$ and $[2]_t \ne 0$. All quantum numbers were computed in \eqref{weirdcasequantums}. If the base field has characteristic zero, then $[k]_t \ne 0$ for all $k$, a contradiction since $[m]_t = 0$. If the base ring has characteristic $p$, then $[2p]_s = [2p]_t = 0$, and $2p$ is the first positive integer for which both quantum numbers are zero. Consequently, $m = 2p$. It remains to construct $JW_{m-1}$ in this case, and prove it is rotatable, which is the content of Theorem \ref{thm:charpcase}. \end{proof}

We are still interested in realizations which are not dihedrally faithful, so that the above theorem does not apply.

\begin{ex} \label{ex:funky3redux} Suppose that $[2]_s = [2]_t =
  0$. This arises for a non-faithful realization when $m=4$. The
  Jones-Wenzl projector $JW_{3_{\reds}}$ exists and is given by
\begin{equation}
  \label{eq:JW3ex}
  \begin{array}{c}
  \begin{tikzpicture}[xscale=-.8]
    \begin{scope}
      \clip (-1.3,.8) rectangle (1.3,-.8);
      \draw[fill=blue!20!white] (-2,1) rectangle (2,-1);
      \draw[fill=red!20!white] (-2,-1) rectangle (-.8,1);
      \draw[fill=red!20!white] (0,-1) rectangle (.8,1);
      \draw[fill=white] (-1,.5) rectangle (1,-.5);
      \node at (0,0) {$JW_{3_{\color{red}s}}$};
    \end{scope}
    \draw[dashed] (-1.3,.8) to (1.3,.8);
    \draw[dashed] (-1.3,-.8) to (1.3,-.8);
  \end{tikzpicture}
  \end{array}
   =
  \begin{array}{c}
  \begin{tikzpicture}[xscale=-.8]
    \begin{scope}
      \clip (-1.3,.8) rectangle (1.3,-.8);
      \draw[fill=blue!20!white] (-2,1) rectangle (2,-1);
      \draw[fill=red!20!white] (-2,-1) rectangle (-.8,1);
            \draw[fill=red!20!white] (0,-1) rectangle (.8,1);
    \end{scope}
    \draw[dashed] (-1.3,.8) to (1.3,.8);
    \draw[dashed] (-1.3,-.8) to (1.3,-.8);
  \end{tikzpicture}
  \end{array}
  -
  \begin{array}{c}
  \begin{tikzpicture}[xscale=-.8]
    \begin{scope}
      \clip (-1.3,.8) rectangle (1.3,-.8);
      \draw[fill=blue!20!white] (-2,1) rectangle (2,-1);
      \draw[fill=red!20!white] (-2,-1) to (-1,-.8) to[out=90,in=-90] 
      (1,.8) to (1.5,1) to (0,.8) to[out=-90,in=0] (-.5,.4) to[out=-180,in=-90] (-1,.8) to
      (-1,1) to (-2,1) to (-2,--1);
      \draw[fill=red!20!white] (0,-.8) to[out=90,in=180] (.5,-.4) to[out=0,in=90] (1,-.8) to
      (1,-1) to (0,-1) to (0,-.8);
    \end{scope}
    \draw[dashed] (-1.3,.8) to (1.3,.8);
    \draw[dashed] (-1.3,-.8) to (1.3,-.8);
  \end{tikzpicture}
  \end{array}
 -
  \begin{array}{c}
  \begin{tikzpicture}[xscale=.8]
    \begin{scope}
      \clip (-1.3,.8) rectangle (1.3,-.8);
      \draw[fill=red!20!white] (-2,1) rectangle (2,-1);
      \draw[fill=blue!20!white] (-2,-1) to (-1,-.8) to[out=90,in=-90] 
      (1,.8) to (1.5,1) to (0,.8) to[out=-90,in=0] (-.5,.4) to[out=-180,in=-90] (-1,.8) to
      (-1,1) to (-2,1) to (-2,--1);
      \draw[fill=blue!20!white] (0,-.8) to[out=90,in=180] (.5,-.4) to[out=0,in=90] (1,-.8) to
      (1,-1) to (0,-1) to (0,-.8);
    \end{scope}
    \draw[dashed] (-1.3,.8) to (1.3,.8);
    \draw[dashed] (-1.3,-.8) to (1.3,-.8);
  \end{tikzpicture}
  \end{array}.
\end{equation}
(The reader can either verify death by caps directly, or follow the discussion of Example \ref{ex:funky3conclusion}.)
One can define $JW_{3_{\bluet}}$ by reversing the roles of the two
colors. In particular, the rotation of $JW_{3_{\reds}}$ is not a
scalar multiple of $JW_{3_{\bluet}}$. The partial trace is
\begin{equation}
  \label{eq:JW3extrace}
  \begin{array}{c}
  \begin{tikzpicture}[xscale=.6]
    \begin{scope}
      \clip (-1.3,.8) rectangle (1.5,-.8);
      \draw[fill=blue!20!white] (2,1) rectangle (-2,-1);
%      \draw[fill=red!20!white] (2,-1) rectangle (.8,1);
      \draw[fill=red!20!white] (0,-1) rectangle (-.8,1);

      \draw[fill=red!20!white] (.6,.35) to[out=90,in=180] (1,.6)
      to[out=0,in=90] (1.4,0) to[out=-90,in=0] (1,-.6) %to (.8,.35);
      to[out=180,in=-90] (.6,-.35) to (.6,.35);

       \draw[fill=white] (-1,.35) rectangle (1,-.35);
      \node at (0,0) {$JW_{3_{\color{blue}s}}$};
    \end{scope}
    \draw[dashed] (-1.3,.8) to (1.5,.8);
    \draw[dashed] (-1.3,-.8) to (1.5,-.8);
  \end{tikzpicture}
  \end{array}
=  -2 
  \begin{array}{c}
  \begin{tikzpicture}[xscale=-.6]
    \begin{scope}
      \clip (-.3,.8) rectangle (1.3,-.8);
      \draw[fill=blue!20!white] (-2,1) rectangle (2,-1);
            \draw[fill=red!20!white] (0,.8) to[out=-90,in=180] (.5,.4) to[out=0,in=-90] (1,.8) to
      (1,1) to (0,1) to (0,.8);
      \draw[fill=red!20!white] (0,-.8) to[out=90,in=180] (.5,-.4) to[out=0,in=90] (1,-.8) to
      (1,-1) to (0,-1) to (0,-.8);
    \end{scope}
    \draw[dashed] (-.3,.8) to (1.3,.8);
    \draw[dashed] (-.3,-.8) to (1.3,-.8);
  \end{tikzpicture}
  \end{array},
\end{equation}
which is nonzero when the characteristic of the base ring is not $2$. Note that $\ptr_1(JW_{3_{\reds}}) = 0$ as predicted by \eqref{ptr1solved}. \end{ex}

\begin{remark} \label{rmk:ptrptr1notequiv} When $JW_{n-1}$ is defined then \eqref{ptrvsptr1} holds, so that $\ptr(JW_n) = 0$ if and only if $\ptr_1(JW_n)=0$. The previous example indicates that $\ptr(JW_n) =
0$ is not equivalent to $\ptr_1(JW_n) = 0$ when $JW_{n-1}$ does not exist. \end{remark}

\begin{remark} \label{rmk:negligiblevs} A common concept is that of the \emph{negligible} Jones-Wenzl projector. Negligibility is the property that, no matter how the morphism is glued into a larger closed diagram (i.e. an endomorphism of $0_{\reds}$ or $0_{\bluet}$), the result is zero. Rotatable Jones-Wenzl projectors are negligible, but the converse is false. In Example \ref{ex:funky3redux}, $JW_{3_{\reds}}$ is negligible but not rotatable. \end{remark}

Here is the main result of this section, which was stated without proof in \cite{EDC}.

\begin{lem} \label{lem:rotinvce} Suppose that $JW_n$ exists and is rotatable. Then the rotation of $JW_{n_{\reds}}$ by one strand (clockwise or
counterclockwise) is equal to $[n]_{\bluet} JW_{n_{\bluet}}$, and vice versa. Moreover, $[n]_{\reds} [n]_{\bluet} = 1$, implying that the rotation of $JW_{n_{\reds}}$ by two strands is again $JW_{n_{\reds}}$. If $n$ is odd then $[n]_{\reds} = [n]_{\bluet} = 1$.
\end{lem}

\begin{proof} 
Suppose that (counterclockwise) rotation of $JW_{n_{\reds}}$ by one strand is equal to $a$ times $JW_{n_{\bluet}}$, and that
rotation of $JW_{n_{\bluet}}$ by one strand is equal to $b$ times
$JW_{n_{\reds}}$. That is:
\begin{gather}
  \label{eq:JWrota}
  \begin{array}{c}
  \begin{tikzpicture}[scale=.4]
    \begin{scope}
         \clip (-3,2) rectangle (3,-2);
    \draw[fill=blue!20!white] (-3.3,3) rectangle (3.3,-3);
             \draw[fill=red!20!white] (-.9,3) rectangle (.9,-3);
             \draw[fill=white] (-.3,3) rectangle (.3,-3);
             \draw[fill=red!20!white] (1.5,3) to (1.5,-1)
         to[out=-90,in=180] (2,-1.5) to[out=0,in=-90] (2.5,-1) to
         (2.5,3) to (1.5,3);
         \draw[fill=red!20!white] (-1.5,-3) to (-1.5,1)
         to[out=90,in=0] (-2,1.5) to[out=180,in=90] (-2.5,1) to
         (-2.5,-3) to (-1.5,-3);
         \draw[fill=white] (-1.9,1) rectangle (1.9,-1);
         \node at (0,0) {$JW_{n_{\color{red} s}}$};
         \node at (0,1.4) {$\dots$}; 
         \node at (0,-1.4) {$\dots$};
         \end{scope}
         \draw[dashed] (-3,2) to (3,2);
         \draw[dashed] (-3,-2) to (3,-2);
       \end{tikzpicture}
     \end{array}
     =
     a
     \begin{array}{c}
  \begin{tikzpicture}[scale=.4]
    \begin{scope}
         \clip (-2.3,2) rectangle (2.3,-2);
    \draw[fill=blue!20!white] (-3.3,3) rectangle (3.3,-3);
             \draw[fill=red!20!white] (-1.5,3) rectangle (1.5,-3);
             \draw[fill=blue!20!white] (-.9,3) rectangle (.9,-3);
             \draw[fill=white] (-.3,3) rectangle (.3,-3);
         \draw[fill=white] (-1.9,1) rectangle (1.9,-1);
         \node at (0,0) {$JW_{n_{\color{blue} t}}$};
         \node at (0,1.4) {$\dots$}; 
         \node at (0,-1.4) {$\dots$};
         \end{scope}
         \draw[dashed] (-2.3,2) to (2.3,2);
         \draw[dashed] (-2.3,-2) to (2.3,-2);
       \end{tikzpicture}
     \end{array}, \\
  \label{eq:JWrotb}
  \begin{array}{c}
  \begin{tikzpicture}[scale=.4]
    \begin{scope}
         \clip (-3,2) rectangle (3,-2);
    \draw[fill=red!20!white] (-3.3,3) rectangle (3.3,-3);
             \draw[fill=blue!20!white] (-.9,3) rectangle (.9,-3);
             \draw[fill=white] (-.3,3) rectangle (.3,-3);
             \draw[fill=blue!20!white] (1.5,3) to (1.5,-1)
         to[out=-90,in=180] (2,-1.5) to[out=0,in=-90] (2.5,-1) to
         (2.5,3) to (1.5,3);
         \draw[fill=blue!20!white] (-1.5,-3) to (-1.5,1)
         to[out=90,in=0] (-2,1.5) to[out=180,in=90] (-2.5,1) to
         (-2.5,-3) to (-1.5,-3);
         \draw[fill=white] (-1.9,1) rectangle (1.9,-1);
         \node at (0,0) {$JW_{n_{\color{blue} t}}$};
         \node at (0,1.4) {$\dots$}; 
         \node at (0,-1.4) {$\dots$};
         \end{scope}
         \draw[dashed] (-3,2) to (3,2);
         \draw[dashed] (-3,-2) to (3,-2);
       \end{tikzpicture}
     \end{array}
     =
     b
     \begin{array}{c}
  \begin{tikzpicture}[scale=.4]
    \begin{scope}
         \clip (-2.3,2) rectangle (2.3,-2);
    \draw[fill=red!20!white] (-3.3,3) rectangle (3.3,-3);
             \draw[fill=blue!20!white] (-1.5,3) rectangle (1.5,-3);
             \draw[fill=red!20!white] (-.9,3) rectangle (.9,-3);
             \draw[fill=white] (-.3,3) rectangle (.3,-3);
         \draw[fill=white] (-1.9,1) rectangle (1.9,-1);
         \node at (0,0) {$JW_{n_{\color{red} s}}$};
         \node at (0,1.4) {$\dots$}; 
         \node at (0,-1.4) {$\dots$};
         \end{scope}
         \draw[dashed] (-2.3,2) to (2.3,2);
         \draw[dashed] (-2.3,-2) to (2.3,-2);
       \end{tikzpicture}
       \end{array}.
\end{gather}
     Using invariance of Jones-Wenzl projectors under vertical flips, we see that clockwise rotation
of $JW_{n_{\reds}}$ is also equal to $a$ times $JW_{n_{\bluet}}$. Now using adjunction, one can argue that $ab = 1$. This also implies that rotation of a rotatable Jones-Wenzl
projector by two strands will recover the Jones-Wenzl projector on the nose. This isn't surprising, owing to the existence of crossingless matchings which are invariant under
counterclockwise rotation by two strands\footnote{Generically, all crossingless matchings have nonzero coefficients in the Jones-Wenzl projector.}.

Rotation by $n$ strands will send $JW_{n_{\reds}}$ to $JW_{{}_{\reds} n}$, implying that $abab\cdots = 1$ where this word has length $n$. If particular, if $n$ is odd then $a =
b = 1$.

In fact, the rotational scalars $a$ and $b$ can be determined
explicitly by the following argument. Let $c_i$ denote the coefficient
inside $JW_{n_{\reds}}$ of the following diagram, which we denote $y_i$:
\begin{equation}
  \label{eq:cidef} y_i = 
  \begin{array}{c}
    \begin{tikzpicture}[xscale=.6,yscale=1]
    \begin{scope}
      \clip (-3.5,1) rectangle (3.5,-1);
      \draw[fill=red!20!white] (-4,2) rectangle (-3,-2);
      \draw[fill=blue!20!white] (-3,2) rectangle (-2,-2);

      \draw[fill=blue!20!white] (0,-1) to[out=90,in=-90] (2,1) to
      (2,2) to (-1,2) to (-1,-2) to (0,-1);

      \draw[fill=red!20!white] (1,-1) to[out=90,in=-90] (3,1) to
      (3,2) to (4,2) to (4,-2) to (0,-1);

      \draw[fill=red!20!white] (0,1) to[out=-90,in=-90] (1,1) to
      (1,2) to (0,2) to ((0,1);

      \draw[fill=blue!20!white] (2,-1) to[out=90,in=90] (3,-1) to
      (3,-2) to (1,-2) to (2,-1);

      \node at (-1.5,0) {$\dots$};
      \node at (.5,-.8) {$\dots$};
      \node at (2.5,.8) {$\dots$};
      
    \end{scope}
    \draw[dashed] (-3.5,1) to (3.5,1);
    \draw[dashed] (-3.5,-1) to (3.5,-1);

          \draw [decorate,decoration={brace,amplitude=10pt},yshift=4pt]
          (1,1) -- (3,1);
          \node at (2,1.7) {$i$};
          % for a balanced picture
          \node[white] at (2,-1.7) {$.$};
  \end{tikzpicture}
  \end{array}
\end{equation}
(There are $i$ strands to the right of
the cup on top. For example, $y_1$ is invariant under a vertical
flip.) After counterclockwise rotation by one strand, $y_{n-1}$
becomes the identity, so $a = c_{n-1}$. We claim
that $a = c_{n-1} = \frac{1}{[n]_{\reds}}$, which exists in
$\Bbbk$. Similarly, $b = \frac{1}{[n]_{\bluet}}$. Since $ab = 1$ we
deduce that $[n]_{\reds} [n]_{\bluet} = 1$, and that $a =
[n]_{\bluet}$ and $b = [n]_{\reds}$. Thus this claim finishes the proof of
the lemma.

One way to verify that $c_{n-1} = \frac{1}{[n]_{\reds}}$ is to verify this fact in the generic Jones-Wenzl projector $JW_{n_{\reds}}^Q$. This can be done most easily using
\eqref{eq:singleclasp}. Using Theorem \ref{thm:generaltospecial} we deduce that $[n]_{\reds}$ is invertible in $\widehat{A}$ and hence in $\Bbbk$, and that $\frac{1}{[n]_{\reds}}$
is also the coefficient inside $JW_{n_{\reds}}^{\Bbbk}$.

There is also an elementary proof that $c_{n-1} = \frac{1}{[n]_{\reds}}$, which we briefly outline. Set $y_0$ to be the identity diagram, and $c_0 = 1$ (the coefficient of the identity diagram), and $c_n = 0$. For $1 \le i \le n-1$, adding a $i$-th cap from the right to $y_i$ (right on top of the cup in $y_i$) yields $-[2]_i$ times a particular diagram $Y$. Here $[2]_i$ is $[2]_{\reds}$ for $i$ odd and $[2]_{\bluet}$
for $i$ even. Meanwhile, adding this $i$-th cap to $y_{i \pm 1}$ yields $Y$ with coefficient $1$. 
No other diagram in $JW_n$ will yield $Y$ after the placement of any cap. By isolating the coefficient of $Y$ in a death by caps relation, we deduce that $[2]_i c_i = c_{i+1} + c_{i-1}$. Now induction gives $c_{n-i} = [n-i]_ic_{n-1}$ where $[n-i]_i$ is either $[n-i]_{\reds}$ or $[n-i]_{\bluet}$. In conclusion, $c_0 = [n]_{\reds} c_{n-1}$. \end{proof}

\subsection{Evaluating Jones-Wenzl projectors} \label{subsec:evaluation}

Consider a realization of $(W,S)$ with base ring $\Bbbk$, pick a pair $s \ne t \in S$, and make $\Bbbk$ an $A$-algebra via \eqref{eq:realizationspecialization}. Given a two-colored
crossingless matching we can associate a monomial in $R = \Bbbk[\a_s, \a_t]$, namely $\a_{\reds}^c \a_{\bluet}^d$, where $c$ is the number of $\reds$-colored regions and $d$ is the
number of $\bluet$-colored regions. To a morphism $f$ in $\TTL$, thought of as a linear combination of crossingless matchings, we can associate the corresponding linear combination
of monomials, which we denote $\poly(f)$.

\begin{lem}\label{lem:evaluation} Suppose that $JW_{m-1}$ exists. We have $\poly(JW_{(m-1)_{\reds}}) = \pi_{\reds,\bluet}$ and $\poly(JW_{(m-1)_{\bluet}}) = \pi_{\bluet,\reds}$. \end{lem}

The polynomials $\pi_{\reds,\bluet}$ and $\pi_{\bluet,\reds}$ were defined in \eqref{eq:pist} and \eqref{eq:pits} respectively. We make no balanced-ness assumption on the
realization, so it is possible that $\pi_{\reds,\bluet} \ne \pi_{\bluet,\reds}$; further discussion of these polynomials is found in \S\ref{subsec:rootspolys}.

\begin{proof} Note that we can also define realizations of $(W,S)$ with base rings $\widehat{A}$ and $Q$ respectively, where $\a_s^\vee(\a_t) := -x_s$ and $\a_t^\vee(\a_s) :=
-x_t$. We need only prove the result when the base ring is $Q$; then the formula will lift to $\widehat{A}$ and specialize to $\Bbbk$ by Theorem \ref{thm:generaltospecial}.

Over $Q$ we can use \eqref{JWrecursionredux}. In \cite{EDC}, it is inductively proven that $\poly(JW_{k_{\reds}})$ is the product of the first $k+1$
roots in the formula for $\pi_{\reds,\bluet}$. More precisely, this is proven when $[2]_s = [2]_t$ and $[m-1]_s = [m-1]_t = 1$ in \cite[Proposition 4.14]{EDC}, and adapted to
the general case in \cite[Chapter A.6]{EDC}. The proof in \cite{EDC} relies upon the recursive formula \eqref{JWrecursionredux}, although that paper fails to mention
this assumption at times. \end{proof}

Lemma \ref{lem:evaluation} is critical to our proof of the well-definedness of $\Lambda$. Note the compatibility of Lemma \ref{lem:evaluation} with Lemma \ref{lem:rotinvce} and
\eqref{eq:pistvspits}.

\subsection{The existence of Jones-Wenzl projectors} \label{subsec:whenexist}

So when does $JW_n^{\Bbbk}$ exist? Is there a precise criterion? When precisely is $JW_n$ rotatable? We have given necessary conditions above, but no sufficient conditions.

Let us first ask the analogous question for the Jones-Wenzl projector in the ordinary (uncolored) Temperley-Lieb algebra, where the circle has value $-[2]_q = -(q+q^{-1})$. In
\cite[Theorem A.2]{ELib}, Ben Webster proved that the Jones-Wenzl projector $JW_n$ exists so long as the quantum binomial coefficients ${n \brack k}_q$ are invertible for all $1
\le k \le n$. The proof uses representation theory, so it does not immediately adapt to the two-colored case. However, it suggests that the appropriate expectation is that the
two-colored Jones-Wenzl projectors $JW_n$ exist so long as the two-colored quantum binomial coefficients ${n \brack k}_s$ and ${n \brack k}_t$ are invertible. This was stated as a
theorem in \cite{ELib} without proof, supposedly as a consequence of \cite[Theorem A.2]{ELib}. Let us revise its status to a conjecture now.

\begin{conj} The Jones-Wenzl projectors $JW_{n_{\reds}}$ and $JW_{n_{\bluet}}$ exist over $\Bbbk$ if and only if ${n \brack k}_s$ and ${n \brack k}_t$ are invertible in $\Bbbk$.
\end{conj}

Just because $[m-1]_s$ is invertible does not imply that ${m-1 \brack k}_s$ is invertible. We do not know a precise criterion for when ${m-1 \brack k}_s$ and ${m-1 \brack k}_t$ are
invertible for all $0 \le k \le m-1$.

\begin{ex} \label{ex:bad6} Suppose that $[2]_{\reds} = [2]_{\bluet}=0$ (let's ignore the colors then), so that $[3] = -1$, $[4] = 0$, $[5] = 1$, and $[6]=0$. Note that
$\frac{[4]}{[2]} = [3] - 1 = -2$, and hence ${5 \brack 2} = -2$. Let $m = 6$. Note that $[m-1]$ is invertible but ${m-1 \brack 2}$ is not unless $2$ is invertible. One can verify
that $JW_5$ is defined in this case, so long as $2$ is invertible. It fails to be rotatable. \end{ex}

\begin{remark} One possible way to prove this conjecture would be to reproduce Morrison's closed formula and its proof from \cite{MorrisonJW} in the two-colored case over $Q$.
Then one could use Theorem \ref{thm:generaltospecial} to deduce that the same formula holds over $\Bbbk$ for any $\Bbbk$. In particular, the well-definedness of the coefficients in
that formula would be equivalent to the existence of the Jones-Wenzl projector. It seems likely that Morrison's formula could be rephrased so that the only denominators required
are quantum binomial coefficients, but this too has not been done. \end{remark}

Above we argued that the rotatability of $JW_n$ implies that $[n]_{\reds} [n]_{\bluet} = 1$, and that $[n]_{\reds} = 1$ when $n$ is odd. We did this by examining very particular
coefficients in the Jones-Wenzl projector. We noted that generically the partial trace of $JW_{n_{\reds}}^Q$ is $\frac{-[n+1]_{\bluet}}{[n]_{\reds}} JW_{(n-1)_{\bluet}}^Q$, so that
one can prove that $JW_n$ is rotatable when $[n+1]_{\reds} = [n+1]_{\bluet} = 0$ and $JW_{n-1}$ exists. However, when $JW_{n-1}$ does not exist, the problem is quite subtle,
Example \ref{ex:bad6} demonstrates that the conditions $[n+1] = 0$ and $[n] = 1$ are not sufficient to imply the rotatability of $JW_n$.

\begin{ex} \label{ex:notdihedralfaithfulstillrotatable} A realization where $[2]_s = [2]_t = 0$ can be thought of as a faithful realization of the dihedral group $A_1 \times A_1$ with $m=2$, or a non-faithful realization for $m=4$ (or any higher even value of $m$). When $[2]_s = [2]_t = 0$, the explicit formula from Example \ref{ex:funky3redux} shows that $JW_3$ is rotatable if and only if the base ring has characteristic $2$. In characteristic $2$, this gives an example of a rotatable Jones-Wenzl projector even when the realization (for $m=4$) is not dihedrally faithful. \end{ex}

\begin{ex} \label{evenbalancedbad} Similarly, suppose that $[2]_s = [2]_t = 1$ and hence $[3] = 0$, and suppose that $2$ is invertible. One can compute that $JW_8$ exists, that $\ptr_1(JW_8) = 0$, but that $\ptr(JW_8)$ is $\frac{3}{2}$ times some particular nonzero morphism. We have omitted the lengthy computation. In characteristic $3$ this gives a rotatable Jones-Wenzl projector, but outside of characteristic $3$ this gives a non-rotatable Jones-Wenzl projector. Note that $[8] = 1$, so either way this is a balanced non-faithful dihedral realization when $m = 9$.\end{ex}

\subsection{Jones-Wenzl projectors in special cases} \label{subsec:specialcases}
In this section we describe Jones-Wenzl projectors $JW_{m-1}$, and prove that they are rotatable, in two special cases of interest: \begin{itemize}
\item When $[2]_s = 0$ but $[2]_t \ne 0$ (or vice versa), and $m = 2p$ in characteristic $p$. This handles the Kac-Moody realizations of type $B_2$ in characteristic $2$, and $G_2$ in characteristic $3$.
\item When $[2]_s = [2]_t = 1$ in characteristic $2$, and $m=6$. This
  handles the Kac-Moody realization of type $G_2$ in characteristic
  $2$.
\end{itemize} 

We begin with type $G_2$ in characteristic $2$.

\begin{thm} Suppose that $[2]_s = [2]_t = 1$ and $\Bbbk$ has characteristic $2$. Then $JW_5$ exists over $\Bbbk$, and is rotatable. \end{thm}

\begin{proof} Because $[2]_s = [2]_t$ we will ignore the colors (as $JW_{n_{\reds}}$ becomes $JW_{n_{\bluet}}$ after applying color swap). Note that if $JW_5$ is rotatable, it must be invariant
  under rotation by Lemma \ref{lem:rotinvce}. It must also be
  invariant under the vertical and horizontal flip. The 42
  crossingless $(5,5)$ matchings split into 6 orbits under these flips
  and rotations. Here are three representatives of orbits, together
  with the sizes of the orbits under flips and rotations:
      \[
      \begin{array}{c}
    \begin{tikzpicture}[scale=.5]
      \draw[dashed] (0,0) circle (2);
      % 18, 54, 90, 126, 162
      \draw (18:2) to[out=-162,in=162] (-18:2);
      \draw (54:2) to[out=-126,in=126] (-54:2);
      \draw (90:2) to (-90:2);
      \draw (126:2) to[out=-54,in=54] (-126:2);
      \draw (162:2) to[out=-18,in=18] (-162:2);
    \end{tikzpicture}\\
        \text{(size 5)}
      \end{array} \qquad
            \begin{array}{c}
    \begin{tikzpicture}[scale=.5]
      \draw[dashed] (0,0) circle (2);
      % 18, 54, 90, 126, 162
      \draw (18:2) to[out=-162,in=162] (-18:2);
      \draw (54:2) to[out=-126,in=126] (-54:2);
      \draw (90:2) to (-90:2);
      % \draw (126:2) to[out=-54,in=54] (-126:2);
      % \draw (162:2) to[out=-18,in=18] (-162:2);
      \draw (162:2) to[out=-18,in=-54] (126:2);
      \draw (-162:2) to[out=18,in=54] (-126:2);
    \end{tikzpicture}\\
        \text{(size 10)}
            \end{array}
            \qquad
            \begin{array}{c}
    \begin{tikzpicture}[scale=.5]
      \draw[dashed] (0,0) circle (2);
      % 18, 54, 90, 126, 162
      \draw (18:2) to[out=-162,in=-126] (54:2);
      \draw (-18:2) to[out=162,in=126] (-54:2);
      \draw (90:2) to (-90:2);
      % \draw (126:2) to[out=-54,in=54] (-126:2);
      % \draw (162:2) to[out=-18,in=18] (-162:2);
      \draw (162:2) to[out=-18,in=-54] (126:2);
      \draw (-162:2) to[out=18,in=54] (-126:2);
    \end{tikzpicture}\\
        \text{(size 5)}
      \end{array}
      \]
The reader can check that there are three remaining orbits (of sizes
10, 10 and 2 respectively). We claim that $JW_5$ is the sum over the
three orbits displayed above (with coefficient $1$); the remaining
orbits (not displayed) have
coefficient zero. It is an exercise to show that this sum satisfies
death by cap. The coefficient of the identity is $1$, so it is the
Jones-Wenzl projector, and it is rotation invariant by
construction. \end{proof}

Now we wish to treat the case when $[2]_s = 0$ and $[2]_t \ne 0$. In such circumstances one expects that $JW_n$ does not exist when $n$ is even. For example, $JW_{2_\reds}$ does not exist since $[2]_s$ is not invertible. However, generically one expects that $JW_n$ exists when $n$ is odd. In order to study odd Jones-Wenzl projectors in the absence of even Jones-Wenzl projectors, we prove a two-step recursion formula.

In this section, when $JW_{n_\reds}$ exists, we write
\begin{equation}
  \label{eq:p1def}
  p_1^{(n)} = \text{coefficient of the identity in }
    \begin{array}{c}
      \begin{tikzpicture}[xscale=.6,yscale=1]
        \begin{scope}
          \clip (-1.8,1) rectangle (2.5,-1);
  \draw[fill=red!20!white] (.5,2) rectangle (3,-2);
  \draw[fill=blue!20!white] (1,.6) to[out=90,in=180] (1.5,.8)
  to[out=0,in=90] (2,0) to[out=-90,in=0] (1.5,-.8) to[out=180,in=-90]
  (1,-.6) to (1,.6);
  \draw[fill=purple!20!white] (-1,2) rectangle (-2,-2);        
  \draw[fill=white] (-1.4,.6) rectangle (1.4,-.6);
  \node at (0,0) {$JW_{n_{\color{red}s}}$};
  \node at (-.25,.8) {$\dots$};
  \node at (-.25,-.8) {$\dots$};
  \end{scope}
  \draw[dashed] (-1.8,1) to (2.5,1);
  \draw[dashed] (-1.8,-1) to (2.5,-1);
\end{tikzpicture}
      \end{array}, 
\end{equation}
\begin{equation}
  \label{eq:p2def}
  p_2^{(n)} = \text{coefficient of the identity in }
    \begin{array}{c}
      \begin{tikzpicture}[xscale=.6,yscale=1]
        \begin{scope}
          \clip (-1.8,1) rectangle (2.5,-1);
  \draw[fill=red!20!white] (.5,2) rectangle (3,-2);
  \draw[fill=blue!20!white] (.8,.6) to[out=90,in=180] (1.5,.9)
  to[out=0,in=90] (2.3,0) to[out=-90,in=0] (1.5,-.9) to[out=180,in=-90]
  (.8,-.6) to (.8,.6);
  \draw[fill=red!20!white] (1.2,.6) to[out=90,in=180] (1.5,.72)
  to[out=0,in=90] (1.8,0) to[out=-90,in=0] (1.5,-.72) to[out=180,in=-90]
  (1.2,-.6) to (1.2,.6);
  \draw[fill=purple!20!white] (-1,2) rectangle (-2,-2);        
  \draw[fill=white] (-1.4,.6) rectangle (1.4,-.6);
  \node at (0,0) {$JW_{n_{\color{red}s}}$};
  \node at (-.25,.8) {$\dots$};
  \node at (-.25,-.8) {$\dots$};
  \end{scope}
  \draw[dashed] (-1.8,1) to (2.5,1);
  \draw[dashed] (-1.8,-1) to (2.5,-1);
\end{tikzpicture}
      \end{array}.
\end{equation}

Just as the diagram in \eqref{eq:p1def} is called the \emph{partial trace} of $JW_{n_{\reds}}$, the diagram in \eqref{eq:p2def} will be called the \emph{double (partial) trace} of $JW_{n_{\reds}}$. Note that this double trace satisfies death by caps. Thus if $JW_{(n-2)_\reds}$ exists, then
\begin{equation} \label{eq:doubletraceJW}
	    \begin{array}{c}
	      \begin{tikzpicture}[xscale=.6,yscale=1]
	        \begin{scope}
	          \clip (-1.8,1) rectangle (2.5,-1);
	  \draw[fill=red!20!white] (.5,2) rectangle (3,-2);
	  \draw[fill=blue!20!white] (.8,.6) to[out=90,in=180] (1.5,.9)
	  to[out=0,in=90] (2.3,0) to[out=-90,in=0] (1.5,-.9) to[out=180,in=-90]
	  (.8,-.6) to (.8,.6);
	  \draw[fill=red!20!white] (1.2,.6) to[out=90,in=180] (1.5,.72)
	  to[out=0,in=90] (1.8,0) to[out=-90,in=0] (1.5,-.72) to[out=180,in=-90]
	  (1.2,-.6) to (1.2,.6);
	  \draw[fill=purple!20!white] (-1,2) rectangle (-2,-2);        
	  \draw[fill=white] (-1.4,.6) rectangle (1.4,-.6);
	  \node at (0,0) {$JW_{n_{\color{red}s}}$};
	  \node at (-.25,.8) {$\dots$};
	  \node at (-.25,-.8) {$\dots$};
	  \end{scope}
	  \draw[dashed] (-1.8,1) to (2.5,1);
	  \draw[dashed] (-1.8,-1) to (2.5,-1);
	\end{tikzpicture}
	      \end{array}
		= p_2^{(n)}
		                  \begin{array}{c}
		        \begin{tikzpicture}[xscale=.3,yscale=0.25]
		          \begin{scope}
		            \clip (-3,3.5) rectangle (3,-3.5);
		            \draw[fill=red!20!white] (0,4) rectangle (3,-4);
		            \draw[fill=blue!20!white] (1,4) rectangle (2,-4);
		         %    \draw[fill=red!20!white] (1,4) to (1,1) to[out=-90,in=180] (1.5,0.5) to[out=0,in=-90] (2,1) to (2,4) to (1,4);
		         % \draw[fill=red!20!white] (1,-4) to (1,-1) to[out=90,in=180] (1.5,-0.5) to[out=0,in=90] (2,-1) to (2,-4) to (1,-4);            
		         \draw[fill=purple!20!white] (-4,4) rectangle (-2,-4);
		         \draw[fill=white] (-2.5,2) rectangle (2.5,-2);
		         \node at (0,0) {$JW_{(n-2)_{\color{red}s}}$};
		%         \node at (-1,0) {$\dots$};
		         \node at (-1,2.7) {$\dots$};
		         \node at (-1,-2.7) {$\dots$};
		       \end{scope}
		       \draw[dashed] (-3,3.5) to (3,3.5);
		       \draw[dashed] (-3,-3.5) to (3,-3.5);
		     \end{tikzpicture}
		                  \end{array}.
\end{equation}
Conversely, if $p_2^{(n)}$ is invertible, then \eqref{eq:doubletraceJW} can be used to define $JW_{(n-2)_\reds}$ from $JW_{n_\reds}$.

\begin{thm} \label{thm:twosteprecursion} Suppose that $JW_{n_{\reds}}$ exists, and consider the morphism $J$ defined below.
\begin{gather*} J = 
            \begin{array}{c}
        \begin{tikzpicture}[xscale=.3,yscale=0.25]
          \begin{scope}
            \clip (-3,3.5) rectangle (2.5,-3.5);
            \draw[fill=red!20!white] (0,4) rectangle (3,-4);
            \draw[fill=blue!20!white] (1,4) rectangle (2,-4);
         \draw[fill=purple!20!white] (-4,4) rectangle (-2,-4);
         \draw[fill=white] (-2.5,2) rectangle (0.5,-2);
         \node at (-.9,0) {$JW_{n_{\color{red}s}}$};
%         \node at (-1,0) {$\dots$};
         \node at (-1,2.7) {$\dots$};
         \node at (-1,-2.7) {$\dots$};
       \end{scope}
       \draw[dashed] (-3,3.5) to (2.5,3.5);
       \draw[dashed] (-3,-3.5) to (2.5,-3.5);
     \end{tikzpicture}
      \end{array}
+        a_{11}^{(n+2)}
      \begin{array}{c}
        \begin{tikzpicture}[xscale=.3,yscale=0.25]
          \begin{scope}
            \clip (-3,3.5) rectangle (3.5,-3.5);
            \draw[fill=blue!20!white] (0,4) rectangle (3,-4);
          \draw[fill=red!20!white] (3,4) rectangle (4,-4);
         \draw[fill=red!20!white] (1,4) to (1,1) to[out=-90,in=180] (1.5,0.5) to[out=0,in=-90] (2,1) to (2,4) to (1,4);
         \draw[fill=red!20!white] (1,-4) to (1,-1) to[out=90,in=180] (1.5,-0.5) to[out=0,in=90] (2,-1) to (2,-4) to (1,-4);            
         \draw[fill=purple!20!white] (-4,4) rectangle (-2,-4);
         \draw[fill=white] (-2.5,3) rectangle (1.5,1);
         \draw[fill=white] (-2.5,-3) rectangle (1.5,-1);
         \node at (-.5,2) {$JW_{n_{\color{red}s}}$};
         \node at (-.5,-2) {$JW_{n_{\color{red}s}}$};
         \node at (-1,0) {$\dots$};
         \node at (-1,3.25) {$\dots$};
         \node at (-1,-3.25) {$\dots$};
       \end{scope}
       \draw[dashed] (-3,3.5) to (3.5,3.5);
       \draw[dashed] (-3,-3.5) to (3.5,-3.5);
     \end{tikzpicture}
      \end{array}
      +        a_{12}^{(n+2)}
      \begin{array}{c}
        \begin{tikzpicture}[xscale=.3,yscale=0.25]
          \begin{scope}
            \clip (-3,3.5) rectangle (3.5,-3.5);
            \draw[fill=blue!20!white] (0,4) rectangle (4,-4);
%          \draw[fill=red!20!white] (3,4) rectangle (4,-4);
         \draw[fill=red!20!white] (1,4) to (1,1) to[out=-90,in=180] (1.5,0.5) to[out=0,in=-90] (2,1) to (2,4) to (1,4);
         \draw[fill=red!20!white] (1,-4) to (1,-1) to[out=90,in=180]
         (1.5,-0.5) to[out=0,in=-90] (3,4) to (4,4) to (4,-4) to (3,-4)
         to[out=90,in=0] (2.5,-2) to[out=180,in=90] (2,-4) to (1,-4); 
         \draw[fill=purple!20!white] (-4,4) rectangle (-2,-4);
         \draw[fill=white] (-2.5,3) rectangle (1.5,1);
         \draw[fill=white] (-2.5,-3) rectangle (1.5,-1);
         \node at (-.5,2) {$JW_{n_{\color{red}s}}$};
         \node at (-.5,-2) {$JW_{n_{\color{red}s}}$};
         \node at (-1,0) {$\dots$};
         \node at (-1,3.25) {$\dots$};
         \node at (-1,-3.25) {$\dots$};
       \end{scope}
       \draw[dashed] (-3,3.5) to (3.5,3.5);
       \draw[dashed] (-3,-3.5) to (3.5,-3.5);
     \end{tikzpicture}
      \end{array}
            +    \\    + a_{21}^{(n+2)}
      \begin{array}{c}
        \begin{tikzpicture}[xscale=.3,yscale=-0.25]
          \begin{scope}
            \clip (-3,3.5) rectangle (3.5,-3.5);
            \draw[fill=blue!20!white] (0,4) rectangle (4,-4);
%          \draw[fill=red!20!white] (3,4) rectangle (4,-4);
         \draw[fill=red!20!white] (1,4) to (1,1) to[out=-90,in=180] (1.5,0.5) to[out=0,in=-90] (2,1) to (2,4) to (1,4);
         \draw[fill=red!20!white] (1,-4) to (1,-1) to[out=90,in=180]
         (1.5,-0.5) to[out=0,in=-90] (3,4) to (4,4) to (4,-4) to (3,-4)
         to[out=90,in=0] (2.5,-2) to[out=180,in=90] (2,-4) to (1,-4); 
         \draw[fill=purple!20!white] (-4,4) rectangle (-2,-4);
         \draw[fill=white] (-2.5,3) rectangle (1.5,1);
         \draw[fill=white] (-2.5,-3) rectangle (1.5,-1);
         \node at (-.5,2) {$JW_{n_{\color{red}s}}$};
         \node at (-.5,-2) {$JW_{n_{\color{red}s}}$};
         \node at (-1,0) {$\dots$};
         \node at (-1,3.25) {$\dots$};
         \node at (-1,-3.25) {$\dots$};
       \end{scope}
       \draw[dashed] (-3,3.5) to (3.5,3.5);
       \draw[dashed] (-3,-3.5) to (3.5,-3.5);
     \end{tikzpicture}
      \end{array} +
       a_{22}^{(n+2)}
      \begin{array}{c}
        \begin{tikzpicture}[xscale=.3,yscale=-0.25]
          \begin{scope}
            \clip (-3,3.5) rectangle (3.5,-3.5);
            \draw[fill=blue!20!white] (0,4) rectangle (4,-4);
%          \draw[fill=red!20!white] (3,4) rectangle (4,-4);
%         \draw[fill=red!20!white] (1,4) to (1,1) to[out=-90,in=180] (1.5,0.5) to[out=0,in=-90] (2,1) to (2,4) to (1,4);
         \draw[fill=red!20!white] (1,-4) to (1,4) to (2,4)
         to[out=-90,in=180] (2.5,2) to[out=0,in=-90] (3,4) to (4,4) to (4,-4) to (3,-4)
         to[out=90,in=0] (2.5,-2) to[out=180,in=90] (2,-4) to (1,-4);
         \draw[fill=purple!20!white] (-4,4) rectangle (-2,-4);
         \draw[fill=white] (-2.5,3) rectangle (1.5,1);
         \draw[fill=white] (-2.5,-3) rectangle (1.5,-1);
         \node at (-.5,2) {$JW_{n_{\color{red}s}}$};
         \node at (-.5,-2) {$JW_{n_{\color{red}s}}$};
         \node at (-1,0) {$\dots$};
         \node at (-1,3.25) {$\dots$};
         \node at (-1,-3.25) {$\dots$};
       \end{scope}
       \draw[dashed] (-3,3.5) to (3.5,3.5);
       \draw[dashed] (-3,-3.5) to (3.5,-3.5);
     \end{tikzpicture}
      \end{array}+
       b^{(n+2)}      \begin{array}{c}
        \begin{tikzpicture}[xscale=.3,yscale=0.25]
          \begin{scope}
            \clip (-3,3.5) rectangle (3.5,-3.5);
            \draw[white, fill=red!20!white] (-.5,4) rectangle (4,-4);
            \draw[fill=blue!20!white] (0,4) to (0,1) to[out=-90,in=180] (1.5,.25) to[out=0,in=-90] (3,1) to (3,4) to (0,4);
            \draw[fill=blue!20!white] (0,-4) to (0,-1) to[out=90,in=180] (1.5,-.25) to[out=0,in=90] (3,-1) to (3,-4) to (0,-4);
          \draw[fill=red!20!white] (1,4) to (1,1) to[out=-90,in=180] (1.5,0.5) to[out=0,in=-90] (2,1) to (2,4) to (1,4);
          \draw[fill=red!20!white] (1,-4) to (1,-1) to[out=90,in=180] (1.5,-0.5) to[out=0,in=90] (2,-1) to (2,-4) to (1,-4);            
         \draw[fill=purple!20!white] (-4,4) rectangle (-2,-4);
         \draw[fill=white] (-2.5,3) rectangle (1.5,1);
         \draw[fill=white] (-2.5,-3) rectangle (1.5,-1);
         \node at (-.5,2) {$JW_{n_{\color{red}s}}$};
         \node at (-.5,-2) {$JW_{n_{\color{red}s}}$};
         \node at (-1,0) {$\dots$};
         \node at (-1,3.25) {$\dots$};
         \node at (-1,-3.25) {$\dots$};
       \end{scope}
       \draw[dashed] (-3,3.5) to (3.5,3.5);
       \draw[dashed] (-3,-3.5) to (3.5,-3.5);
     \end{tikzpicture}
      \end{array}
  \end{gather*}
Then $J$ satisfies death by cap (and is therefore equal to $JW_{(n+2)_\reds}$) if and only if the following equations hold.
\begin{subequations} \label{deathbycapconditionstwostep} 
\begin{equation} 1 + a_{12}^{(n+2)} - [2]_s a_{22}^{(n+2)} = 0, \qquad 1 + a_{21}^{(n+2)} + p_1^{(n)} a_{11}^{(n+2)} = 0, \end{equation}
\begin{equation} a_{22}^{(n+2)} + p_1^{(n)} a_{12}^{(n+2)} = 0, \qquad a_{11}^{(n+2)} - [2]_s a_{21}^{(n+2)} = 0, \end{equation}
\begin{equation} 1 + p_1^{(n)} a_{11}^{(n+2)} + p_2^{(n)} b^{(n+2)} = 0. \end{equation}
\end{subequations}
These equations have a solution if and only if $1 + [2]_s p_1^{(n)}$ is invertible and $p_2^{(n)}$ is invertible. In this case, letting 
\begin{equation} \label{kappadefn} \kappa^{(n+2)} = \kappa := \frac{-1}{1 + [2]_s p_1^{(n)}}, \end{equation} the unique solution is
\begin{gather*} \label{JWrecursiontwostep}
	\begin{array}{c}
        \begin{tikzpicture}[xscale=.3,yscale=0.25]
          \begin{scope}
            \clip (-3,3.5) rectangle (3,-3.5);
            \draw[fill=red!20!white] (0,4) rectangle (3,-4);
            \draw[fill=blue!20!white] (1,4) rectangle (2,-4);
         %    \draw[fill=red!20!white] (1,4) to (1,1) to[out=-90,in=180] (1.5,0.5) to[out=0,in=-90] (2,1) to (2,4) to (1,4);
         % \draw[fill=red!20!white] (1,-4) to (1,-1) to[out=90,in=180] (1.5,-0.5) to[out=0,in=90] (2,-1) to (2,-4) to (1,-4);            
         \draw[fill=purple!20!white] (-4,4) rectangle (-2,-4);
         \draw[fill=white] (-2.5,2) rectangle (2.5,-2);
         \node at (0,0) {$JW_{(n+2)_{\color{red}s}}$};
%         \node at (-1,0) {$\dots$};
         \node at (-1,2.7) {$\dots$};
         \node at (-1,-2.7) {$\dots$};
       \end{scope}
       \draw[dashed] (-3,3.5) to (3,3.5);
       \draw[dashed] (-3,-3.5) to (3,-3.5);
     \end{tikzpicture}
                  \end{array} = 
            \begin{array}{c}
        \begin{tikzpicture}[xscale=.3,yscale=0.25]
          \begin{scope}
            \clip (-3,3.5) rectangle (2.5,-3.5);
            \draw[fill=red!20!white] (0,4) rectangle (3,-4);
            \draw[fill=blue!20!white] (1,4) rectangle (2,-4);
         \draw[fill=purple!20!white] (-4,4) rectangle (-2,-4);
         \draw[fill=white] (-2.5,2) rectangle (0.5,-2);
         \node at (-.9,0) {$JW_{n_{\color{red}s}}$};
%         \node at (-1,0) {$\dots$};
         \node at (-1,2.7) {$\dots$};
         \node at (-1,-2.7) {$\dots$};
       \end{scope}
       \draw[dashed] (-3,3.5) to (2.5,3.5);
       \draw[dashed] (-3,-3.5) to (2.5,-3.5);
     \end{tikzpicture}
      \end{array}
+        [2]_s \kappa
      \begin{array}{c}
        \begin{tikzpicture}[xscale=.3,yscale=0.25]
          \begin{scope}
            \clip (-3,3.5) rectangle (3.5,-3.5);
            \draw[fill=blue!20!white] (0,4) rectangle (3,-4);
          \draw[fill=red!20!white] (3,4) rectangle (4,-4);
         \draw[fill=red!20!white] (1,4) to (1,1) to[out=-90,in=180] (1.5,0.5) to[out=0,in=-90] (2,1) to (2,4) to (1,4);
         \draw[fill=red!20!white] (1,-4) to (1,-1) to[out=90,in=180] (1.5,-0.5) to[out=0,in=90] (2,-1) to (2,-4) to (1,-4);            
         \draw[fill=purple!20!white] (-4,4) rectangle (-2,-4);
         \draw[fill=white] (-2.5,3) rectangle (1.5,1);
         \draw[fill=white] (-2.5,-3) rectangle (1.5,-1);
         \node at (-.5,2) {$JW_{n_{\color{red}s}}$};
         \node at (-.5,-2) {$JW_{n_{\color{red}s}}$};
         \node at (-1,0) {$\dots$};
         \node at (-1,3.25) {$\dots$};
         \node at (-1,-3.25) {$\dots$};
       \end{scope}
       \draw[dashed] (-3,3.5) to (3.5,3.5);
       \draw[dashed] (-3,-3.5) to (3.5,-3.5);
     \end{tikzpicture}
      \end{array}
      +        \kappa
      \begin{array}{c}
        \begin{tikzpicture}[xscale=.3,yscale=0.25]
          \begin{scope}
            \clip (-3,3.5) rectangle (3.5,-3.5);
            \draw[fill=blue!20!white] (0,4) rectangle (4,-4);
%          \draw[fill=red!20!white] (3,4) rectangle (4,-4);
         \draw[fill=red!20!white] (1,4) to (1,1) to[out=-90,in=180] (1.5,0.5) to[out=0,in=-90] (2,1) to (2,4) to (1,4);
         \draw[fill=red!20!white] (1,-4) to (1,-1) to[out=90,in=180]
         (1.5,-0.5) to[out=0,in=-90] (3,4) to (4,4) to (4,-4) to (3,-4)
         to[out=90,in=0] (2.5,-2) to[out=180,in=90] (2,-4) to (1,-4); 
         \draw[fill=purple!20!white] (-4,4) rectangle (-2,-4);
         \draw[fill=white] (-2.5,3) rectangle (1.5,1);
         \draw[fill=white] (-2.5,-3) rectangle (1.5,-1);
         \node at (-.5,2) {$JW_{n_{\color{red}s}}$};
         \node at (-.5,-2) {$JW_{n_{\color{red}s}}$};
         \node at (-1,0) {$\dots$};
         \node at (-1,3.25) {$\dots$};
         \node at (-1,-3.25) {$\dots$};
       \end{scope}
       \draw[dashed] (-3,3.5) to (3.5,3.5);
       \draw[dashed] (-3,-3.5) to (3.5,-3.5);
     \end{tikzpicture}
      \end{array}
            +    \\    + \kappa
      \begin{array}{c}
        \begin{tikzpicture}[xscale=.3,yscale=-0.25]
          \begin{scope}
            \clip (-3,3.5) rectangle (3.5,-3.5);
            \draw[fill=blue!20!white] (0,4) rectangle (4,-4);
%          \draw[fill=red!20!white] (3,4) rectangle (4,-4);
         \draw[fill=red!20!white] (1,4) to (1,1) to[out=-90,in=180] (1.5,0.5) to[out=0,in=-90] (2,1) to (2,4) to (1,4);
         \draw[fill=red!20!white] (1,-4) to (1,-1) to[out=90,in=180]
         (1.5,-0.5) to[out=0,in=-90] (3,4) to (4,4) to (4,-4) to (3,-4)
         to[out=90,in=0] (2.5,-2) to[out=180,in=90] (2,-4) to (1,-4); 
         \draw[fill=purple!20!white] (-4,4) rectangle (-2,-4);
         \draw[fill=white] (-2.5,3) rectangle (1.5,1);
         \draw[fill=white] (-2.5,-3) rectangle (1.5,-1);
         \node at (-.5,2) {$JW_{n_{\color{red}s}}$};
         \node at (-.5,-2) {$JW_{n_{\color{red}s}}$};
         \node at (-1,0) {$\dots$};
         \node at (-1,3.25) {$\dots$};
         \node at (-1,-3.25) {$\dots$};
       \end{scope}
       \draw[dashed] (-3,3.5) to (3.5,3.5);
       \draw[dashed] (-3,-3.5) to (3.5,-3.5);
     \end{tikzpicture}
      \end{array} -
       p_1^{(n)} \kappa
      \begin{array}{c}
        \begin{tikzpicture}[xscale=.3,yscale=-0.25]
          \begin{scope}
            \clip (-3,3.5) rectangle (3.5,-3.5);
            \draw[fill=blue!20!white] (0,4) rectangle (4,-4);
%          \draw[fill=red!20!white] (3,4) rectangle (4,-4);
%         \draw[fill=red!20!white] (1,4) to (1,1) to[out=-90,in=180] (1.5,0.5) to[out=0,in=-90] (2,1) to (2,4) to (1,4);
         \draw[fill=red!20!white] (1,-4) to (1,4) to (2,4)
         to[out=-90,in=180] (2.5,2) to[out=0,in=-90] (3,4) to (4,4) to (4,-4) to (3,-4)
         to[out=90,in=0] (2.5,-2) to[out=180,in=90] (2,-4) to (1,-4);
         \draw[fill=purple!20!white] (-4,4) rectangle (-2,-4);
         \draw[fill=white] (-2.5,3) rectangle (1.5,1);
         \draw[fill=white] (-2.5,-3) rectangle (1.5,-1);
         \node at (-.5,2) {$JW_{n_{\color{red}s}}$};
         \node at (-.5,-2) {$JW_{n_{\color{red}s}}$};
         \node at (-1,0) {$\dots$};
         \node at (-1,3.25) {$\dots$};
         \node at (-1,-3.25) {$\dots$};
       \end{scope}
       \draw[dashed] (-3,3.5) to (3.5,3.5);
       \draw[dashed] (-3,-3.5) to (3.5,-3.5);
     \end{tikzpicture}
      \end{array}+
       \frac{\kappa}{p_2^{(n)}}      \begin{array}{c}
        \begin{tikzpicture}[xscale=.3,yscale=0.25]
          \begin{scope}
            \clip (-3,3.5) rectangle (3.5,-3.5);
            \draw[white, fill=red!20!white] (-.5,4) rectangle (4,-4);
            \draw[fill=blue!20!white] (0,4) to (0,1) to[out=-90,in=180] (1.5,.25) to[out=0,in=-90] (3,1) to (3,4) to (0,4);
            \draw[fill=blue!20!white] (0,-4) to (0,-1) to[out=90,in=180] (1.5,-.25) to[out=0,in=90] (3,-1) to (3,-4) to (0,-4);
          \draw[fill=red!20!white] (1,4) to (1,1) to[out=-90,in=180] (1.5,0.5) to[out=0,in=-90] (2,1) to (2,4) to (1,4);
          \draw[fill=red!20!white] (1,-4) to (1,-1) to[out=90,in=180] (1.5,-0.5) to[out=0,in=90] (2,-1) to (2,-4) to (1,-4);            
         \draw[fill=purple!20!white] (-4,4) rectangle (-2,-4);
         \draw[fill=white] (-2.5,3) rectangle (1.5,1);
         \draw[fill=white] (-2.5,-3) rectangle (1.5,-1);
         \node at (-.5,2) {$JW_{n_{\color{red}s}}$};
         \node at (-.5,-2) {$JW_{n_{\color{red}s}}$};
         \node at (-1,0) {$\dots$};
         \node at (-1,3.25) {$\dots$};
         \node at (-1,-3.25) {$\dots$};
       \end{scope}
       \draw[dashed] (-3,3.5) to (3.5,3.5);
       \draw[dashed] (-3,-3.5) to (3.5,-3.5);
     \end{tikzpicture}
      \end{array}
  \end{gather*}
Moreover, we have
\begin{equation} \label{eq:p1p2recursion} p_1^{(n+2)} = -[2]_t - p_1^{(n)} \kappa, \qquad p_2^{(n+2)} = [3] + \kappa(-[3] + \frac{1}{p_2^{(n)}}). \end{equation}  
\end{thm}

\begin{proof} Let us ignore the superscripts $(n+2)$ in the proof. If $J$ satisfies death by cap, then it is killed by composition with any of the three diagrams with coefficients $a_{21}$, $a_{22}$, and $b$. Relatively straightforward computations (exercises for the reader) show that these three compositions are zero if and only if \eqref{deathbycapconditionstwostep} holds.

Let us compute the solution to \eqref{deathbycapconditionstwostep}. Using two equations we get $a_{11} = [2]_s a_{21}$ and $a_{22} = -p_1^{(n)} a_{12}$. Plugging these back in to the first equation we get
\[ a_{12}(1 + [2]_s p_1^{(n)}) = -1, \]
from which we deduce that $a_{12} = \kappa$. Similarly, $a_{21} = \kappa$. 
The last equation becomes
\[ p_2^{(n)} b = -1-[2]_s p_1^{(n)} a_{21} = \frac{-1-[2]_s p_1^{(n)} + [2]_s p_1^{(n)}}{1 + [2]_s p_1^{(n)}} = \kappa. \]
Thus $b = \frac{\kappa}{p_2^{(n)}}$.

Computing $p_1^{(n+2)}$ from \eqref{JWrecursiontwostep} is easy. The only diagrams which could possibly yield a multiple of the identity after taking the partial trace are the first diagram (which produces $-[2]_t$), and the diagram with coefficient $a_{22}$.

Computing $p_2^{(n+2)}$ from \eqref{JWrecursiontwostep} is relatively straightforward. Once is computing the coefficient of $JW_{n_{\reds}}$ in the double trace of $JW_{(n+2)_{\reds}}$. The first diagram yields $[2]_s [2]_t$, the second diagram yields $-[2]_t a_{11}$, the third diagram yields $a_{12}$ and the fourth diagram yields $a_{21}$, the fifth diagram yields $-[2]_s a_{22}$, and the final diagram yields $b$. Adding these together and performing easy simplifications, one obtains the desired formula \eqref{eq:p1p2recursion}. \end{proof}

Now we solve the recursion set up in Theorem \ref{thm:twosteprecursion}, to deduce the recursive formula for odd Jones-Wenzl projectors.

\begin{thm} \label{thm:twostepsolved} Suppose that $n$ is odd. Suppose that either $[k]_t$ or $[k]_s$ is invertible for all $1 \le k \le n$. For $\ell \ge 3$ odd, set $\frac{[\ell-3]}{[\ell-1]}$ to be either $\frac{[\ell-3]_s}{[\ell-1]_s}$ or $\frac{[\ell-3]_t}{[\ell-1]_t}$, whichever formula has an invertible denominator. Note that if both denominators are invertible then the two formulas agree by \eqref{kover2}. Then $JW_{n_{\reds}}$ exists, and 
\begin{equation} \label{eq:recursionsolved} \kappa^{(n)} = \frac{[n-2]}{[n]}, \qquad b^{(n)} = \frac{[n-2][n-3]}{[n][n-1]}, \qquad p_1^{(n)} = \frac{-[n+1]_t}{[n]}, \qquad p_2^{(n)} = \frac{[n+1]}{[n-1]}. \end{equation}
By swapping the roles of the two colors, $JW_{n_{\bluet}}$ also exists.

If $[n+1]_s = [n+1]_t = 0$, then $JW_{n_{\reds}}$ and $JW_{n_{\bluet}}$ are rotatable.
\end{thm}

\begin{proof} We prove that $JW_{n_{\reds}}$ exists by induction on $n$ odd, using Theorem \ref{thm:twosteprecursion}. Note that $JW_1$ exists, and we set $a_{ij}^{(1)} = 0$ and $b^{(1)} = 0$. It is easy to verify that \eqref{eq:recursionsolved} solves the recursive formulas \eqref{kappadefn} and \eqref{JWrecursiontwostep} and \eqref{eq:p1p2recursion}.

To prove the rotatability statement, we need to prove that the rotation of $JW_n$ satisfies death by caps. Without loss of generality, we assume that $[n-1]_t$ is invertible. We know that $JW_{n-2}$ exists, and $p_1^{(n-2)} = \frac{-[n-1]_t}{[n-2]}$, which is invertible. Thus we can use \eqref{JWrecursion} to prove that $JW_{(n-1)_\bluet}$ exists (although $JW_{(n-1)_{\reds}}$ need not)! Now $p_1^{(n)} = 0$, which combined with \eqref{ptrvsptr1} implies that the partial trace of $JW_{n_{\reds}}$ is the zero morphism. Thus the rotation of $JW_{n_{\reds}}$ by one strand satisfies death by caps, and is a multiple of $JW_{n_{\bluet}}$. Note that the rotation of $JW_{n_{\reds}}$ by two strands also must satisfy death by caps, and thus is a multiple of $JW_{n_{\reds}}$. This implies that the rotation of $JW_{n_{\bluet}}$ by one strand is a multiple of $JW_{n_{\reds}}$ (though we could not use the same argument with colors swapped to deduce this). In conclusion, the Jones-Wenzl projectors $JW_n$ are rotatable. \end{proof}

Now we apply Theorem \ref{thm:twostepsolved} in the special case where $[2]_s = 0$.

\begin{thm} \label{thm:charpcase} Suppose that $[2]_s = 0$ and $[2]_t \ne 0$, and $\Bbbk$ is a field. If the characteristic of $\Bbbk$ is zero, then all odd Jones-Wenzl projectors exist. If the characteristic of $\Bbbk$ is $p$, then $JW_{2k+1}$ exists for all $0 \le k < p$, and $JW_{2p-1}$ is rotatable. \end{thm}

\begin{proof} By \eqref{weirdcasequantums}, all odd quantum numbers are invertible, and $[2k]_t$ is invertible for $k$ less than the characteristic. In characteristic $p$, $[2p]_t = [2p]_s = 0$. Now the result is an immediate consequence of Theorem \ref{thm:twostepsolved}. \end{proof}
	
\subsection{Jones-Wenzl projectors in crystallographic
  type} \label{subsec:ABG} Because they appear most frequently, we wish to summarize the conditions needed for $JW_{m-1}$ to exist and be rotatable, when given a realization of $I_2(m)$ over a domain, with $m
\in \{2,3,4,6\}$. Thanks to Simon Riche for inspiring and assisting with this section.

\begin{lem} In type $A_1 \times A_1$, $JW_1$ always exists and is always rotatable and balanced. \end{lem}
	
\begin{proof} This is obvious.  Note that $\ptr(JW_{1_{\reds}}) = -[2]_s$ and $\ptr(JW_{1_{\bluet}}) = -[2]_t$, and these are both zero by the
definition of a realization. \end{proof}

\begin{lem} In type $A_2$, $JW_2$ always exists and is rotatable. \end{lem}
	
\begin{proof} Note in general that $JW_2$ exists when $[2]_s$ and $[2]_t$ are invertible. Since $JW_1$ always exists, $JW_2$ is rotatable if $\ptr_1(JW_2) = 0$, which holds when $[3]=0$. By the definition of a realization, $[3] = 0$, and consequently $[2]_s [2]_t = 1$, so both are invertible.  \end{proof}
	
Note however when $m=3$ that the realization need not be balanced! There is no requirement that $[2]_s = [2]_t = 1$.

\begin{lem} In type $B_2$, $JW_3$ always exists, and it is rotatable if and only if it is balanced if and only if $[3] = 1$. Moreover, any realization over a domain $\Bbbk$ where $JW_3$ is not rotatable satisfies: \begin{itemize} \item $[2]_s = [2]_t = 0$, and $2 \ne 0$. \end{itemize}  \end{lem}
	
\begin{proof} In general, $JW_3$ exists when $[3]$ is invertible. One can directly compute that $\ptr(JW_{3_{\reds}})$ is equal to the identity diagram with coefficient
$\frac{[4]_t}{[3]}$, and the cup-cap with coefficient $\frac{1 - [3]}{[3]}$. For a realization with $m=4$, we have $[4]_s = [4]_t = 0$ and $[3]^2 = 1$, so $JW_3$ exists, and it is
rotatable if $[3]=1$, which is the definition of balanced.

Since $[3]^2 = 1$ and $\Bbbk$ is a domain, $[3] = \pm 1$. In characteristic $2$, this implies $[3] = 1$ and the realization is balanced. Otherwise, suppose that $[3] = -1 \ne 1$. Then $[2]_s [2]_t = 0$, so at least one of them is zero. If $[2]_s = 0$ and $[2]_t \ne 0$ then by \eqref{weirdcasequantums}, $[4]_t = -2 [2]_t \ne 0$, a contradiction. Thus $[2]_s = [2]_t = 0$. \end{proof}

\begin{lem} In type $G_2$, $JW_5$ exists if and only if $[3]-1$ is invertible. Moreover, for any realization over a domain $\Bbbk$ where $JW_5$ exists, it is rotatable if and only if $[3] = 2$ (which implies that $[5] = 1$, so the realization is balanced). Moreover, any realization over a domain for which $JW_5$ does not exist or is not rotatable satisfies one of the following:
\begin{itemize}
\item $[2]_s = [2]_t = 0$ and $2$ is not invertible ($JW_5$ does not exist).
\item $[2]_s = [2]_t = 0$ and $2$ is invertible but $3 \ne 0$ ($JW_5$ exists but is not rotatable).
\item $[3]=0$ and $2 \ne 0$ ($JW_5$ exists but is not rotatable).
\end{itemize}
 \end{lem}

\begin{proof} It is straightforward if time-consuming to compute $JW_5$ generically. An explicit formula for $JW_{5_{\bluet}}$ can be found in \cite[Figure 2.1]{AMRWmonodromic} One deduces that $JW_5$ exists if and only if $[5]$ and
$\frac{[4]_t}{[2]_t}$ are invertible. They must be invertible for existence, since two of the coefficients appearing are $\mu = \frac{1}{[5]}$ and $\nu = \frac{[2]_t}{[4]_t}$. The remaining
coefficients live in $\ZM[[2]_s, [2]_t] \cdot \mu \nu$. Note that $\frac{[4]_t}{[2]_t} = [3]-1 = \frac{[4]_s}{[2]_s}$, so by symmetry $JW_{5_{\reds}}$ exists if and only if $JW_{5_{\bluet}}$ exists.

When $[6]_s = [6]_t = 0$ (as in any realization of type $G_2$), $[5]$ is invertible since $[5]^2 = 1$. So $JW_5$ exists if and only if $[3]-1$ is invertible.

Note that $[6]_t = [2]_t \cdot [3] \cdot ([3] - 2)$. Over a domain, at least one of these factors must be zero. If $[3] = 0$ or $[3] = 2$ then $[3]-1$ is invertible and $JW_5$ exists. Otherwise, we must have $[2]_s = [2]_t = 0$, in which case $[3] = -1$. Then $[3]-1 = -2$ is invertible if and only if $2$ is invertible. Note that $[2]_s = [2]_t = 0$ in characteristic $2$ yields a well-defined, balanced realization in type $G_2$, but one where $JW_5$ does not exist.

Assume now that $[3]-1$ is invertible. One can explicitly compute $\ptr(JW_{5_{\reds}})$. Every coefficient appearing is either a multiple of $[6]_t$, a
multiple of $\frac{[6]_t}{[2]_t}$, or a multiple of $\frac{[6]_t}{[3]}$. Therefore, $JW_{5_{\reds}}$ is rotatable if and only if $\frac{[6]_t}{[2]_t}$ and $\frac{[6]_t}{[3]}$ are both zero. (Note that this should also be equivalent to the condition that $JW_{5_{\bluet}}$ is rotatable, which requires that $\frac{[6]_s}{[3]} = 0$ as well. We will see below that this is a consequence.)

Both $\frac{[6]_t}{[2]_t}$ and $\frac{[6]_t}{[3]}$ (and $\frac{[6]_s}{[3]}$) have a common factor $[3] - 2$, so if $[3] = 2$ then $JW_5$ is rotatable. Also, if $[3] = 2$ then $[3]^2 = 1 + [3] + [5]$ implies that $[5] = 1$.

Over a domain, if $[3] \ne 2$, then $\frac{[6]_t}{[3]} = 0$ implies $[2]_t = 0$. In this case, using \eqref{weirdcasequantums}, $\frac{[6]_t}{[2]_t} = [5] - [3] + 1 = 3$, so $3 =
0$. But then $[3] = -1 = 2$. So, however you slice it, $[3] = 2$ is required for rotatability. Note that we get a rotatable Jones-Wenzl projector in characteristic $3$ if $[2]_t =
0$ or $[2]_s = 0$ or both, a result slightly more general than Theorem \ref{thm:charpcase}.

Finally, if $[6]_t = [6]_s = 0$ over a domain, but $[3] \ne 2$, then either $[3] = 0 \ne 2$, or $[2]_s = [2]_t = 0$. In the latter case, $[3] = -1 \ne 2$, so $3 \ne 0$.
 \end{proof}

\section{The unbalanced case} \label{sec-unbalanced}

The definition of a balanced and unbalanced realization was given above in \S\ref{subsec:unbalanced}. Let us describe the modifications to the construction and proofs above
required for the unbalanced case. For more details on everything in this section, see the appendix of \cite{EDC}.

\subsection{Positive roots} \label{subsec:rootspolys}

It is more convenient to write \eqref{eq:m-k} in a fashion which depends on the parity of $m$ rather than that of $k$. One has
\begin{equation} \label{eq:m-keven} [m-1][k]_s = [m-k]_s, \quad [m-1][k]_t = [m-k]_t, \quad \text{ if $m$ is even,} \end{equation}
\begin{equation} \label{eq:m-kodd} [m-1]_s [k]_t = [m-k]_s, \quad [m-1]_t [k]_s = [m-k]_t, \quad \text{ if $m$ is odd.} \end{equation}	

Recall from \S\ref{subsec-pos-roots} the definition of $X_s$ and $X_t$. Let us compare the sets
\[ \{ \beta_{\un{w}} \mid \un{w} \in X_s \} \quad \text{and} \quad \{ \beta_{\un{w}} \mid \un{w} \in X_t\}, \]
which (morally speaking) represent two different ways one might attempt to enumerate the ``positive roots.''

Let $\un{x} = (s,t,s,\ldots)$ be an alternating expression which starts in $s$, and has length $k > 0$. One can verify that (see \S\ref{subsec-pos-roots} for the definition of $\beta_{\un{x}}$)
\begin{equation} \label{eq:betax} \beta_{\un{x}} = [k]_s \alpha_s + [k-1]_t \alpha_t. \end{equation}
In particular, if $k=m$ so that $\un{x}$ is a reduced expression for the longest element, then
\begin{equation} \beta_{\un{x}} = [m-1]_t \alpha_t. \end{equation}
This is an indication of the importance of the quantum number $[m-1]$.

Similarly, let $\un{y} = (t, s, t,\ldots)$ be an alternating expression which starts in $t$, and has length $m+1-k > 0$. Then by reversing the roles of $s$ and $t$ we have
\begin{equation} \label{eq:betay} \beta_{\un{y}} = [m-k]_s \alpha_s + [m+1-k]_t \alpha_t. \end{equation}

Now let $1 \le k \le m$ so that both \eqref{eq:betax} and \eqref{eq:betay} hold. The comparison between $\beta_{\un{x}}$ and $\beta_{\un{y}}$ is different based on whether $m$ is even or odd. If $m$ is even then
\begin{equation} \beta_{\un{y}} = [m-1] \cdot ([k]_s \alpha_s + [k-1]_t \alpha_t) = [m-1] \beta_{\un{x}}. \end{equation}
Meanwhile, if $m$ is odd then the comparison depends on whether $k$ is even or odd. If $m$ is odd and $k$ is odd then
\begin{equation} \beta_{\un{y}} = [m-k]_s \alpha_s + [m+1-k] \alpha_t = [m-1]_s [k] \alpha_s + [m-1]_s [k-1]_t \alpha_t = [m-1]_s \cdot \beta_{\un{x}}. \end{equation}
If $m$ is odd and $k$ is even then
\begin{equation} \beta_{\un{y}} = [m-k] \alpha_s + [m+1-k]_t \alpha_t = [m-1]_t [k]_s \alpha_s + [m-1]_t [k-1] \alpha_t = [m-1]_t \cdot \beta_{\un{x}}. \end{equation}
In a balanced realization, $\beta_{\un{y}} = \beta_{\un{x}}$.

Define $\pi_{s,t}$ and $\pi_{t,s}$ as in \eqref{eq:pist} and \eqref{eq:pits}. Thus $\pi_{s,t}$ is the product of terms like $\beta_{\un{x}}$ above, and $\pi_{t,s}$ is the product of terms like $\beta_{\un{y}}$ above. Consequently,
\begin{equation} \pi_{t,s} = [m-1]^m \pi_{s,t} = \pi_{s,t}, \quad \text{ if $m$ is even}, \end{equation}
where we used that $[m-1]^2 = 1$. Similarly, 
\begin{equation} \pi_{t,s} = [m-1]_s [m-1]_t \cdots [m-1]_s \pi_{s,t} = [m-1]_s \pi_{s,t}, \quad \text{ if $m$ is odd}, \end{equation}
where we used that $[m-1]_s [m-1]_t = 1$. Equivalently, 
\begin{equation} \label{eq:pistvspits} \pi_{s,t} = [m-1]_t \pi_{t,s} \quad \text{ if $m$ is odd}. \end{equation}
	
\begin{ex} Suppose that $m_{st} = 3$ and $[2]_t = q^{-1}$ and $[2]_s = q$. Then $s(\a_t) = \a_t + q \a_s$ and $t(\a_s) = \a_s + q^{-1} \a_t$, which differ by a factor of $q$. We
have $\pi_{s,t} = q^{-1} \a_s \a_t s(\a_t)$ and $\pi_{t,s} = \a_s \a_t s(\a_t)$. \end{ex}

\begin{ex} Suppose that $m_{st} = 4$ and $[2]_s = [2]_t = 0$, so that $[3] = -1$. Then $s(\a_t) = \a_t$ and $t(\a_s) = \a_s$. Both $\pi_{s,t}$ and $\pi_{t,s}$ are equal to $\a_s
\a_t (-\a_s) (-\a_t) = \a_s^2 \a_t^2$. \end{ex}

\begin{lem} \label{lem:antiinvtiffevenbalanced} The polynomial $\pi_{s,t}$ (resp. $\pi_{t,s}$) is anti-invariant, meaning that
\begin{equation} s(\pi_{s,t}) = t(\pi_{s,t}) = - \pi_{s,t}, \end{equation}
if and only if the realization is even-balanced. \end{lem}

\begin{proof} Since $\pi_{s,t}$ and $\pi_{t,s}$ are colinear, we need only investigate $\pi_{s,t}$. Apply $s$ to
\[ \pi_{s,t} = \prod_{\un{x} \in X_s} \beta_{\un{x}}. \]
The first term in the product (when $X_s$ is sorted by length) is $\alpha_s$, which goes to $-\alpha_s$. The remaining $m-1$ terms in the product go to the first $m-1$ terms in $\pi_{t,s}$, by the definition of these terms. The final missing term in $\pi_{t,s}$ is $[m-1]_s \alpha_s$. Consequently, 
\begin{equation} s(\pi_{s,t}) = \frac{-\a_s}{[m-1]_s \a_s} \pi_{t,s} = \frac{-1}{[m-1]_s} \pi_{t,s}. \end{equation}
This leads to a major difference depending on the parity of $m$. If $m$ is odd then $\pi_{t,s} = [m-1]_s \pi_{s,t}$ so
\begin{equation} s(\pi_{s,t}) = -\pi_{s,t}, \quad \text{ if $m$ is odd}. \end{equation}
Meanwhile, if $m$ is even then $\pi_{s,t} = \pi_{t,s}$ so
\begin{equation} s(\pi_{s,t}) = -[m-1] \pi_{s,t}, \quad \text{ if $m$ is even}. \end{equation}
So when $m$ is even, $\pi_{s,t}$ is $s$-antiinvariant if and only if the realization is balanced for $\{s,t\}$. This can be witnessed in the examples for $m=3$ and $m=4$ above.
\end{proof}

\subsection{The definition of the category}

Now we remind the reader about some features of $2m$-valent vertices, and how the definition of the category $\HC$ changes in the unbalanced case. To motivate
this story, we hold as self-evident that the Soergel Hom formula should describe the dimension of the morphism spaces in $\HC$.

Pick a pair $\reds \ne \bluet \in S$ with $m = m_{st} < \infty$. Let $w_0$ denote the longest element of this dihedral parabolic subgroup. By the Soergel Hom formula, the space of
degree zero morphisms between any two of these three objects should be one-dimensional: $B_s B_t \cdots$, $B_t B_s \cdots$, and $B_{w_0}$ (both expressions have length $m$). If the
composition map $B_s B_t \cdots \to B_{w_0} \to B_t B_s \cdots$ were zero then the Soergel categorification theorem would fail, as these two expressions would not have a common top
summand. By construction, the $2m$-valent vertex will span the one-dimensional space of maps $B_s B_t \cdots \to B_t B_s \cdots$, so it agrees with the projection to the common summand $B_{w_0}$ up to scalar. Similarly, there is another $2m$-valent vertex spanning the degree zero morphisms from $B_t B_s \cdots$ to $B_s B_t \cdots$.

Here are four features of the $2m$-valent vertices which must be true if the Soergel Categorification Theorem is to hold. \begin{itemize}
\item Composing the $2m$-valent vertices should yield, up to invertible scalar, the idempotent projecting to the common summand $B_{w_0}$.
\item Rotating the $2m$-valent vertex by two strands will produce a scalar multiple of itself. Rotating by one strand will produce a scalar multiple of the other $2m$-valent vertex.
\item The $2m$-valent vertex satisfies death by pitchfork.
\item Placing $2$ adjacent dots on the $2m$-valent vertex will yield a degree $+2$ morphism built from univalent and trivalent vertices, which must be a deformation retract of a linear combination of crossingless matchings (see \cite[\S 5.3.3]{EDC}). So there is some relation like \eqref{Eq:JWreln}. \end{itemize}
What is this linear combination of crossingless matchings appearing in
\eqref{Eq:JWreln}? It must satisfy death by cap, in order for the
$2m$-valent vertex to satisfy death by pitchfork. It must have a
rotational eigenvalue, in order for the $2m$-valent vertex to have one
as well. It must have an invertible coefficient of the identity map,
in order for the composition of the $2m$-valent vertices to be an
idempotent up to invertible scalar. With these constraints, it must be
a scalar multiple of a rotatable Jones-Wenzl projector $JW_{m-1}$. If
there is no rotatable Jones-Wenzl projector, then there is no way to
define $\HC$ such that the properties listed above hold.

Henceforth we assume that $JW_{m-1}$ exists and is rotatable. By Lemma \ref{lem:rotinvce}, this implies that our realization is even-balanced. There is now nothing to discuss for dihedral groups with $m$ even, so we focus on the case where $m$ is odd.

The source of the difficulty in the unbalanced case is that rotating $JW_{(m-1)_{\reds}}$ by one strand yields $JW_{(m-1)_{\bluet}}$ (and vice versa) only up to the scalar factor
$[m-1]_{\reds}^{\pm 1}$. The appendix of \cite{EDC} suggests two diagrammatic options for keeping track of this discrepancy. The first option involves fixing once and for all a
scalar multiple of the Jones-Wenzl projector, and adding scalars to various relations in $\HC$ to make them match this choice. This makes the presentation much less intuitive and
easy to remember (although this option is to be preferred for singular Soergel calculus). The second option involves keeping track of two different versions of the Jones-Wenzl
projector, and two different versions of the $2m$-valent vertex, which agree with each other up to a scalar. This second option is our preference here. The details of how $\HC$
should be constructed in this second option can be found in \cite[\S 5.6]{EW}, having extrapolated the answer from \cite[Appendix]{EDC}. Because some statements (e.g. the analog of
\eqref{eq:alldots}) are not stated explicitly anywhere and would require extrapolation from multiple sources to piece together, we will recall everything in one place here. We use a different convention than \cite{EW} for shading the vertices.

In the definition of $\HC$ by generators and relations, there are now two different versions of the $2m$-valent vertex as morphism generators, one shaded $\reds$ and one shaded
$\bluet$. Both are rotation-invariant. We have
\begin{equation}
  \label{eq:rescale}
  \begin{array}{c}
    \begin{tikzpicture}[scale=0.7]
                    \draw[dashed] (-1.5,-1) to (1.5, -1); \draw[dashed] (-1.5,1) to (1.5, 1);
%\draw (0,1) rectangle (3.5,-1);
\draw[red] (-1,-1) to[out=90,in=-150] (0,0);
\draw[blue] (0,-1) to (0,0);
\draw[red] (1,-1) to[out=90,in=-30] (0,0);
\draw[blue] (-1,1) to[out=-90,in=150] (0,0);
\draw[red] (0,1) to (0,0);
\draw[blue] (1,1) to[out=-90,in=30] (0,0);
\draw[fill=blue,blue] (0,0) circle (3pt);
\end{tikzpicture} \end{array}
    =
    [m-1]_{\reds}
    \begin{array}{c}
        \begin{tikzpicture}[scale=0.7]
                    \draw[dashed] (-1.5,-1) to (1.5, -1); \draw[dashed] (-1.5,1) to (1.5, 1);
%\draw (0,1) rectangle (3.5,-1);
\draw[red] (-1,-1) to[out=90,in=-150] (0,0);
\draw[blue] (0,-1) to (0,0);
\draw[red] (1,-1) to[out=90,in=-30] (0,0);
\draw[blue] (-1,1) to[out=-90,in=150] (0,0);
\draw[red] (0,1) to (0,0);
\draw[blue] (1,1) to[out=-90,in=30] (0,0);
\draw[fill=red,red] (0,0) circle (3pt);
\end{tikzpicture}
    \end{array},
  \end{equation}
so these differently-shaded vertices are equal up to an invertible scalar. Because of this relation any diagram with $\reds$-shaded vertices can be replaced with a diagram with only $\bluet$-shaded vertices by rescaling. It helps to have both shadings in practice, and the relations of the category will be preserved by the symmetry that swaps $\reds$ with $\bluet$ (in diagrams and in coefficients, so $[m-1]_{\reds}$ is swapped with $[m-1]_{\bluet}$).

\begin{remark} The normalization on these vertices designed as follows. When the $\bluet$-shaded one is rotated so that it gives a morphism from $B_\bluet B_\reds \cdots$ to
$B_\reds B_\bluet \cdots$, then after applying $\FC$ it should send the 1-tensor to the 1-tensor. The same is true switching $\reds$ and $\bluet$. This is opposite to the convention used in \cite{EW}. \end{remark}

Now we explain how to modify the relations of $\HC$ accordingly. The
cyclicity relation now states that the $\bluet$-shaded $2m$-valent
vertex rotates to itself:
    \begin{equation}
  \label{eq: shadedcyclicity}
  \begin{array}{c}
    \begin{tikzpicture}[scale=0.7]
                    \draw[dashed] (-2.2,-1) to (2.2, -1); \draw[dashed] (-2.2,1) to (2.2, 1);
%\draw (0,1) rectangle (3.5,-1);
\draw[red] (-2,1) to[out=-90,in=180] (-1,-.5) to[out=0,in=-150] (0,0);
\draw[blue] (0,-1) to (0,0);
\draw[red] (1,-1) to[out=90,in=-30] (0,0);
\draw[blue] (-1,1) to[out=-90,in=150] (0,0);
\draw[red] (0,1) to (0,0);
\draw[blue] (2,-1) to[out=90,in=0] (1,.5) to[out=-180,in=30] (0,0);
\draw[fill=blue,blue] (0,0) circle (3pt);
\end{tikzpicture} \end{array}
=
  \begin{array}{c}
    \begin{tikzpicture}[scale=0.7]
                    \draw[dashed] (-1.5,-1) to (1.5, -1); \draw[dashed] (-1.5,1) to (1.5, 1);
%\draw (0,1) rectangle (3.5,-1);
\draw[blue] (-1,-1) to[out=90,in=-150] (0,0);
\draw[red] (0,-1) to (0,0);
\draw[blue] (1,-1) to[out=90,in=-30] (0,0);
\draw[red] (-1,1) to[out=-90,in=150] (0,0);
\draw[blue] (0,1) to (0,0);
\draw[red] (1,1) to[out=-90,in=30] (0,0);
\draw[fill=blue,blue] (0,0) circle (3pt);
\end{tikzpicture} \end{array}
  \end{equation}
After rescaling this implies that the $\reds$-shaded $2m$-valent vertex also rotates to itself.

The two-colored associativity relation now states the following:
  \begin{equation} \label{shadedassoc}
\begin{array}{c}
  \tikz[scale=0.8]{
              \draw[dashed] (1.5,-2) to (4.5, -2); \draw[dashed] (1.5,1) to (4.5,1);
\draw[color=red] (2.1,-1) to[out=90,in=-150] (3,0) to (3,1); \draw[color=red] (3.9,-2) to[out=90,in=-30] (3,0);
\draw[red] (2.5,-2) to (2.1,-1) to (1.5,-2);
\draw[color=blue] (2.1,1) to[out=-90,in=150] (3,0) to (3,-2);
  \draw[color=blue] (3.9,1) to[out=-90,in=30] (3,0);
\draw[fill=blue,blue] (3,0) circle (3pt);
  }
\end{array}
=
\begin{array}{c}
  \tikz[scale=0.7]{
                \draw[dashed] (1,-1) to (4.5, -1); \draw[dashed] (1,2.5) to (4.5, 2.5);
\draw[color=red] (2,-1) to[out=90,in=-150] (3,0);% to (3,1);
\draw[color=red] (3.9,-1) to[out=90,in=-30] (3,0);
\draw[color=blue] (2.1,1.5) to[out=-90,in=150] (3,0) to (3,-1);
  \draw[color=blue] (3.9,1) to[out=-90,in=30] (3,0);
\draw[blue] (3.9,1) to[out=90,in=-30] (3,2);
\draw[blue] (2.1,1.5) to (3,2) to (3,2.5);
  \draw[blue] (2.1,1.5) to[out=150,in=-90] (1.2,2.5);
  \draw[fill=blue,blue] (3,0) circle (3pt);
\draw[red] (3,0) to[out=90,in=-30] (2.1,1.5);
\draw[red] (1.5,-1) to[out=90,in=-150] (2.1,1.5);
  \draw[red] (2.1,1.5) to (2.1,2.5);
  \draw[fill=blue,blue] (2.1,1.5) circle (3pt);
}
\end{array}
\end{equation}

To state the analog of \eqref{Eq:JWreln}, we set 
\begin{equation} \JWalt_{\reds,\bluet} := \Sigma(JW_{(m-1)_{\reds}}), \qquad \JWalt_{\bluet,\reds} := \Sigma(JW_{(m-1)_{\bluet}}). \end{equation}
Note that $(m-1)_{\reds} = {}_{\reds} (m-1)$, because $m$ is odd. So $\JWalt_{\reds,\bluet}$ is an endomorphism of the alternating Bott-Samelson object $B_{\reds} B_{\bluet} \cdots$ of length $m$, and the coefficient of the identity is $1$. Just as in \S\ref{subsec:JWmain}, we can define morphisms $\JWalttwo_{\reds,\bluet}$, $JW_{\reds, \bluet}$, and $JW'_{\reds, \bluet}$. 

For example, $JW_{\reds,\bluet}$ is a degree $+2$ morphism from $B_{\reds} B_{\bluet} \cdots$ (alternating of length $2(m-1)$) to the monoidal identity. The fact that $JW_{(m-1)_{\reds}}$ has coefficient $+1$ of the identity implies that $JW_{\reds, \bluet}$ has coefficient $+1$ of the diagram
\begin{equation} \label{eq:coeffis1}
    \begin{array}{c}
        \begin{tikzpicture}[yscale=0.5,xscale=-0.7]
                    \draw[dashed] (-2,-1) to (2.5, -1); \draw[dashed] (-2,2) to (2.5, 2);
%\draw (0,1) rectangle (3.5,-1);
\draw[blue] (-1.5,-1) to[out=90,in=180] (0,1.5) to[out=0,in=90] (1.5,-1);
\draw[red] (-1,-1) to[out=90,in=180] (0,1) to[out=0,in=90] (1,-1);
\draw[blue] (-.5,-1) to[out=90,in=180] (0,.5) to[out=0,in=90] (.5,-1);
\draw[red] (0,-1) to[out=90,in=-90] (0,-.2);
\draw[fill=red,red] (0,-.2) circle (3pt);
\draw[red] (2,-1) to[out=90,in=-90] (2,-.2);
\draw[fill=red,red] (2,-.2) circle (3pt);
\end{tikzpicture}
    \end{array}
\end{equation}
Meanwhile, the coefficient in $JW_{\reds, \bluet}$ of the diagram
\begin{equation} \label{eq:coeffisnt1}
    \begin{array}{c}
        \begin{tikzpicture}[yscale=0.5,xscale=0.7]
                    \draw[dashed] (-2,-1) to (2.5, -1); \draw[dashed] (-2,2) to (2.5, 2);
%\draw (0,1) rectangle (3.5,-1);
\draw[red] (-1.5,-1) to[out=90,in=180] (0,1.5) to[out=0,in=90] (1.5,-1);
\draw[blue] (-1,-1) to[out=90,in=180] (0,1) to[out=0,in=90] (1,-1);
\draw[red] (-.5,-1) to[out=90,in=180] (0,.5) to[out=0,in=90] (.5,-1);
\draw[blue] (0,-1) to[out=90,in=-90] (0,-.2);
\draw[fill=blue,blue] (0,-.2) circle (3pt);
\draw[blue] (2,-1) to[out=90,in=-90] (2,-.2);
\draw[fill=blue,blue] (2,-.2) circle (3pt);
\end{tikzpicture}
    \end{array}
\end{equation}
is exactly the coefficient $c_{n-1}$ computed in the proof of Lemma \ref{lem:rotinvce}, namely $[m-1]_{\bluet}$.

Lemma \ref{lem:rotinvce} implies that
 \begin{equation} \label{Eq:whattocallit}
\begin{array}{c}
  \tikz[xscale=0.45,yscale=0.45]{
      \draw[dashed] (-1,-2.5) to (4.5, -2.5); \draw[dashed] (-1,3) to (4.5,3);
\draw (0,1) rectangle (3,-1);
\node at (1.5,0) {$JW_{{\color{red}s},{\color{blue}t}}$};
\draw[red] (.5,-1) to[out=-90,in=0] (0,-1.5) to[out=180,in=-90] (-1,0) to[out=90,in=180] (1.5,2.5) to[out=0,in=90] (3.5,0) to[out=-90,in=90] (3.5,-2.5);
  \draw[blue] (1,-1) to (1,-2.5);
  % to[out=-90,in=0] (0,-2) to[out=180,in=-90] (-1.5,0) to[out=90,in=180] (1.5,3) to[out=0,in=90] (4,0) to[out=-90,in=90] (4,-2.5);
  \draw[red] (1.5,-1) to (1.5,-2.5);
\draw[blue] (2,-1) to (2,-2.5);
%  \draw[red] (2.5,-1) to (2.5,-2.5);
%\node[circle,fill,draw,inner sep=0mm,minimum size=1mm,color=blue] at (1,-2) {};
\node at (2.8,-1.75) {\small $\dots$};
  } \end{array}
=
[m-1]_{\color{blue}t}
\begin{array}{c}
  \tikz[xscale=0.45,yscale=0.45]{
      \draw[dashed] (-1,-2.5) to (4.5, -2.5); \draw[dashed] (-1,3) to
  (4.5,3);
\draw (0,1) rectangle (3,-1);
\node at (1.5,0) {$JW_{{\color{blue}t},{\color{red}s}}$};
\draw[red] (.5,-1) to (.5,-2.5);
\draw[blue] (1,-1) to (1,-2.5);
  \draw[blue] (2,-1) to (2,-2.5);
\draw[red] (2.5,-1) to (2.5,-2.5);
%  \draw[red] (2.5,-1) to (2.5,-2.5);
%\node[circle,fill,draw,inner sep=0mm,minimum size=1mm,color=blue] at (1,-2) {};
\node at (1.55,-1.75) {\small $\dots$};
  } \end{array}
\end{equation}

Lemma \ref{lem:evaluation} implies that   \begin{equation}
    \label{eq:6again}
        \begin{array}{c}
  \begin{tikzpicture}[xscale=.5,yscale=1]
     \draw[dashed] (-1.6,-1) to (1.6, -1);
     \draw[dashed] (-1.6, 1) to (1.6, 1);
     \draw[red] (-1.2,-.8) to (-1.2,.8);
     \draw[blue] (-.7,-.8) to (-.7, .8);
     \node at (0,-.7) {$\dots$};
          \node at (0,.7) {$\dots$};
     \draw[blue] (.7,-.8) to (.7, .8);
     \draw[red] (1.2,-.8) to (1.2, .8);
     \draw[fill=white] (-1.4,-.6) rectangle (1.4,.6);
     \node at (0,0) {$\JWalt_{\reds,\bluet}$};
      \node[circle,fill,draw,inner
      sep=0mm,minimum size=1mm,color=red] at (-1.2,.8) {};
      \node[circle,fill,draw,inner
      sep=0mm,minimum size=1mm,color=red] at (1.2,.8) {};
      \node[circle,fill,draw,inner
      sep=0mm,minimum size=1mm,color=blue] at (-.7,.8) {};
            \node[circle,fill,draw,inner
            sep=0mm,minimum size=1mm,color=blue] at (.7,.8) {};
                  \node[circle,fill,draw,inner
      sep=0mm,minimum size=1mm,color=blue] at (-.7,-.8) {};
            \node[circle,fill,draw,inner
      sep=0mm,minimum size=1mm,color=blue] at (.7,-.8) {};
      \node[circle,fill,draw,inner
      sep=0mm,minimum size=1mm,color=red] at (1.2,-.8) {};
            \node[circle,fill,draw,inner
      sep=0mm,minimum size=1mm,color=red] at (-1.2,-.8) {};
    \end{tikzpicture}
        \end{array}
        =
        \begin{array}{c}
  \begin{tikzpicture}[xscale=.5,yscale=1]
     \draw[dashed] (-1.6,-1) to (1.6, -1);
     \draw[dashed] (-1.6,1) to (1.6,1);
     \draw[red] (-1.2,-.8) to (-1.2,0);
     \draw[blue] (-.7,-.8) to (-.7, .8);
     \node at (0,-.7) {$\dots$};
          \node at (0,.7) {$\dots$};
     \draw[blue] (.7,-.8) to (.7, .8);
     \draw[red] (1.2,0) to (1.2, .8);
     \draw[fill=white] (-1.4,-.6) rectangle (1.4,.6);
     \node at (0,0) {$\JWalttwo_{\reds,\bluet}$};
                       \node[circle,fill,draw,inner
      sep=0mm,minimum size=1mm,color=blue] at (-.7,-.8) {};
            \node[circle,fill,draw,inner
      sep=0mm,minimum size=1mm,color=blue] at (.7,-.8) {};
      \node[circle,fill,draw,inner
      sep=0mm,minimum size=1mm,color=red] at (-1.2,-.8) {};
                             \node[circle,fill,draw,inner
      sep=0mm,minimum size=1mm,color=blue] at (-.7,.8) {};
            \node[circle,fill,draw,inner
      sep=0mm,minimum size=1mm,color=blue] at (.7,.8) {};
      \node[circle,fill,draw,inner
      sep=0mm,minimum size=1mm,color=red] at (1.2,.8) {};
\end{tikzpicture}
        \end{array}
        = \pi_{\reds,\bluet},
      \end{equation}
exactly as for \eqref{eq:6}.

Now we can finally state the relation which resolves what happens when
a dot is placed on a $2m$-valent vertex. Analogous to \eqref{eq:JWlem}
we have
\begin{equation} \label{twocolordotvertex1}
    \begin{array}{c}
    \tikz[xscale=0.45,yscale=0.45]{
    
    \draw[dashed] (-1,-2.5) to (4.5, -2.5);
    \draw[dashed] (-1,2.5) to (4.5,2.5);
    \draw[red] (1.5,0) to[out=150,in=90] (-.5,-2.5);
    \draw[blue] (1.5,0) to[out=-150,in=90] (.25,-2.5);
    \draw[red] (1.5,0) to[out=-90,in=90] (1,-2.5);
    \draw[red] (1.5,0) to[out=-30,in=90] (3,-2.5);
    \draw[blue] (1.5,0) to[out=30,in=-150] (2.5,.5);
          \node[circle,fill,draw,inner sep=0mm,minimum
    size=2mm,color=red] at (1.5,0) {};
              \node[circle,fill,draw,inner sep=0mm,minimum size=1mm,color=blue] at (2.5,.5) {};
    % \draw (0,1) rectangle (3,-1);
%   \node at (1.5,0) {$JW_{{\color{red}s},{\color{blue}t}}$};
%   \draw[red] (.5,-1) to (.5,-2.5);
% \draw[red] (.5,-1.5) to[out=180,in=0] (0,-1.5) to[out=180,in=-90] (-1,0) to[out=90,in=180] (1.5,2) to[out=0,in=90] (3.5,0) to[out=-90,in=90] (3.5,-2.5);
% \draw[blue] (1,-1) to (1,-2.5);
%   \draw[red] (1.5,-1) to (1.5,-2.5);
% \draw[blue] (2,-1) to (2,-2.5);
% %  \draw[red] (2.5,-1) to (2.5,-2.5);
% %\node[circle,fill,draw,inner sep=0mm,minimum size=1mm,color=blue] at (1,-2) {};
\node at (2,-1.75) {\small $\dots$};
    } \end{array}
  =
\begin{array}{c}
  \tikz[xscale=0.45,yscale=0.45]{
      \draw[dashed] (-1,-2.5) to (4.5, -2.5); \draw[dashed] (-1,2.5) to (4.5,2.5);
\draw (0,1) rectangle (3,-1);
  \node at (1.5,0) {$JW_{{\color{red}s},{\color{blue}t}}$};
  \draw[red] (.5,-1) to (.5,-2.5);
\draw[red] (.5,-1.5) to[out=180,in=0] (0,-1.5) to[out=180,in=-90] (-1,0) to[out=90,in=180] (1.5,2) to[out=0,in=90] (3.5,0) to[out=-90,in=90] (3.5,-2.5);
\draw[blue] (1,-1) to (1,-2.5);
  \draw[red] (1.5,-1) to (1.5,-2.5);
\draw[blue] (2,-1) to (2,-2.5);
%  \draw[red] (2.5,-1) to (2.5,-2.5);
%\node[circle,fill,draw,inner sep=0mm,minimum size=1mm,color=blue] at (1,-2) {};
\node at (2.8,-1.75) {\small $\dots$};
  } \end{array}
\end{equation}

We also have
\begin{equation} \label{twocolordotvertex2}
  \begin{array}{c}
    \tikz[xscale=-0.45,yscale=0.45]{
    
    \draw[dashed] (-1,-2.5) to (4.5, -2.5);
    \draw[dashed] (-1,2.5) to (4.5,2.5);
    \draw[blue] (1.5,0) to[out=150,in=90] (-.5,-2.5);
    \draw[red] (1.5,0) to[out=-150,in=90] (.25,-2.5);
    \draw[blue] (1.5,0) to[out=-90,in=90] (1,-2.5);
    \draw[blue] (1.5,0) to[out=-30,in=90] (3,-2.5);
    \draw[red] (1.5,0) to[out=30,in=-150] (2.5,.5);
          \node[circle,fill,draw,inner sep=0mm,minimum
    size=2mm,color=red] at (1.5,0) {};
              \node[circle,fill,draw,inner sep=0mm,minimum size=1mm,color=red] at (2.5,.5) {};
    % \draw (0,1) rectangle (3,-1);
%   \node at (1.5,0) {$JW_{{\color{red}s},{\color{blue}t}}$};
%   \draw[red] (.5,-1) to (.5,-2.5);
% \draw[red] (.5,-1.5) to[out=180,in=0] (0,-1.5) to[out=180,in=-90] (-1,0) to[out=90,in=180] (1.5,2) to[out=0,in=90] (3.5,0) to[out=-90,in=90] (3.5,-2.5);
% \draw[blue] (1,-1) to (1,-2.5);
%   \draw[red] (1.5,-1) to (1.5,-2.5);
% \draw[blue] (2,-1) to (2,-2.5);
% %  \draw[red] (2.5,-1) to (2.5,-2.5);
% %\node[circle,fill,draw,inner sep=0mm,minimum size=1mm,color=blue] at (1,-2) {};
\node at (2,-1.75) {\small $\dots$};
    } \end{array}
  =
\begin{array}{c}
  \tikz[xscale=-0.45,yscale=0.45]{
      \draw[dashed] (-1,-2.5) to (4.5, -2.5); \draw[dashed] (-1,2.5) to (4.5,2.5);
\draw (0,1) rectangle (3,-1);
  \node at (1.5,0) {$JW_{{\color{red}s},{\color{blue}t}}$};
  \draw[blue] (.5,-1) to (.5,-2.5);
\draw[blue] (.5,-1.5) to[out=180,in=0] (0,-1.5) to[out=180,in=-90] (-1,0) to[out=90,in=180] (1.5,2) to[out=0,in=90] (3.5,0) to[out=-90,in=90] (3.5,-2.5);
\draw[red] (1,-1) to (1,-2.5);
  \draw[blue] (1.5,-1) to (1.5,-2.5);
\draw[red] (2,-1) to (2,-2.5);
%  \draw[red] (2.5,-1) to (2.5,-2.5);
%\node[circle,fill,draw,inner sep=0mm,minimum size=1mm,color=blue] at (1,-2) {};
\node at (2.8,-1.75) {\small $\dots$};
  } \end{array}
\end{equation}
Note that the color swap of \eqref{twocolordotvertex1} gives the
relation for a red dot on a blue vertex, which can be rescaled to give
the relation for a red dot on a red vertex, yielding
\eqref{twocolordotvertex2}. We need only check
\eqref{twocolordotvertex1} below, and the other relation will follow
by symmetry.

The analog of \eqref{mosttermsP} is
\begin{equation} \label{mosttermsPshaded}
    \begin{array}{c}
      \begin{tikzpicture}[scale=0.4]
                    \draw[dashed] (-1.5,1.5) to (1.5,1.5);
      \draw[dashed] (-1.5,-1.5) to (1.5,-1.5);
      \node[circle,fill,draw,inner sep=0mm,minimum size=1mm,color=red] at (-1,-1.2) {};
  \draw[color=red] (-1,-1.2) to[out=90,in=-150] (0,0); % (-1,-.7);
  \draw[color=blue] (0,-1.5) to (0,0);
\draw[color=red] (1,-1.5) to[out=90,in=-30] (0,0);
\draw[color=blue] (-1,1.5) to[out=-90,in=150] (0,0);
  \draw[color=red] (0,1.5) to (0,0);
  \draw[color=blue] (1,1.5) to[out=-90,in=30] (0,0);
%  \draw (-1.5,-.7) rectangle (1.5,.7);
%  \node at (0,0) {$G_{s,t}$};
        \node[circle,fill,draw,inner sep=0mm,minimum size=2mm,color=blue] at (0,0) {};
\end{tikzpicture}
  \end{array}
  =
    \begin{array}{c}
      \begin{tikzpicture}[scale=0.4]
                    \draw[dashed] (-1.5,1.5) to (1.5,1.5);
      \draw[dashed] (-1.5,-1.5) to (1.5,-1.5);
%   \draw[color=red] (-1,-1.2) to (-1,-.7);
%   \draw[color=blue] (0,-1.5) to (0,-.7);
% \draw[color=red] (1,-1.5) to (1,-.7);
\draw[color=blue] (-1,1.5) to (-1,-1.5);
  \draw[color=red] (0,1.5) to (0, -1.5);
  \draw[color=blue] (1,1.5) to (1,.7);
%  \draw (-1.5,.7) rectangle (1.5,.7);
      \node[circle,fill,draw,inner sep=0mm,minimum size=1mm,color=blue] at (1,.7) {};
    \end{tikzpicture}
  \end{array}
      +
          \begin{array}{c}
            \begin{tikzpicture}[scale=0.4]
                          \draw[dashed] (-1.5,1.5) to (2,1.5);
      \draw[dashed] (-1.5,-1.5) to (2,-1.5);
%   \draw[color=red] (-1,-1.2) to (-1,-.7);
%   \draw[color=blue] (0,-1.5) to (0,-.7);
% \draw[color=red] (1,-1.5) to (1,-.7);
\draw[color=blue] (-1,1.5) to (-1,-1.5);
  \draw[color=red] (0,1.5) to (0, -1.5);
  \draw[color=blue] (1,1.5) to (1,.7);
  \draw[fill=white] (-1.5,-.7) rectangle (1.5,.7);
  \node at (0,0) {$P$};
%      \node[circle,fill,draw,inner sep=0mm,minimum size=1mm,color=blue] at (1,.7) {};
    \end{tikzpicture}
  \end{array}.
\end{equation}
Meanwhile,
\begin{equation}
    \begin{array}{c}
    \begin{tikzpicture}[scale=0.4]
      \node[circle,fill,draw,inner sep=0mm,minimum size=1mm,color=red]
      at (-1,-1.2) {};
            \draw[dashed] (-1.5,1.5) to (1.5,1.5);
      \draw[dashed] (-1.5,-1.5) to (1.5,-1.5);
  \draw[color=red] (-1,-1.2) to[out=90,in=-150] (0,0); % (-1,-.7);
  \draw[color=blue] (0,-1.5) to (0,0);
\draw[color=red] (1,-1.5) to[out=90,in=-30] (0,0);
\draw[color=blue] (-1,1.5) to[out=-90,in=150] (0,0);
  \draw[color=red] (0,1.5) to (0,0);
  \draw[color=blue] (1,1.5) to[out=-90,in=30] (0,0);
%  \draw (-1.5,-.7) rectangle (1.5,.7);
%  \node at (0,0) {$G_{s,t}$};
        \node[circle,fill,draw,inner sep=0mm,minimum size=2mm,color=red] at (0,0) {};
\end{tikzpicture}
  \end{array}
  = [m-1]_{\bluet}
    \begin{array}{c}
      \begin{tikzpicture}[scale=0.4]
                    \draw[dashed] (-1.5,1.5) to (1.5,1.5);
      \draw[dashed] (-1.5,-1.5) to (1.5,-1.5);
%   \draw[color=red] (-1,-1.2) to (-1,-.7);
%   \draw[color=blue] (0,-1.5) to (0,-.7);
% \draw[color=red] (1,-1.5) to (1,-.7);
\draw[color=blue] (-1,1.5) to (-1,-1.5);
  \draw[color=red] (0,1.5) to (0, -1.5);
  \draw[color=blue] (1,1.5) to (1,.7);
%  \draw (-1.5,.7) rectangle (1.5,.7);
      \node[circle,fill,draw,inner sep=0mm,minimum size=1mm,color=blue] at (1,.7) {};
    \end{tikzpicture}
  \end{array}
      +
          \begin{array}{c}
            \begin{tikzpicture}[scale=0.4]
                          \draw[dashed] (-1.5,1.5) to (2,1.5);
      \draw[dashed] (-1.5,-1.5) to (2,-1.5);
%   \draw[color=red] (-1,-1.2) to (-1,-.7);
%   \draw[color=blue] (0,-1.5) to (0,-.7);
% \draw[color=red] (1,-1.5) to (1,-.7);
\draw[color=blue] (-1,1.5) to (-1,-1.5);
  \draw[color=red] (0,1.5) to (0, -1.5);
  \draw[color=blue] (1,1.5) to (1,.7);
  \draw[fill=white] (-1.5,-.7) rectangle (1.5,.7);
  \node at (0,0) {$P$};
%      \node[circle,fill,draw,inner sep=0mm,minimum size=1mm,color=blue] at (1,.7) {};
    \end{tikzpicture}
  \end{array}.
\end{equation}

Finally, one also has the 3-color relations, which are shaded so that
the 1-tensor is sent to the 1-tensor. For example, here is the $A_3$
Zamolodchikov relation (again, note that we use the opposite
convention from \cite{EW} for the shadings).
\begin{equation*}
\begin{array}{l}
\begin{tikzpicture}[scale=0.7]

\node[color=blue] (b1) at (0,0) {$3$};
\node[color=red] (b2) at (1,0) {$2$};
\node (b3) at (2,0) {$1$};
\node[color=blue] (b4) at (3,0) {$3$};
\node[color=red] (b5) at (4,0) {$2$};
\node[color=blue] (b6) at (5,0) {$3$};

\node (t1) at (0,7) {$1$};
\node[color=red] (t2) at (1,7) {$2$};
\node (t3) at (2,7) {$1$};
\node[color=blue] (t4) at (3,7) {$3$};
\node[color=red] (t5) at (4,7) {$2$};
\node (t6) at (5,7) {$1$};

\node[circle,draw,fill,red] (c1) at (1,5.5) {};
\node[circle,draw,fill,blue] (c2) at (3,4.5) {};
\node[circle,draw,fill,red] (c3) at (3,2.5) {};
\node[circle,draw,fill,blue] (c4) at (1,1.5) {};

\draw[red] (t2) -- (c1) ;
\draw[red] (c1) -- (c2) ;
\draw[red] (c2) -- (c3) ;
\draw[red] (c3) -- (c4) ;
\draw[red] (c4) -- (b2) ;
\draw[red] (t5) to [out=-90,in=30] (c2);
\draw[red] (c3) to [out=-30,in=90] (b5);
\draw[red] (c1) to [out=-150,in=150] (c4);

\draw[blue] (t4) to [out=-90,in=90] (c2);
\draw[blue] (c2) to [out=-30,in=90] (b6);
\draw[blue] (c2) to [out=-150,in=90] (c4);
\draw[blue] (c4) to [out=-150,in=90] (b1);
\draw[blue] (c4) to [out=-30,in=90] (b4);

\draw (t3) to [out=-90,in=30] (c1) ;
\draw (t6) to [out=-90,in=30] (c3) ;
\draw (t1) to [out=-90,in=150] (c1);
\draw (c1) to [out=-90,in=150] (c3);
\draw (c3) to [out=-90,in=90] (b3) ;

\draw[dashed] (-.5,6.65) to (5.5,6.65);
\draw[dashed] (-.5,.35) to (5.5,.35);
\end{tikzpicture}
\end{array} = 
\begin{array}{c}
\begin{tikzpicture}[scale=0.7]

\node[color=blue] (b1) at (0,0) {$3$};
\node[color=red] (b2) at (1,0) {$2$};
\node (b3) at (2,0) {$1$};
\node[color=blue] (b4) at (3,0) {$3$};
\node[color=red] (b5) at (4,0) {$2$};
\node[color=blue] (b6) at (5,0) {$3$};

\node (t1) at (0,7) {$1$};
\node[color=red] (t2) at (1,7) {$2$};
\node (t3) at (2,7) {$1$};
\node[color=blue] (t4) at (3,7) {$3$};
\node[color=red] (t5) at (4,7) {$2$};
\node (t6) at (5,7) {$1$};

\node[circle,draw,fill,red] (c1) at (4,5.5) {};
\node[circle,draw,fill,blue] (c2) at (2,4.5) {};
\node[circle,draw,fill,red] (c3) at (2,2.5) {};
\node[circle,draw,fill,blue] (c4) at (4,1.5) {};

\draw[red] (t5) -- (c1) ;
\draw[red] (c1) -- (c2) ;
\draw[red] (c2) -- (c3) ;
\draw[red] (c3) -- (c4) ;
\draw[red] (c4) -- (b5) ;
\draw[red] (t2) to [out=270,in=150] (c2);
\draw[red] (c3) to [out=-150,in=90] (b2);
\draw[red] (c1) to [out=-30,in=30] (c4);

\draw[blue] (t4) to [out=-90,in=90] (c2);
\draw[blue] (c2) to [out=-30,in=90] (c4);
\draw[blue] (c2) to [out=-150,in=90] (b1);
\draw[blue] (c4) to [out=-150,in=90] (b4);
\draw[blue] (c4) to [out=-30,in=90] (b6);

\draw (t3) to [out=-90,in=150] (c1) ;
\draw (t6) to [out=-90,in=30] (c1) ;
\draw (t1) to [out=-90,in=150] (c3);
\draw (c1) to [out=-90,in=30] (c3);
\draw (c3) to [out=-90,in=90] (b3) ;

\draw[dashed] (-.5,6.65) to (5.5,6.65);
\draw[dashed] (-.5,.35) to (5.5,.35);
\end{tikzpicture}
\end{array}
\end{equation*}

\subsection{The definition of the functor $\Lambda$}

We have two versions of $E_{s,t}$: the version $E^{\reds}_{s,t}$ which uses the $\reds$-shaded $2m$-valent vertex, and the version $E^{\bluet}_{s,t}$ which uses the $\bluet$-shaded
$2m$-valent vertex (the source of the morphism $E_{s,t}$ has not changed, only the shading of the vertex). Define $\Lambda(E^{\reds}_{s,t})$ exactly as in \eqref{fdef}, using the
polynomial $\pi_{\reds,\bluet}$. Define $\Lambda(E^{\bluet}_{s,t})$ in the analogous way using the polynomial $\pi_{\bluet,\reds}$.

Note that the definition is symmetric under swapping $\reds$ and $\bluet$, so we only need check one version of each relation and the other follows by symmetry.

\subsection{Checking the relations}

The proof of Lemma \ref{lem:cyclicity} found in \S\ref{subsec:cyclicityof2m} goes through with only one major change. The morphisms $E^{\reds}_{s,t}$ and $E^{\reds}_{t,s}$ being
compared have the same shading, and lead to the same polynomial $\pi_{s,t}$, so every instance of $\pi_{t,s}$ in that proof should be replaced with $\pi_{s,t}$. The proof of Lemma
\ref{lem:cyclicity} splits the problem into two cases. The second case uses the fact that $s(\pi_{s,t}) = - \pi_{s,t}$, which we proved for the unbalanced case in Lemma
\ref{lem:antiinvtiffevenbalanced}. The first case uses the equality\footnote{The proof found in \S\ref{subsec:cyclicityof2m} appears to use the equality $\pi_{s,t} = \pi_{t,s}$, an equality which holds in the balanced case and fails in the unbalanced case. However, this is a {\color{red} red}
herring.  When using shadings appropriately, both sides should be the same polynomial $\pi_{s,t}$. } $\pi_{s,t} = \pi_{s,t}$. 

The proof of Lemma \ref{lem:JW} goes through verbatim.

The proof of Lemma \ref{lem:2assoc} also goes through verbatim, thanks to \eqref{mosttermsPshaded}. 

The analysis of the 3-color relations goes through verbatim.

\def\cprime{$'$} \def\cprime{$'$} \def\cprime{$'$}

% \bibliographystyle{myalpha}
% \bibliography{gen}

\end{document}